\documentclass{amsart}
\pdfoutput=1

\usepackage{amssymb}
\usepackage{amsthm}
\usepackage{amsfonts}
\usepackage{amsmath}
\usepackage{appendix}
\usepackage{pmboxdraw}
\usepackage{verbatim} 
\usepackage{graphicx}
\usepackage{color}
\usepackage[colorlinks=true, citecolor=blue, filecolor=black, linkcolor=black, urlcolor=black]{hyperref}
\usepackage{cite}
\usepackage[normalem]{ulem}
\usepackage{subcaption}
\usepackage{bbm}
\usepackage{bm}
\usepackage{mathtools}
\usepackage{enumerate}
\usepackage{enumitem}
\usepackage{todonotes}
\usepackage{kantlipsum}
\usepackage{comment}
\usepackage{booktabs}
\usepackage{multirow}
\usepackage{array}
\usepackage{nicematrix}
\usepackage{float}
\allowdisplaybreaks
\usepackage{colortbl}
\usetikzlibrary{decorations, arrows.meta, patterns}
\usetikzlibrary{decorations.pathreplacing}

\usepackage{tikz}
\usetikzlibrary{positioning}
\newcolumntype{C}[1]{>{\centering\arraybackslash}m{#1}}

\newcommand{\RN}[1]{%
	\textup{\uppercase\expandafter{\romannumeral#1}}%
}

\def\bfs{\boldsymbol}

\def\wh{\widehat}

\def\Im{ \mathrm{Im}}

\DeclareMathOperator{\Res}{Res}

\def\C{\mathbb{C}}

\def\R{\mathbb{R}}

\newcommand{\supp}{\operatorname{supp}}
\newcommand{\Li}{\operatorname{Li}}
\newcommand{\pv}{\operatorname{p.v.}}

\usepackage{tikz-cd}
\newcommand{\qbinom}[2]{{\begin{bmatrix}#1\\#2\end{bmatrix}}_q}
\newcommand{\floor}[1]{\left\lfloor#1\right\rfloor}

\newcommand{\Cr}{\mathrm{cr}}

\usepackage{tcolorbox}
\tcbuselibrary{most}

\theoremstyle{plain}
\newtheorem{thm}{Theorem}[section]

\newtheorem{lem}[thm]{Lemma}

\newtheorem{prop}[thm]{Proposition}
\newtheorem{rhp}{RHP}[section]

\theoremstyle{remark}
\newtheorem{rem}{Remark}

\newtcolorbox{defn}[1]{breakable,
colbacktitle=gray!50!white, fonttitle=\bfseries, coltitle=black, title=Definition: {#1}}

\newcommand{\abs}[1]{\lvert#1\rvert}
\numberwithin{equation}{section}

\begin{document}

\title[$q$-deformation of the Marchenko-Pastur law]{$q$-deformation of the Marchenko-Pastur law}
\author{Sung-Soo Byun}
\address{Department of Mathematical Sciences and Research Institute of Mathematics, Seoul National University, Seoul 151-747, Republic of Korea}
\email{sungsoobyun@snu.ac.kr}

\author{Yeong-Gwang Jung}
\address{Department of Mathematical Sciences, Seoul National University, Seoul 151-747, Republic of Korea}
\email{wollow21@snu.ac.kr} 

\author{Guido Mazzuca}
\address{Department of mathematics, Tulane University, 6823 St Charles Ave, New Orleans, LA 70118, USA} 
\email{gmazzuca@tulane.edu}

\begin{abstract}
We study a $q$-deformed random unitary ensemble associated with the little-$q$ Laguerre weight, which provides a discrete analogue of the classical Laguerre unitary ensemble. In the double scaling regime $q=e^{-\lambda/N}$, where $N$ is the system size and $\lambda \ge 0$, we derive the limiting spectral distribution as $N\to \infty$, which yields a $q$-deformation of the Marchenko-Pastur law. The limiting density undergoes a phase transition at an explicitly determined critical value $\lambda_c$: for $\lambda<\lambda_c$, the support consists of a single band region, whereas for $\lambda>\lambda_c$ an additional saturated region emerges adjacent to the band region. 
Our derivation of the limiting distribution is based on three complementary approaches: the method of moments, the analysis of a constrained equilibrium problem, and the asymptotic zero distribution of orthogonal polynomials. As a consequence, we establish the convergence of the empirical measure as well as a large deviation principle. In addition, we derive closed-form expressions for the spectral moments using the combinatorial structure of orthogonal polynomials, and obtain large-$N$ expansions for these moments.
\end{abstract}

\maketitle


\section{Introduction and main results}

The Marchenko-Pastur law is one of the fundamental distributions of random matrix theory, governing the universal limiting behaviour of eigenvalues of Wishart (or sample covariance) matrices \cite{MP67}. It forms a one-parameter family indexed by $c \ge 0$, which encodes the aspect ratio of the matrix, with density given by
\begin{equation} \label{def of MP law}
\rho _{\mathrm {MP}}^{(c)}(x) := \frac{1}{2\pi} \frac{ \sqrt{ (x_+-x)(x-x_-) } }{ x } \cdot \mathbbm{1}_{ (x_-,x_+) }(x), \qquad x_\pm = ( \sqrt{c+1} \pm 1 )^2. 
\end{equation}
Beyond its classical formulation, a growing body of work across integrable probability, representation theory, and mathematical physics has highlighted the remarkable role of \emph{$q$-deformations} as a unifying mechanism interpolating between discrete and continuous structures; see e.g. \cite{Jo01,BGG17,Ol21}. Here, the parameter $q \in (0,1)$ is referred to as the quantisation parameter, with the classical regime recovered in the limit $q \to 1$. Such $q$-deformations often reveal hidden algebraic symmetries, connect probabilistic models with exactly solvable frameworks, and generate families of limiting objects that retain rich integrability while deviating substantially from their classical counterparts. From this perspective, understanding how fundamental spectral laws in random matrix theory behave under $q$-deformation is a natural and compelling problem. Nevertheless, despite significant recent developments on $q$-analogues of classical ensembles 
\cite{BF25a,FLSY23,Fo21,MPS20,BFO24,BJO25,HMO25,BO24,LSYF25,CM25,LWYZ25,FL20}, a $q$-deformation of the Marchenko-Pastur law has remained largely unexplored.

In this work, we aim to contribute this direction by providing a unified and comprehensive understanding of the $q$-deformed Marchenko-Pastur law. In our main result Theorem~\ref{Thm:limiting density in growth regime} below, we analyse a natural $q$-deformed Laguerre unitary ensemble (LUE) and identify its limiting spectral distribution through three complementary approaches: 
\begin{itemize}
    \item[(A)] \textbf{asymptotic behaviour of the spectral moments};
    \smallskip 
    \item[(B)] \textbf{equilibrium measure of logarithmic energy with upper  constraint};  
    \smallskip 
    \item[(C)] \textbf{limiting empirical zeros of orthogonal polynomials}. 
\end{itemize}  
Each perspective captures a different structural aspect of the ensemble: moment methods emphasise combinatorial interpretations of spectral moments, potential theoretic techniques reveal the underlying variational principles together with the framework of the large deviation theory, and the asymptotics of orthogonal polynomial zeros offer an analytic route to the limiting spectral distribution. 

Before introducing our model and methods in detail, we first present the explicit form of the $q$-deformed Marchenko-Pastur law, which serves as the central object of this work. In addition to the parameter $c \ge 0$, we introduce a parameter $\lambda \ge 0$ that encodes the quantisation via $q = e^{-\lambda/N}$, where $N$ denotes the system size. 

\begin{defn}{$q$-deformed Marchenko-Pastur law ($q=e^{-\lambda/N}$)} 
Let $c$ and $\lambda$ be given non-negative real numbers. Let $\lambda_c$ be a unique positive number satisfying
    \begin{equation} \label{eqn:def of lambdac}
        e^{-\lambda}( 1+e^{-c\lambda} )=1.
    \end{equation} 
Write $\mathsf{s}=e^{-\lambda}$ and define 
\begin{equation} \label{def of endpts a and b}
a=\Big(\sqrt{ \mathsf{s}(1-\mathsf{s}^{c+1}) } -\sqrt{ \mathsf{s}^{c+1}(1-\mathsf{s}) } \Big)^2, \qquad b=\Big(\sqrt{ \mathsf{s}(1-\mathsf{s}^{c+1}) } +\sqrt{ \mathsf{s}^{c+1}(1-\mathsf{s}) } \Big)^2. 
\end{equation}  
Then the \textit{$q$-deformed Marchenko-Pastur law} is defined by
\begin{equation}\label{eqn:def of limiting density v0}
            \rho^{(c)}(x) :=\frac{2}{\pi\lambda x}\arctan\sqrt{  \frac{ x- \mathsf{s} -\mathsf{s}^{c+1} +2\mathsf{s}  \sqrt{\mathsf{s}^{c}(1-x)}  }{ \mathsf{s} +\mathsf{s}^{c+1} +2 \mathsf{s}  \sqrt{\mathsf{s}^{c}(1-x)} -x  }    }\mathbbm{1}_{(a,b)}(x) +\begin{cases}
                0& \textup{if } \lambda \leq \lambda_c , \smallskip \\
                \displaystyle\frac{1}{\lambda x}\mathbbm{1}_{(b,1)}(x) & \textup{if } \lambda > \lambda_c. 
            \end{cases} 
    \end{equation}
Equivalently,  
\begin{equation} \label{eqn:def of limiting density}
   \rho^{(c)}(x) = 
  \begin{cases}
 \displaystyle  \frac{2}{\pi\lambda x}\arctan \bigg(  \frac{ \sqrt{(x-a)(b-x)}  }{ ( \sqrt{1-x}+\sqrt{1-a} )(\sqrt{1-x}+\sqrt{1-b}) } \bigg)  \mathbbm{1}_{(a,b)}(x)&\textup{if } \lambda\le \lambda_c, 
 \medskip 
 \\
 \displaystyle \frac{2}{\pi\lambda x}\arctan \bigg(  \sqrt{ \frac{x-a }{b-x} } \frac{\sqrt{1-x}+\sqrt{1-b} }{ \sqrt{1-x}+\sqrt{1-a} }  \bigg)  \mathbbm{1}_{(a,b)}(x) + \frac{1}{\lambda x}\mathbbm{1}_{(b,1)}(x)   &\textup{if } \lambda>\lambda_c. 
  \end{cases}  
\end{equation}   
\end{defn}

It is straightforward to check that, after the rescaling  $   \widehat{\rho}^{(c)}(x):= \lambda \rho^{(c)}(\lambda x)$, the classical Marchenko-Pastur law \eqref{def of MP law} is recovered in the 
limit \(q \to 1\), equivalently \(\lambda \to 0\):
 \begin{equation} \label{qMP to MP}
        \lim_{\lambda\rightarrow0}\widehat{\rho}^{(c)}(x)=\rho _{\mathrm {MP}}^{(c)}(x).
\end{equation}
In addition, one observes a non-trivial phase transition at a critical value $\lambda_c$; see Figure~\ref{Fig_limiting density} for the graphs of $\rho^{(c)}$.  

\subsection{LUE and Marchenko-Pastur law revisited}

We begin by revisiting the classical developments that lead to the 
Marchenko-Pastur law. By definition, the LUE or the Wishart matrix is given by $W = X^{*}X$ \cite[Chapter 3]{Fo10}, where $X$ is an $M \times N$ ($M \geq N$) random matrix with independent standard complex Gaussian entries \cite{BF25}. 
The joint probability distribution function of eigenvalues $\mathbf{x}=\{x_j\}_{j=1}^N$ of $W$ is given by
\begin{equation}\label{def of joint eigenvalue of unitary ensemble}
    d\mathbf{P}(\mathbf{x})= \frac{1}{Z_{N}} \prod_{1\leq j<k\leq N}(x_{j}-x_{k})^{2}\prod_{j=1}^{N}w(x_j) \,dx_{j}. 
\end{equation}
Here, $Z_N$ is a normalisation constant and the weight function $w \equiv w^{ \rm (L) }$ is given by 
\begin{equation}\label{def of Laguerre weight}
    w^{(\mathrm L)}(x):=x^{\alpha}e^{-x}\mathbbm{1}_{[0,\infty)}(x), \qquad \alpha=M-N. 
\end{equation} 

When analysing the large-$N$ asymptotic behaviour of the LUE, it is necessary to specify the scaling of the parameter $\alpha$. 
A standard and sufficiently general choice is to assume a linear scaling of the form
\begin{equation}\label{def of alpha scaling}
    \alpha = cN + d, \qquad c \geq 0,\quad d \in \mathbb{N}.
\end{equation} 
Then as $N \to \infty$, the empirical measure of $\{x_j/N\}_{j=1}^N$ converges to the Marchenko-Pastur law \eqref{def of MP law}. 
We outline how the three approaches mentioned above can be adapted to derive the Marchenko-Pastur law~\eqref{def of MP law}.

In the first approach, we consider the spectral moments, a method that goes back to Wigner’s seminal work \cite{Wi55} deriving the semicircle law. (See also \cite{CMOS19,MS13,GGR21,WF14,CDO21,ABGS21a,CCO20,FR21} for more recent development.) With $p$ a non-negative integer, the $p$-th spectral moment of a random matrix is defined by
\begin{equation} \label{def of spectral moments mNp}
    m_{N,p}:= \mathbb{E}\Big[\sum_{j=1}^{N}x_{j}^{p}\Big], 
\end{equation}
where the expectation is taken with respect to the distribution \eqref{def of joint eigenvalue of unitary ensemble}. This definition applies to general random matrix ensembles; for the LUE we add the superscript $(\mathrm L)$ to indicate the corresponding moments. 
In this case, an exact evaluation of $ m_{N,p}^{ \rm (L) } $ is available in the literature; see e.g. \cite[Theorem 2.5]{HSS92} and \cite[Proposition 9.1]{HT03}. (See also \eqref{evaluation of LUE moments} and \eqref{eqn:LUE moment 2} below.) Furthermore, when $\alpha$ is scaled according to \eqref{def of alpha scaling}, it is well known that  
\begin{equation} \label{leading order conv for LUE moment}
    \lim_{ N \to \infty }\frac{1}{N^{p+1}} m_{N,p}^{ \rm (L) } =  \sum_{j=0}^{p}\mathrm{N}(p,j)(1+c)^{j}, \qquad \mathrm{N}(p,j):=\frac{1}{p}\binom{p}{j}\binom{p}{j-1}. 
\end{equation}
Here, $\mathrm{N}(p,j)$ is the Narayana number.
The Marchenko-Pastur distribution is uniquely determined by its moments, 
and we have 
\begin{equation}
    \int_{x_-}^{x_+} x^p\,  \rho _{\mathrm {MP}}^{(c)}(x)\,dx =  \sum_{j=0}^{p}\mathrm{N}(p,j)(1+c)^{j}
\end{equation}
from which the Marchenko-Pastur law \eqref{def of MP law} follows.

The second approach originates from the viewpoint of statistical physics.  
For the normalised eigenvalues $\{x_j/N\}_{j=1}^N$, the Gibbs measure \eqref{def of joint eigenvalue of unitary ensemble} is proportional to $\exp(-\frac{\beta}{2}H_N)$, where $\beta=2$ and 
\begin{equation}
H_N := \sum_{ 1\le j \not = k \le N } \log \frac{1}{|x_j-x_k|} +N \sum_{j=1}^N V(x_j),  \qquad  V(x):= x - c \log x. 
\end{equation}
In the large-$N$ limit, the continuum version of this Hamiltonian corresponds to the energy functional
\begin{equation} \label{def of I_V LUE}
I_V[\mu]:= \iint_{ \R_+^2 } \log \frac{1}{|x-y|}\,\mu(dx)\,\mu(dy) + \int_{ \R_+ } V(x) \,\mu(dx). 
\end{equation}
From the perspective of statistical physics, one naturally expects that the particles arrange themselves so as to minimise their Hamiltonian. In fact, by the standard theory of log-gases and large deviation principles (see e.g. \cite{Jo98,Fo10,Se24,BG97,BF25b}), the limiting global distribution is characterised as a unique minimiser of $I_V$. Then by applying the classical methods in logarithmic potential theory \cite{ST97}, one finds that this minimiser is precisely the Marchenko-Pastur distribution.

The third approach can be regarded as the $\beta=\infty$ extremal counterpart of the second method, but formulated at finite $N$. In the log-gas formulation \eqref{def of joint eigenvalue of unitary ensemble}, it is well known that the global limiting distribution does not depend on the (fixed) inverse temperature $\beta$. The idea is to consider the opposite extreme, namely the freezing low-temperature regime, where the configuration of the minimiser of the discrete Hamiltonian coincides with the zeros of the associated orthogonal polynomial.  
For the LUE, the associated orthogonal polynomials are the generalised Laguerre polynomials $L_n^{(\alpha)}$, defined through the three-term recurrence relation
\begin{equation}\label{eqn:three term of laguerre}
    (n+1)L_{n+1}^{(\alpha)}(x)-(2n+\alpha +1-x)L_{n}^{(\alpha)}(x)+(n+\alpha)L_{n-1}^{(\alpha)}(x)=0,
\end{equation} 
with initial conditions $L_0^{(\alpha)}(x)=1$ and $L_1^{(\alpha)}(x)=1+\alpha - x$.  
From the theory of orthogonal polynomials, the Marchenko-Pastur law \eqref{def of MP law} can then be obtained as the limiting empirical distribution of the zeros of the rescaled Laguerre polynomial $L_N^{(\alpha)}(x/N)$.

\subsection{Main results}

We first introduce the $q$-deformation of the LUE. In the $q$-deformed framework, the basic idea is that the integer $n$ is replaced by the $q$-integer 
\begin{equation}
[n]_q=\frac{1-q^n}{1-q}, \qquad q \in (0,1). 
\end{equation}
Building on this, the other notions such as $q$-binomials or $q$-Pochhammer symbols naturally arise. For instance, the standard notations include 
\begin{equation}
 \qbinom{n}{m}= \frac{[n]_q!}{ [m]_q![n-m]_q! } , \qquad   (x;q)_{n}=\prod_{j=0}^{n-1}(1-xq^{j}),\qquad (x_{1},\cdots,x_{k};q)_{\infty}=\prod_{j=0}^{\infty}  \prod_{l=1}^k (1-x_l q^j).  
\end{equation} 
Using this, we define the little $q$-Laguerre weight (see e.g. \cite[Eq.~(2.52)]{FLSY23}) by
\begin{equation} \label{def of qLUE weight}
    w^{(\mathrm{qL})}(x;q)=x^{\alpha}(qx;q)_{\infty},\qquad \alpha >-1.
\end{equation}
Notice here that by \cite[Eq.~(2.16)]{BF25a}, we have 
\begin{equation} \label{q1 limit of weight}
\lim_{q \to 1} (1-q)^{-\alpha} w^{(\mathrm{qL})}((1-q)x;q) =  w^{(\mathrm{L})}(x). 
\end{equation}
Next, we consider the system $\bfs{x}=\{x_j\}_{j=1}^N$ supported on the one-sided $q$-lattice $\{1,q,q^2,\dots \}$ having the joint distribution of the form (see e.g. \cite{Ol21,BFO24})
\begin{equation} \label{def of jpdf for qL}
 d\mathcal{P}(\bfs{x}) = \frac{1}{\mathcal{Z}_N}\prod_{1\le j<k \le N}(x_j-x_k)^2 \prod_{j=1}^N \omega^{\rm (qL)}(x_j), 
\end{equation}
where $\mathcal{Z}_N$ is the normalisation constant. We call the ensemble \eqref{def of jpdf for qL} as the \emph{(little) $q$-LUE}. By \eqref{q1 limit of weight}, the renormalised system $x_j \mapsto (1-q)x_j$ recovers the classical LUE as $ q \to 1$.

Following the general notion \eqref{def of spectral moments mNp}, we define  
\begin{equation} \label{def of qLUE moments}
m_{N,p}^{ \rm (qL) } := \mathbb{E} \Big[ \sum_{j=1}^N x_j^p \Big]
\end{equation}
for the spectral moments of the $q$-LUE. Here, the expectation is taken with respect to \eqref{def of jpdf for qL}.
Unlike the continuum case $q=1$, where the spectral moments can be derived using well-known properties of Laguerre polynomials together with complex-analytic techniques, the $q$-deformed setting lacks comparable analytic tools for obtaining closed-form expressions (although the discrete loop equation approach of \cite{BGG17} may be adaptable to this setting; cf.~\cite{WF14} for the continuum case).
In recent years, significant progress has been made for other $q$-deformed 
ensembles, with algebraic combinatorial methods -- such as symmetric function theory \cite{BF25a,FLSY23,Fo21,MPS20} and the combinatorics of orthogonal polynomials \cite{BFO24,BJO25} -- yielding explicit summation formulas. In particular, the Flajolet--Viennot theory based on techniques in \cite{BFO24,BJO25} produces positive formulas well suited for large-$N$ asymptotics.  
In Theorem~\ref{Thm:Spectralmoment} below, we establish analogous results for the $q$-LUE, where the underlying combinatorial structure turns out to be richer than in \cite{BFO24,BJO25}.

Our first results establish a closed formula for the spectral moments and their large-$N$ asymptotic behaviour. For the latter, we recall that for $a,b>0$, the regularised incomplete beta function is defined by
\begin{equation} \label{def of beta ftn}
I_x(a,b)=  \frac{ \Gamma(a+b) }{ \Gamma(a)\Gamma(b) } \int_0^x t^{a-1} (1-t)^{b-1}\,dt, \qquad (0<x<1);
\end{equation}
see \cite[Section~8.17]{NIST}.

\begin{thm}[\textbf{Spectral moments of the $q$-Laguerre unitary ensemble}] \label{Thm:Spectralmoment} Let $ {m}_{N,p}^{\rm (qL)}$ be the spectral moment of the $q$-LUE given by \eqref{def of qLUE moments}. 
\begin{itemize}
    \item[\textup{(A)}] \textup{\textbf{(Closed formula)}} For any nonnegative integers $p,\alpha$ and $N \in \mathbb{N}$, we have  
    \begin{equation}  \label{eqn:m}
  {m}_{N,p}^{\rm (qL)} = (1-q)^{p} [p]_q! \sum_{j=0}^{N-1}\sum_{i=0}^{p} q^{ (p-i)(\alpha-i)+pj }  \qbinom{p}{i} \qbinom{\alpha+j}{i} \qbinom{p-i+j}{j}. 
\end{equation}  
    \item[\textup{(B)}] \textup{\textbf{(Large-$N$ expansion)}}  Let $\alpha$ be scaled according to \eqref{def of alpha scaling}. Let 
    \begin{equation}\label{eqn:q scale}
    q=e^{-\frac{\lambda}{N}},\qquad \lambda\geq0.
\end{equation} 
Then we have 
\begin{equation} \label{def of leading order moment}
\lim_{N \to \infty} \frac{ 1 }{N}m_{N,p}^{(\mathrm{qL})} = \mathcal{M}_{p}^{ \rm (qL) } :=  \frac{1}{\lambda p}
\sum_{m=0}^p  \binom{p}{m} \mathsf{s}^{c m}  (1-\mathsf{s}^c)^{p-m}\,
I_{1-\mathsf{s}}(m+1,p),
\end{equation} 
where we recall that $\mathsf{s}=e^{-\lambda}$. Furthermore, for $c=0$, we have 
\begin{equation} \label{qLUE moment expansion c0}
m_{N,p}^{(\mathrm{qL})}|_{c=0}= N \mathcal{M}_{p}^{ \rm (qL) }|_{c=0} + \mathcal{M}_{p,1/2}^{ \rm (qL) }(d) +O\Big(\frac{1}{N}\Big), 
\end{equation}
where  
\begin{equation} \label{def of subleading term qLUE moment}
  \mathcal{M}_{p,1/2}^{ \rm (qL) }(d) := \frac12 I_{1- \mathsf{s} }(p+1,p)+\frac{d}{2}  \binom{2p}{p} \Big(  \mathsf{s}^p (1-\mathsf{s})^{p} - \delta_{p,0} \Big).  
\end{equation}
Here, $\delta$ is the Kronecker delta function.  
\end{itemize}
\end{thm}

\begin{figure}[t]
	\begin{subfigure}{0.45\textwidth}
	\begin{center}	
		\includegraphics[width=\textwidth]{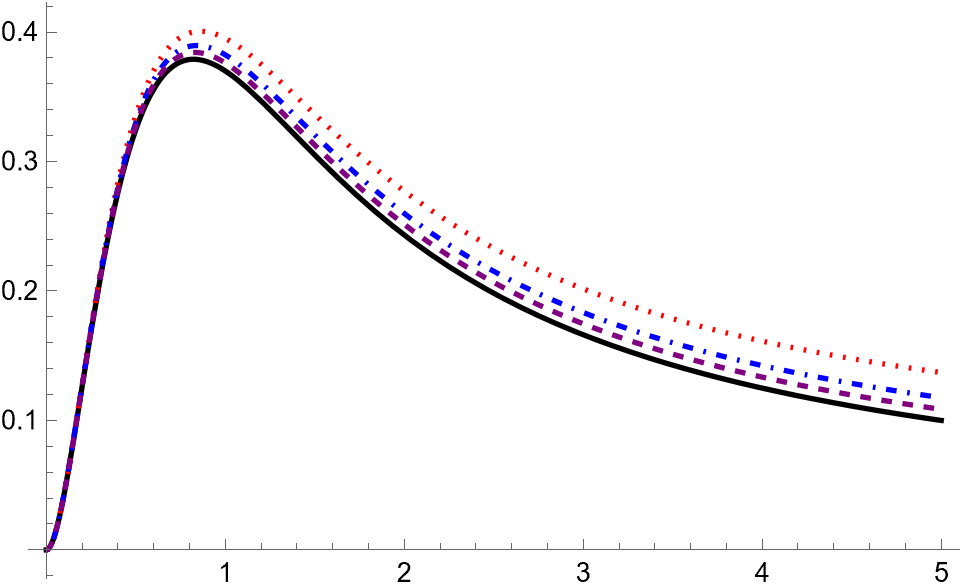}
	\end{center}
	\subcaption{ $\alpha=N$ $(c=1, d=0)$}
\end{subfigure}	 \qquad 
	\begin{subfigure}{0.45\textwidth}
	\begin{center}	
		\includegraphics[width=\textwidth]{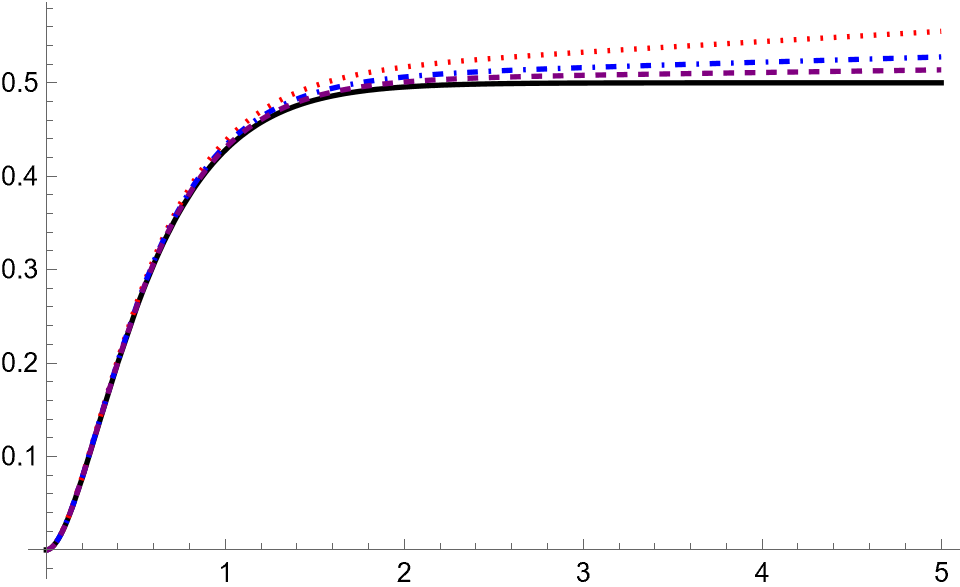}
	\end{center}
	\subcaption{ $\alpha=1$ $(c=0, d=1)$}
\end{subfigure}		
	\caption{(A): The graph of $\lambda \mapsto \mathcal{M}_{p}^{\rm (qL)}$ (black solid) together with $\lambda \mapsto \frac{1}{N} m_{N,p}^{\rm (qL)}$ for $N=15$ (red, dotted), $N=30$ (blue, dot--dashed), and $N=60$ (purple, dashed). (B): The corresponding plot of $\lambda \mapsto m_{N,p}^{\rm (qL)} - N \mathcal{M}_{p}^{\rm (qL)}$ and its comparison with $\lambda \mapsto \mathcal{M}_{p,1/2}^{\rm (qL)}(d)$.
Here $p=2$. } \label{Fig_moments asymptotic}
\end{figure}

See Figure~\ref{Fig_moments asymptotic} for the numerics of Theorem~\ref{Thm:Spectralmoment} (B).
Theorem~\ref{Thm:Spectralmoment} will be proved in Section~\ref{Section_moments combinatorics} using the combinatorics of orthogonal polynomials together with a bijection to a certain matching problem; see Figure~\ref{fig:comparison_models}. We note that \eqref{eqn:m} provides one of the very few instances in which the spectral moments of $q$-deformed models can be evaluated explicitly in closed form.

Note that when $n \ge m \ge 0$ are integers, we have 
\begin{equation} \label{def of beta ftn for integers}
I_x(m,n-m+1)=\sum_{j=m}^n \binom{n}{j} x^j (1-x)^{n-j}. 
\end{equation}
By \eqref{def of beta ftn for integers}, the first few values of $\mathcal{M}_{p}^{ \rm (qL) }$ is given as follows: 
\begin{align} \label{eq:first_moment}
\mathcal{M}^{ \rm (qL) }_{0}=1, \qquad \mathcal{M}^{ \rm (qL) }_{1}=\frac{1}{\lambda}  (1-\mathsf{s})(1-\mathsf{s}^{c+1}), \qquad \mathcal{M}^{ \rm (qL) }_{2}= \frac{1}{2\lambda} (1-\mathsf{s})(1-\mathsf{s}^{c+1})(1+\mathsf{s}+\mathsf{s}^{c+1}-3\mathsf{s}^{c+2}).
\end{align}
In particular, the first moment $\mathcal{M}^{\rm (qL)}_{1}$ plays an important role in the  analysis of the equilibrium measure problem discussed later.

\bigskip

Next, we discuss the macroscopic limit of the density. For this, we recall that for a parameter $\mathsf{a}\in (0,1/q)$, the \emph{little $q$-Laguerre} (also known as \emph{Wall}) polynomial $p_{n}(x;\mathsf{a}|q)\equiv p_{n}(x)$ \cite[Section 14.20]{KLS10} is defined by the three-term recurrence relation 
\begin{equation}\label{eqn:three term of qlaguerre}
    q^{n}(1-\mathsf{a}q^{n+1})p_{n+1}(x)+(x-q^{n}(1-\mathsf{a}q^{n+1})-\mathsf{a}q^{n}(1-q^{n}))p_{n}(x)+\mathsf{a}q^{n}(1-q^{n})p_{n-1}(x)=0
\end{equation}
with initial conditions $p_0(x)=1$ and $p_1(x)=-x/(1-\mathsf{a} q)+1$. 
Throughout this paper, we put 
\begin{equation}
   \mathsf{a}=q^{\alpha},\qquad \alpha>-1, \qquad \alpha = cN + d. 
\end{equation} 
By taking the continuum limit, one can recover the classical Laguerre polynomial. More precisely, we have 
\begin{equation}\label{eqn: continuum limit to Laguerre polynomial}
    \lim_{q\rightarrow1}p_{n}((1-q)x;q^{\alpha}\vert q)=\frac{n! \,\Gamma(\alpha+1)}{ \Gamma(n+\alpha+1) }\,L_{n}^{(\alpha)}(x). 
\end{equation} 
The three-term recurrence relation \eqref{eqn:three term of laguerre} is also  recovered from \eqref{eqn:three term of qlaguerre} via this scaling limit. 

To present our main results, we introduce the associated finite-$N$ observables. 
The first is the averaged density (or one-point function) of the $q$-LUE, which, by the general theory of determinantal point processes, is expressed in terms of the $q$-Laguerre polynomials as
\begin{equation}
    \rho_{N}(x):= \sum_{j=0}^{N-1} \frac{ p_{j}(x;q^{\alpha}\,|\,q)^2 }{ h_j }  w^{(\mathrm{qL})}(x;q), \qquad h_j:=(1-q)q^{(\alpha+1)n}\frac{(q;q)_{n}}{(q^{\alpha+1};q)_{n}}\frac{(q;q)_{\infty}}{(q^{\alpha+1};q)_{\infty}}. 
\end{equation}  
The second observable is the empirical measure of the $q$-LUE  \eqref{def of jpdf for qL}: 
 \begin{equation}
  \mu_N(dx) := \frac{1}{N} \sum_{j=1}^N \delta_{x_j} .
 \end{equation}
The third is the empirical zero distribution of the $q$-Laguerre polynomials, defined by
\begin{equation} \label{def of ESD of zeros}
    \nu_{n,N}(dx):=\frac{1}{n}\sum_{j=1}^{n} \delta_{x_{n,N}^{(j)}},
\end{equation}
where $\delta_{x_{n,N}^{(1)}},\cdots,\delta_{x_{n,N}^{(n)}}$ are the zeros of $p_{n}(x;q^\alpha|q)$.  

In addition, in order to formulate the large deviation principle, we introduce the set 
\begin{equation}
\label{eq:prob_space}
    \mathcal{P}_{\lambda}(0,1):=\bigg\{ \mu(dx)\;\vert\; \supp(\mu(dx))\subseteq[0,1],\; \frac{\delta\mu(dx)}{\delta x}\leq \frac{1}{\lambda x}\bigg\}
\end{equation}
 of probability measures, where $\supp(\mu(dx))$ is the support of the measure and $\frac{\delta\mu(dx)}{\delta x}$ is the Radon-Nikodym derivative of the measure. Here, the upper constraint condition is the essential structural element inherited from the underlying discrete model; cf. Remark~\ref{Rem_upper constraint}. 
We also introduce the energy functional
 \begin{equation}
    \label{eq:functional_rhp}
        I[\mu] : = - \int_{0}^{1}\int_{0}^{1} \log(|x-y|)\, \mu(dx) \, \mu(dy)  + \int_0^1 \Big( \frac{1}{\lambda}\Li_2(x)-c \log x \Big)\, \mu(dx)\,,
    \end{equation}
where $\Li_2(x)$ is the dilogarithm \cite[Chapter~25.12]{NIST}, defined for 
$x<1$ by
    \begin{equation}
    \label{eq:Li2_def}
        \Li_2(x) = - \int_0^x \frac{\log(1-t)}{t}\, dt.
    \end{equation}
Here, the appearance of the dilogarithm arises from the asymptotic behaviour of the weight function $w^{(\mathrm{qL})}$; see Lemma~\ref{Lem_asymp of weight}. We are now ready to state the main result of the paper. 

\begin{thm}[\textbf{$q$-deformed Marchenko-Pastur law via three approaches}]\label{Thm:limiting density in growth regime}
Let $c \ge 0$ and $d \in \mathbb{N}$ be fixed, and let $q$ be scaled according to \eqref{eqn:q scale}. Let $\rho^{(c)}$ be the $q$-deformed Marchenko-Pastur law \eqref{eqn:def of limiting density}. 
\begin{enumerate}[label = \textup{(\Alph*)}]
    \item \textup{\textbf{(Probability measure determined by moments)}} 
    As $N \to \infty$, in the sense of integration against continuous test functions $f \in C([0,1])$, we have 
\begin{equation}
    \frac{1}{N}\rho_{N}( x ) \to  \rho^{ (c) }( x ).  
\end{equation} 
    \item \textup{\textbf{(Equilibrium measure of the $q$-logarithmic energy)}} \label{thm:point_b}
    The sequence of measures $ \{ \mu_N(dx) \}_{N\in\mathbb{N}}$ satisfies a large deviation principle in the space of probability measures $\mathcal{P}_{\lambda}(0,1)$ with scale $N^2$ and a good convex rate function $J[\mu] := I[\mu] - \inf_{\mu\in\mathcal{P}_{\lambda}(0,1)} I[\mu]$, where $I[\mu]$ is given by \eqref{eq:functional_rhp}. 
    Furthermore, 
    \begin{equation}
    I[\rho^{(c)}(x)\,dx] = \inf_{\mu\in\mathcal{P}_{\lambda}(0,1)} I[\mu]
    \end{equation}
    and consequently, as $N \to \infty$, we have 
    \begin{equation} 
    \mu_N(dx)   \to \rho^{(c)}(x)\,dx \qquad \text{almost surely}\,.
    \end{equation} 
\item \textup{\textbf{(Limiting zero distribution of little $q$-Laguerre polynomial)}} As $N \to \infty$ with $n/N\rightarrow1$, in the sense of integration against continuous test functions $f \in C([0,1])$, we have  
\begin{equation}
     \nu_{n,N}(dx) \to  \rho^{(c)}(x)\,dx. 
\end{equation} 
\end{enumerate} 
\end{thm}

See Figure~\ref{Fig_limiting density} for the graphs of $q$-deformed Marchenko-Pastur law.  We stress that the three approaches in Theorem~\ref{Thm:limiting density in growth regime} are fundamentally different, yet complementary. Most notably, in order to establish the large deviation principle in Theorem~\ref{Thm:limiting density in growth regime} (B), the characterisation of the minimising measure requires precise information on the first moment of $\rho^{(c)}$, which is obtained from Theorem~\ref{Thm:limiting density in growth regime} (A).

\begin{figure}[t]   \centering
    \begin{tabular}{ | c | c | c  | } 
        \hline
        \cellcolor{gray!25}  & \cellcolor{gray!10} \parbox[c]{5.0cm}{\centering \vspace{.5\baselineskip} $ \lambda < \lambda_c $ \\ \emph{band region}   \vspace{.5\baselineskip} }  & \cellcolor{gray!10} \parbox[c]{5.0cm}{\centering \vspace{.5\baselineskip} $ \lambda > \lambda_c $ \\ \emph{band \& saturated region} \vspace{.5\baselineskip}}     
        \\
        \hline
        \cellcolor{gray!10} \parbox[c]{1.5cm}{\centering  $c=0$ \\ \emph{hard edge} \vspace{.5\baselineskip} }  
        &   \rule{0pt}{3cm} \raisebox{-0.4\height}{\includegraphics[width=0.4\linewidth]{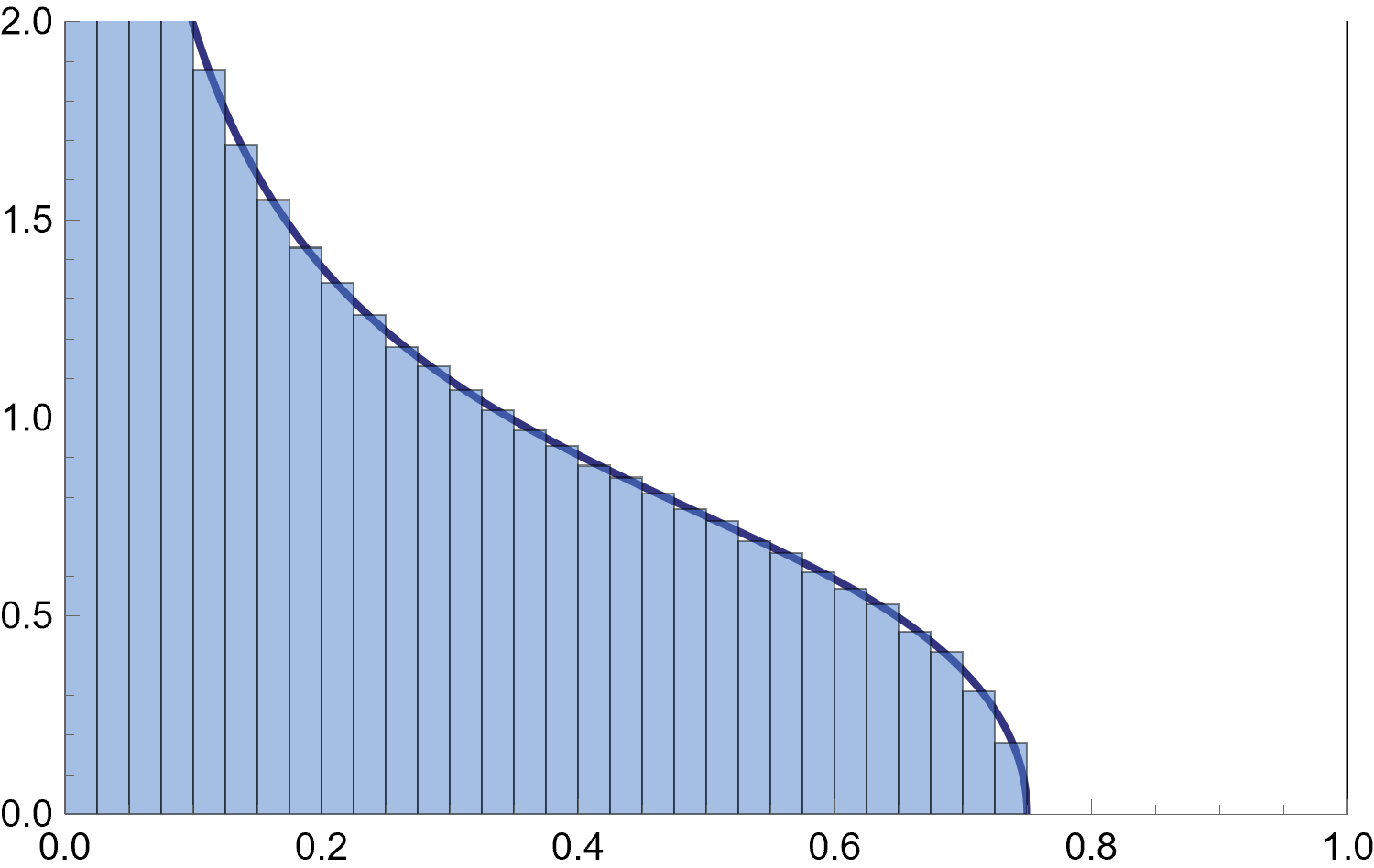}  }   
        &     \rule{0pt}{3cm} \raisebox{-0.4\height}{\includegraphics[width=0.4\linewidth]{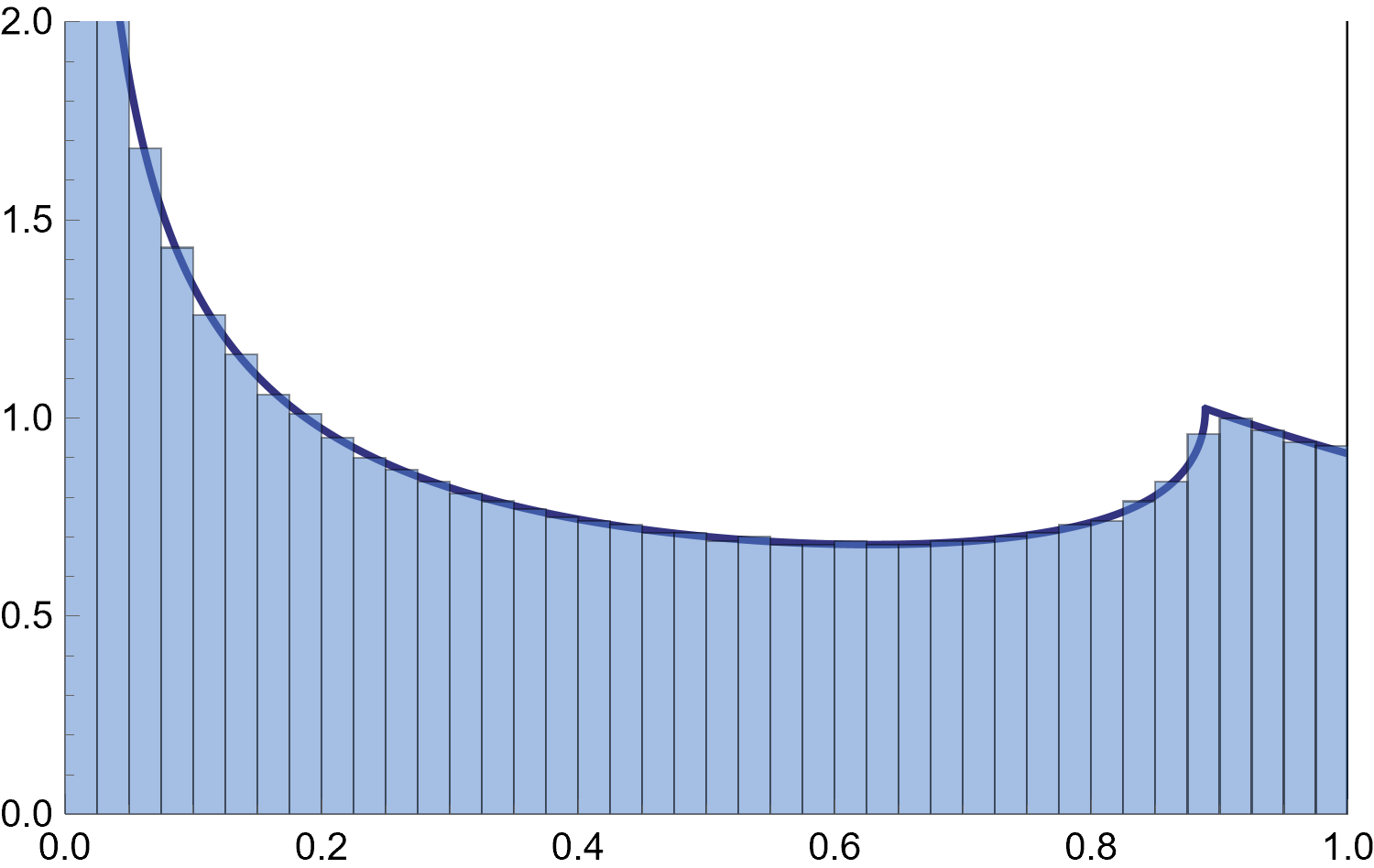} }  
         \\
        \hline
        \cellcolor{gray!10}  \parbox[c]{1.5cm}{\centering \vspace{-3em} $c>0$ \\ \emph{soft edge}  }    &    \rule{0pt}{3cm} \raisebox{-0.4\height}{\includegraphics[width=0.4\linewidth]{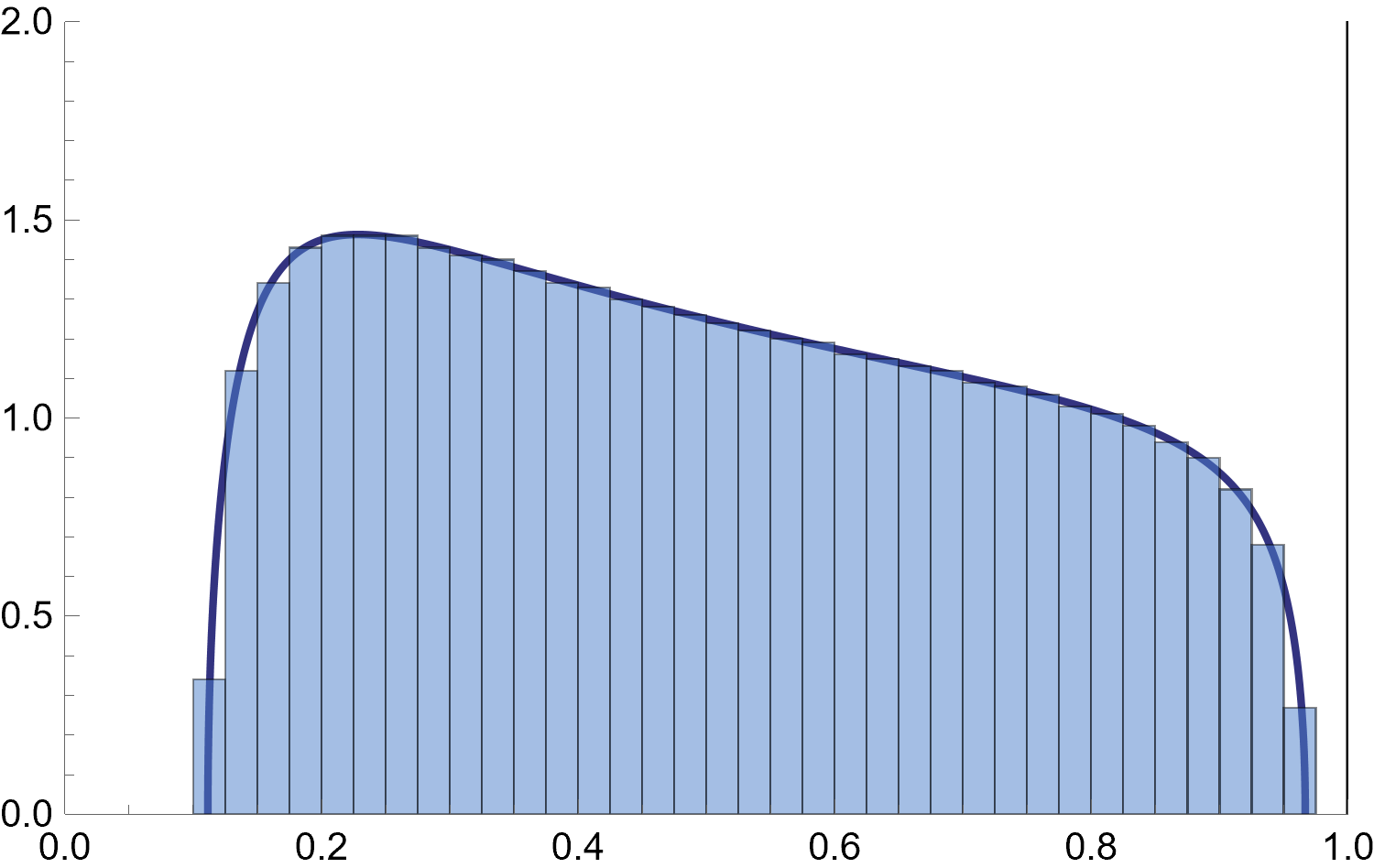}   }     &    \rule{0pt}{3cm} \raisebox{-0.4\height}{\includegraphics[width=0.4\linewidth]{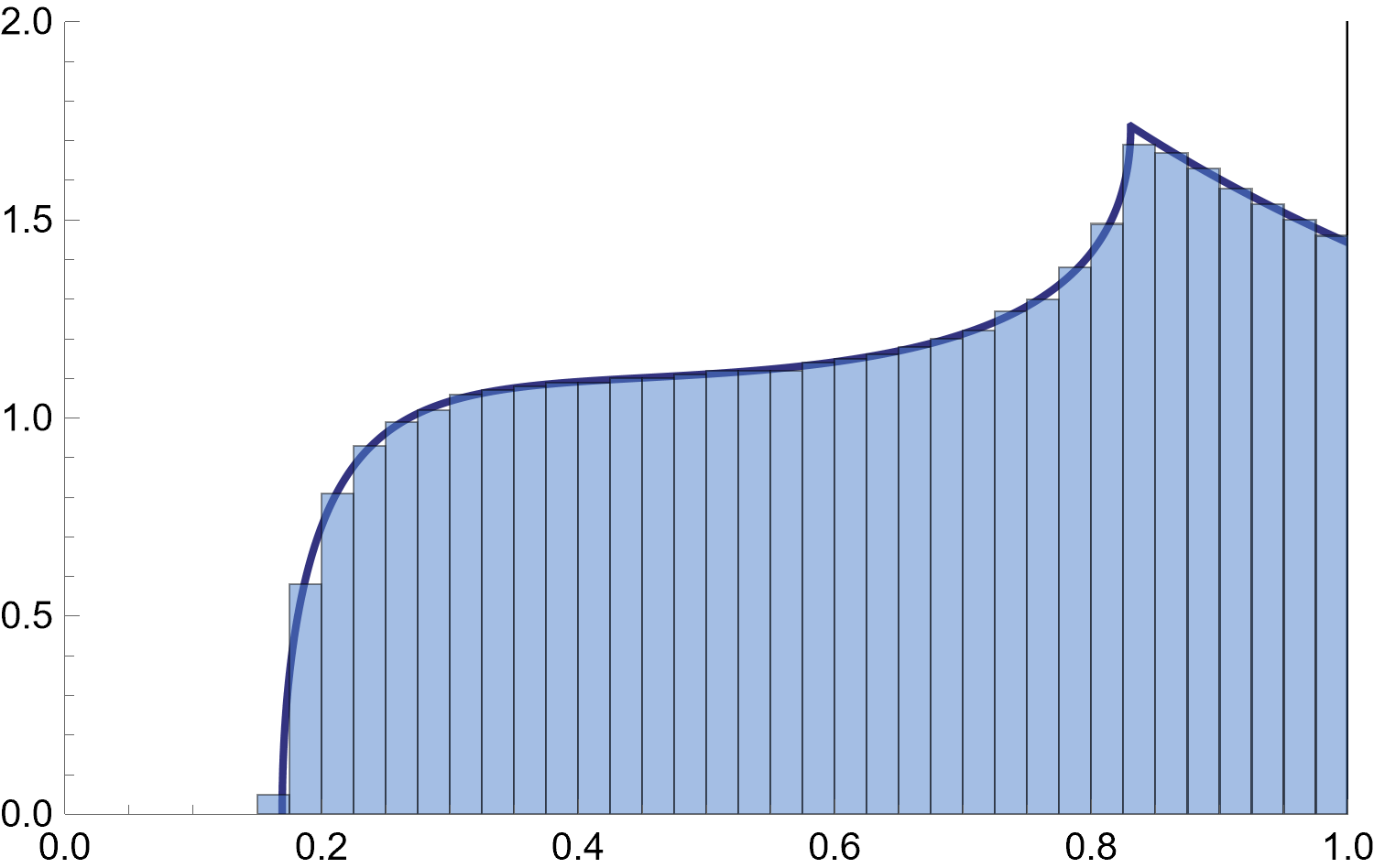} }    
        \\
        \hline
    \end{tabular}
     \caption{The plots illustrate the density \eqref{eqn:def of limiting density} for $c=0$ (first row) and $c=2$ (second row). The corresponding critical values are $\lambda_c\vert_{c=0}=\log 2$ and $\lambda_c\vert_{c=2}\approx 0.6823$. For $c=0$, the parameters are $\lambda=\log(4/3)<\lambda_c$ (left) and $\lambda=\log 3>\lambda_c$ (right). For $c=2$, the parameters are $\lambda=\log(4/3)<\lambda_c$ (left) and $\lambda=\log 2>\lambda_c$ (right). In addition, histograms of the zeros of the $q$-Laguerre polynomials for $N=4000$ are superimposed.}
    \label{Fig_limiting density}
\end{figure}

\subsection{Discussions and related works}

We end this section by providing several remarks on our main results.

\begin{rem}[Saturated and band region]
In the definition of the $q$-deformed Marchenko-Pastur law \eqref{eqn:def of limiting density}, one can easily check that the critical value $\lambda_c$ defined in \eqref{eqn:def of lambdac} is decreasing function of $c$. For instance, we have 
\begin{equation}
\lambda_c|_{c=0} = \log 2, \qquad  \lambda_c|_{c=1} = \log \frac{ \sqrt{5}+1 }{2}.  
\end{equation}
In addition, we note that the critical value $\lambda_c$ is characterised by the condition $b<1$.

When $\lambda > \lambda_c$, the limiting density \eqref{eqn:def of limiting density} develops a portion of its support where it coincides with the upper constraint in \eqref{eq:prob_space}. This part of the support is referred to as the \emph{saturated region}. The other part, where the density is strictly below the upper constraint is called the \emph{band region}; see e.g. \cite{BKMM}. 

Such a phase separation is closely aligned with the well-known decomposition into \textit{frozen (solid)} and \textit{liquid (rough)} regions in random growth models; see e.g. \cite{BCG16,DK21,CC25} and the references therein. 
In that setting, frozen regions correspond to portions of the limit shape where the local slope is forced to saturate boundary constraints, whereas liquid regions arise where the height function varies freely and exhibits nontrivial fluctuations.
The appearance of saturated and band regions in our model follows an analogous mechanism: the equilibrium measure becomes pinned to its upper constraint in the saturated phase, while in the band phase it remains strictly within the admissible range, leading to genuinely fluctuating behaviour. This structural parallel highlights a connection between discrete log-gas ensembles and integrable probabilistic models; see Remark~\ref{Rem_LPP} for more on this direction.
\end{rem}

\begin{rem}[Jackson $q$-integral]
In the $q$-deformed setting, the natural analogue of the integral, denoted by $d_q$, is the Jackson $q$-integral defined by
\begin{equation} \label{def of Jackson integral} 
    \int_{0}^{A}f(x)\,d_{q}x=(1-q)\sum_{j=0}^{\infty}Aq^{j}f(Aq^{j}). 
\end{equation} 
Due to the discrete nature of the Jackson integral, the usual tools of complex analysis are no longer applicable, which is one of the main sources of difficulty in the $q$-deformed setting. Nonetheless, as an intermediate step in the proof of Theorem~\ref{Thm:Spectralmoment} (A), we obtain the following identity for the $q$-deformed Laguerre polynomials:
\begin{align}
\begin{split}
&\quad \int_{0}^{1}x^{p}p_{j}(x;q^{\alpha}|q)^{2}w^{(\mathrm{qL})}(x;q)\,d_{q}x 
\\
&= (1-q)^{p+1}q^{(p+\alpha+1)j}\frac{(q;q)_{j}}{(q^{\alpha+1};q)_{j}}\frac{(q;q)_{\infty}}{(q^{\alpha+1};q)_{\infty}}  [p]_q!  \sum_{i=0}^{p} q^{ (p-i)(\alpha-i) }  \qbinom{p}{i} \qbinom{\alpha+j}{i} \qbinom{p-i+j}{j}.
\end{split}
\end{align}
This expression extends \eqref{eqn:orthogonality of little q Laguerre}, which is recovered by taking $p=0$. Such explicit evaluations of Jackson $q$-integrals are of independent interest. In addition, such exact evaluations are naturally connected to the superintegrable structure of the model and to symmetric function techniques; see \cite[Section~3.3]{LWYZ25}. 
\end{rem}

\begin{rem}[The special case $c=0$: hard edge at the origin and $q$-deformed semi-circle law] 
When $c=0$, the RHS of \eqref{leading order conv for LUE moment} reduces to the Catalan number $C_p := \frac{1}{p+1} \binom{2p}{p}$.  
This is consistent with the well--known fact that the moments of the semicircle law are given by the Catalan numbers, and that the Marchenko-Pastur law~\eqref{def of MP law} with $c=0$ reduces to the semicircle law under the quadratic change of variables $x \mapsto x^{2}$.  

This phenomenon can also be observed in the $q$-deformed model. 
Namely, if $c=0$, we have
    \begin{equation}
        \mathcal{M}^{ \rm (qL) }_{p}\Big|_{c=0}=\frac{1}{\lambda p}I_{1- \mathsf{s} }(p+1,p).
    \end{equation}
This expression coincides with the leading term of the $2p$-th spectral moment of the $q$-deformed GUE recently obtained in \cite[Theorem~2.3]{BFO24}.  
Moreover, for the special case $c=0$, the density \eqref{eqn:def of limiting density} simplifies to  
\begin{align}
\begin{split}
  \rho^{(0)}(x)
            &=\frac{2}{\pi\lambda x}\arctan\sqrt{\frac{1-\sqrt{1-x}}{1+\sqrt{1-x}}\frac{\sqrt{1-x}+1-2e^{-\lambda}}{\sqrt{1-x}-1+2e^{-\lambda}}}\mathbbm{1}_{(0,4e^{-\lambda}(1-e^{-\lambda}) )}(x)
            \\
            &\quad +
            \begin{cases}
                0&\textup{if }\lambda \leq \log 2,
                \smallskip 
                \\
                \displaystyle\frac{1}{\lambda x}\mathbbm{1}_{( 4e^{-\lambda}(1-e^{-\lambda}) ,1)}(x)&\textup{if }\lambda>\log2. 
            \end{cases}
\end{split}
\end{align} 
    In this case the left edge coincides with the origin, producing a singularity that forms a hard edge. Moreover, its quadratic transform $|x| \rho^{(0)}(x^2)$ corresponds to the $q$-deformed semicircle law described in \cite[Theorem~2.6]{BFO24} (see also \cite[Remark~4]{BJO25}).
\end{rem}

\begin{rem}[Continuum limit $q\to1$] \label{Rem_continuum limit q1}
By construction, the leading-order moment $\mathcal{M}_p^{\rm (qL)}$ can be written as 
\begin{equation}
\mathcal{M}_p^{ \rm (qL) } = \int_{a}^b x^p \rho^{(c)}(x)\,dx. 
\end{equation}
As noted earlier in \eqref{qMP to MP}, the appropriately rescaled density 
$\widehat{\rho}^{(c)}$ converges to $\rho_{\rm MP}^{(c)}$ as $\lambda \to 0$. 
The same limiting phenomenon can also be observed directly in the corresponding moments. 
For this, notice that 
   \begin{align*}
\lim_{\lambda\rightarrow0}  \frac{1}{\lambda^{p}}\mathcal{M}_{p}^{ \rm (qL) } = \frac{1}{  p} \sum_{j=0}^p \binom{p}{j} \binom{2p-j}{p-1} c^{j}  
= \frac{1}{p} \sum_{j=0}^p \bigg( \sum_{k=j}^p \binom{p}{k}\binom{p}{k-1} \binom{k}{j} \bigg) c^j  =   \sum_{j=0}^p N(p,j)\sum_{k=0}^j \binom{j}{k} c^k. 
\end{align*} 
Therefore, as $\lambda \to 0$, we obtain
  \begin{equation} 
        \lim_{\lambda\rightarrow0}\frac{1}{\lambda^{p}}\mathcal{M}_{p}^{ \rm (qL) }  =\sum_{j=0}^{p}\mathrm{N}(p,j)(1+c)^{j} ,
    \end{equation} 
which recovers \eqref{evaluation of LUE moments}.

In addition, by \cite[Eq.~(3.46)]{FRW17}, for $c=0$, the spectral moment of the LUE satisfies the expansion  
\begin{align} \label{LUE moment expansion}
 \frac{1}{N^p}m_{N,p}^{ \rm (L) }|_{c=0} = N C_p + \frac{d}{2}\bigg( \binom{2p}{p}-\delta_{p,0} \bigg) +O\Big(\frac{1}{N}\Big).  
\end{align}
Note that by \eqref{def of subleading term qLUE moment}, we also have 
\begin{align*}
\lim_{\lambda \to 0} \frac{1}{\lambda^p} \mathcal{M}_{p,1/2}^{ \rm (qL) }(d) = \frac{d}{2}\bigg( \binom{2p}{p}-\delta_{p,0} \bigg). 
\end{align*}
Therefore, the expansion \eqref{LUE moment expansion} follows by taking $\lambda \to 0$ limit of \eqref{qLUE moment expansion c0}.

Such a limiting behaviour is also reflected at the level of the rate function in the large deviation principle. Notice first that as $\lambda \to 0$, the upper constraint in \eqref{eq:prob_space} disappears. Moreover, by the asymptotic behaviour  
\begin{equation}
\frac{1}{\lambda} \Li_2(\lambda x)= x +O(\lambda), \qquad \textup{as } \lambda \to 0, 
\end{equation}
the energy functional \eqref{eq:functional_rhp} converges to the one \eqref{def of I_V LUE} associated with the LUE.
\end{rem}

\begin{rem}[Large-$N$ expansion of spectral moments] While Theorem~\ref{Thm:Spectralmoment}~(B) concerns primarily the leading-order asymptotic behaviour of ${m}_{N,p}^{(\mathrm{qL})}$, with some additional effort one can obtain the large-$N$ expansion
    \begin{equation}\label{eqn:qpMNp expansion}
        q^{p}{m}_{N,p}^{(\mathrm{qL})}
        =\mathcal{M}_{p}^{ \rm (qL) }N+\mathcal{M}_{p,1/2}^{ \rm (qL) }+\mathcal{M}_{p,1}^{ \rm (qL) } \frac{1}{N} + \cdots. 
    \end{equation}
Although the coefficients in this expansion can in principle be computed explicitly, the resulting formulas are rather involved, and we refrain from presenting them here. For analogous large-$N$ expansions in the classical setting, we refer the reader to \cite{FRW17,WF14} and the references therein. 
\end{rem}

\begin{rem}[Upper constraint condition] \label{Rem_upper constraint}
   Compared with the classical case, the most notable feature of the large deviation principle for the $q$-deformed model is the presence of the upper constraint in \eqref{eq:prob_space}. We briefly explain how this constraint naturally emerges from the discrete nature of the model; see also the discussion at the end of \cite[Subsection~1.2.2]{DD22}.

    Since the model \eqref{def of jpdf for qL} lies on the exponential lattice $q^j$, there exists a sequence of $\ell_j\in \mathbb{N}$ such that $x_j = q^{\ell_j}$ and $|\ell_j - \ell_{k}|>1$ for all $j,k$. Suppose that the upper constraint is violated, so that there exists an interval $I=(a,b)$ and some $\varepsilon>0$ for which
    \begin{equation*}
        \mu(I)>\int_{a}^b \frac{1}{\lambda x} \, dx + \varepsilon = \frac{\log(b) -\log(a)}{\lambda} +\varepsilon\,.
    \end{equation*}
    Note that one can compute the measure of the set $I$ according to the empirical measure $\mu_N(dx)$: 
    \begin{align*}
          \mu_N(I) & = \frac{1}{N} \#\{j\in\mathbb{N} \,|\, a<x_j<b\}
          \\
          & = \frac{1}{N} \#\Big\{j\in\mathbb{N} \,|\, -N \frac{\log(b)}{\lambda}<\ell_j<-N \frac{\log(a)}{\lambda}\Big\} \leq \frac{1}{N}\Big(N\frac{\log(b) - \log(a)}{\lambda} + 1 \Big).
    \end{align*} 
    Therefore it follows that $  \mu(I) - \mu_N(I) > \varepsilon - 1/N.$ On the other hand, if $N> 2/\varepsilon$ the previous inequality reads $\mu(I) - \mu_N(I) > \varepsilon/2$, which in turn implies that $ d(\mu,\mu_N)>\varepsilon/2$ for any admissible distance $d$. This yields a contradiction. 
\end{rem}

\begin{rem}[Zero distributions of orthogonal polynomials] 
The zero distribution of orthogonal polynomials plays a central role in many aspects of the theory, most notably in describing their asymptotic behaviour. A well-established framework for determining the limiting zero distribution is available, and in our setting we follow the general approach developed by Kuijlaars and Van Assche \cite{KV99}; see also \cite{Ha18} and references therein. 
Although Theorem~\ref{Thm:limiting density in growth regime} (C) follows relatively directly from this classical theory, it is still useful to state it explicitly, as it confirms the limiting object obtained in parts (A) and (B) and clarifies how these results fit together with the additional viewpoints discussed above.
In addition, we mention that recent developments have introduced a complementary perspective 
based on finite free probability, offering further structural insights into such zero distributions; see \cite{JKM25,JKM25a}. 
\end{rem}

\begin{rem}[$\beta$-ensemble generalisation]
In analogy with the classical $\beta$-ensemble, one may consider
finite-temperature extensions of the joint probability density function
\eqref{def of jpdf for qL} of the $q$-Laguerre ensemble, in which the
interaction is deformed by an additional temperature parameter.
A natural example is the $(q,t)$-extension (see e.g. \cite{Ol21,BF25a}), which arises naturally from Macdonald polynomial theory and its underlying integrable structure. We also refer to \cite{HMO25} for a different type of deformation of the $q$-Laguerre ensemble.
In the regime $q^{-\lambda/N}$, the natural associated functional is given by 
    \begin{equation}
    \label{eq:functional_rhp_beta}
        I_\beta[\mu] : = - \frac{\beta}{2}\int_{0}^{1}\int_{0}^{1} \log(|x-y|)\, \mu(dx) \, \mu(dy)  + \int_0^1 \Big( \frac{1}{\lambda}\Li_2(x)-c \log x \Big)\, \mu(dx)\,,
    \end{equation}
where $\mu(dx)\in\mathcal{P}_{\lambda}(0,1)$.
Although the upper constraint on the support precludes obtaining the
minimiser by a trivial rescaling of our main result, the arguments developed in Subsection~\ref{Section_Thm B} extend in a straightforward manner to this setting. 

In addition, in the high-temperature regime where $\beta = O(1/N)$, the resulting Hamiltonian involves not only an energy functional of the form \eqref{eq:functional_rhp_beta} but also an entropy term; see e.g.  \cite{ABG12,ABMV12}. A recent study \cite{CD25} investigates such a high-temperature regime in discrete ensembles, and extending this analysis to the present setting appears to be an interesting direction for future research. 
\end{rem}

\begin{rem}[Local statistics and scaling limits]
One of the most natural directions for further investigation is the study of local statistics in $q$-deformed ensembles. Owing to the underlying determinantal structure, these local statistics can be analysed through the asymptotic behaviour of the associated orthogonal polynomials. In particular, in the regime $q^{-\lambda/N}$, the exponential $q$-lattice $\{1,q,q^2,\dots\}$ effectively behaves like a linear lattice. Consequently, one expects that the universal local statistics identified in \cite{BKMM} emerge at the bulk and soft edges of the band regions, as well as at the transition points where the band meets the saturated regions. (We also refer to \cite{BO24,GS25} for related developments in integrable probability, which---guided by the universality philosophy---are expected to be relevant to the present model, both in the hard-edge scaling limit and in the regime where the soft edge meets the hard edge.)
This constitutes an interesting direction that requires an asymptotic analysis of orthogonal polynomials in a double-scaling regime, for instance via the Riemann--Hilbert method, which we plan to pursue in future work.
\end{rem}

\begin{rem}[Connection to the last passage percolation] \label{Rem_LPP}
It is by now well understood that random matrix theory (or log gases) provides an effective framework for deriving fundamental properties of integrable probability models. One prominent example is the seminal work of Johansson~\cite{Joh00}, who used a bijection between last passage percolation (LPP) with i.i.d. geometric weights and the Meixner ensemble to obtain the law of large numbers, fluctuations, and large deviations of the last passage time. These ideas have since been developed in several variants~\cite{BR01,FW04}; see also the recent works~\cite{BS25,BCMS25,GS25,CMR26} and references therein. More recently, connections to classical random matrix ensembles have also emerged. In particular, the above LPP model is linked to the Jacobi unitary ensemble~\cite{BCMS25}. As a special case, replacing geometric weights by exponential ones yields a model related to the classical LUE. From this perspective, it is perhaps surprising that the law of large numbers for the last passage time can be derived directly from the Marchenko-Pastur law.

Beyond the classical ensembles, $q$-deformations play a significant role in integrable probability. In~\cite{BO24}, LPP with \emph{non-identically distributed} weights was related to a $q$-deformed model known as the discrete Muttalib--Borodin ensemble. Our model arises as a special case of this general framework, and, following the viewpoint of~\cite{BCMS25}, we expect that the law of large numbers for the last passage time in this broader setting should be governed by the $q$-deformed Marchenko-Pastur law. 
\end{rem}

\subsection*{Organisation of the paper} In Section~\ref{Section_moments combinatorics}, we introduce the combinatorial framework for orthogonal polynomials and prove Theorem~\ref{Thm:Spectralmoment}.
Section~\ref{Section_Thm main proof} is devoted to the proof of Theorem~\ref{Thm:limiting density in growth regime}. In particular, statements (A), (B), and (C) are established in Subsections~\ref{Subsec_Thm main (A)}, \ref{Section_Thm B}, and \ref{Subsec_Thm main (C)}, respectively.
In Appendix~\ref{app:integrals}, we collect several integral evaluations that play an important role in the analysis of Section~\ref{Section_Thm B}.

\subsection*{Acknowledgements}
Sung-Soo Byun was supported by the National Research Foundation of Korea grants (RS-2023-00301976, RS-2025-00516909).
We thank the organisers of the MATRIX program “Log-gases in Caeli Australi”, held in Creswick, Victoria, Australia in August 2025, for creating a stimulating environment that inspired the present collaboration. We also thank Peter J. Forrester and Jaeseong Oh for many inspiring discussions during the preparation of this work.

\section{Proof of Theorem~\ref{Thm:Spectralmoment}} \label{Section_moments combinatorics}

In this section, we prove Theorem~\ref{Thm:Spectralmoment}. 
Subsection~\ref{Subsection_FV theory} recalls the basic integrable structure  of random unitary ensembles together with the Flajolet--Viennot theory on the combinatorics of orthogonal polynomials. 
In Subsection~\ref{Subsection_LUE moment revisit}, we revisit the spectral
moments of the LUE and re-derive them within the Flajolet--Viennot framework.
This serves not only as a warm-up but also as an important intermediate step towards the general $q$-deformed setting. 
Subsections~\ref{Subsection_Them moment A} and~\ref{Subsection_Them moment B} are devoted to the proofs of Theorem~\ref{Thm:Spectralmoment}~(A) and~(B), respectively.

\subsection{Integrable structure of $q$-LUE and Flajolet-Viennot theory} \label{Subsection_FV theory}

For a given weight function $w : \R \to \R_{\ge 0}$, the $1$-point function of \eqref{def of joint eigenvalue of unitary ensemble} is defined by
\begin{equation}
    \rho_{N}(x)=\rho_{N,w}(x)=N\int_{\mathbb{R}^{N-1}}\mathbf{P}(x,x_2,\cdots,x_N)\,dx_2\cdots\,dx_{N}.
\end{equation}
The $1$-point function can be effectively analysed in terms of the associated orthogonal polynomials. Consider the sequence of monic orthogonal polynomials $\{P_{n}\}_{n\geq0}$ satisfying the orthogonality condition
\begin{equation}
    \int_{\mathbb{R}}P_n (x) P_m (x)w(x)\,dx =h_{n}\delta_{nm}.
\end{equation}
Then it is well known that 
\begin{equation}\label{def of 1point function with orthogonal polynomial}
    \rho_{N}(x)=\sum_{j=0}^{N-1}\frac{1}{h_{j}}P_{j}(x)^{2}w(x). 
\end{equation} 
Furthermore, the spectral moment $m_{N,p}$ in \eqref{def of spectral moments mNp} can be expressed using one-point function as 
\begin{equation}\label{eqn:spectral moment and 1point function}
    m_{N,p}=\int_{\mathbb{R}}x^{p}\rho_{N}(x)\,dx.
\end{equation}
Then, using the notation
\begin{equation}\label{def of mathfrak m}
    \mathfrak{m}_{p,j}= \frac{1}{h_{j}}\int_{\mathbb{R}}x^{p} P_{j}(x)^{2}w(x)\,dx,
\end{equation}
we have the orthogonal polynomial representation of the spectral moments
\begin{equation}\label{eqn: relation between m and mathfrak m}
    m_{N,p}= \sum_{j=0}^{N-1}\mathfrak{m}_{p,j}.
\end{equation}
These representations remain valid in the framework of the $q$-unitary ensemble and the $q$-orthogonal polynomial, where the usual integral is replaced with the Jackson $q$-integral \eqref{def of Jackson integral}.

The theory of Flajolet and Viennot \cite{Fl80,Vi00} establishes a correspondence between the moments of orthogonal polynomials and certain weighted lattice paths; see also \cite{CKS16} for a review. We briefly recall the aspects of this framework that are relevant to our discussion.
Within this theory, the fundamental structural feature of orthogonal polynomials is their three-term recurrence relation
\begin{equation}\label{eqn:three term recurrence of orthogonal polynomial}
    P_{n+1}(x)=(x-b_{n})P_{n}(x)-\lambda_{n}P_{n-1}(x),  
\end{equation}
where $\{b_n\}_{n\ge 0}$ and $\{\lambda_n\}_{n\ge 0}$ are associated coefficient sequences. These sequences determine the weights assigned to the corresponding lattice paths.

A lattice path is a finite sequence $w= (s_{0}, s_{1},\cdots ,s_{n})$, where $s_{j}=(x_j,y_j)\in\mathbb{Z}\times\mathbb{Z}$. We call an adjacent pair $(s_{j},s_{j+1})$ a \textit{step}. In particular, we refer to each step $(s_j ,s_{j+1})$ as follows:
\begin{itemize}
    \item \emph{East step} if $x_{j+1}=x_{j}+1$ and $y_{j+1}=y_{j}$;
    \smallskip 
    \item \emph{North-East step} if $x_{j+1}=x_{j}+1$ and $y_{j+1}=y_{j}+1$;
     \smallskip 
    \item \emph{South-East step} if $x_{j+1}=x_{j}+1$ and $y_{j+1}=y_{j}-1$. 
\end{itemize}
Furthermore, we say a step $(s_j,s_{j+1})$ is of height $k$ if $y_j=k$. 
A \emph{Motzkin path} is a lattice path in the first quadrant that consists 
only of East, North-East, and South-East steps. 
Given the coefficient sequences $\{b_n\}_{n\ge 0}$ and $\{\lambda_n\}_{n\ge 0}$, we assign weights to the steps of a Motzkin path as follows:
\begin{itemize}
    \item An East step of height $k$ has weight $b_k$;
     \smallskip 
    \item A North-East step has weight $1$;
     \smallskip 
    \item A South-East step of height $k$ has weight $\lambda_{k}$.
\end{itemize}
Then a weight of a given Motzkin path $\omega$ is defined as the product of the weights of every step in $\omega$, and denoted by $\mathrm{wt}_{b,\lambda}(\omega)$. See Figure~\ref{Fig_Motzkin}. 

\begin{figure}[h]
\centering
\begin{tikzpicture}[scale=0.8]
    \draw[step=1cm,gray!30,very thin] (0,0) grid (12,4);
    \draw[very thick, blue] (0,3) -- (1,3) ;
    \draw[very thick, blue] (1,3) -- (3,1) ;
    \draw[very thick, blue] (3,1) -- (4,1) ;
    \draw[very thick, blue] (4,1) -- (5,2) ;
    \draw[very thick, blue] (5,2) -- (7,2) ;
    \draw[very thick, blue] (7,2) -- (9,4) ;
    \draw[very thick, blue] (9,4) -- (12,1) ;
    \fill (0, 3) circle (2pt);
    \fill (1, 3) circle (2pt);
    \fill (2, 2) circle (2pt);
    \fill (3, 1) circle (2pt);
    \fill (4, 1) circle (2pt);
    \fill (5, 2) circle (2pt);
    \fill (6, 2) circle (2pt);
    \fill (7, 2) circle (2pt);
    \fill (8, 3) circle (2pt);
    \fill (9, 4) circle (2pt);
    \fill (10, 3) circle (2pt);
    \fill (11, 2) circle (2pt);
    \fill (12, 1) circle (2pt);
    \node at (0.5,3.2) {$b_{3}$};
    \node at (3.5,1.2) {$b_{1}$};
    \node at (5.5,2.2) {$b_{2}$};
    \node at (6.5,2.2) {$b_{2}$};
    \node at (1.7,2.7) {$\lambda_{3}$};
    \node at (2.7,1.7) {$\lambda_{2}$};
    \node at (9.7,3.7) {$\lambda_{4}$};
    \node at (10.7,2.7) {$\lambda_{3}$};
    \node at (11.7,1.7) {$\lambda_{2}$}; 
\end{tikzpicture}
\caption{A Motzkin path from $(0,3)$ to $(12,1)$ with weight $b_{1}b_{2}^{2}b_{3}\lambda_{2}^{2}\lambda_{3}^{2}\lambda_{4}$} \label{Fig_Motzkin}
\end{figure}
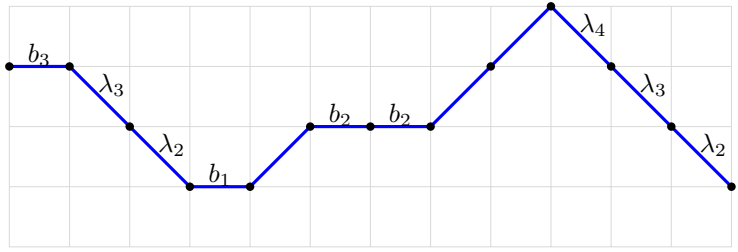

For nonnegative integers $n$, $k$ and $l$, let $\mathrm{Mot}_{n,k,l}$ denote the set of all Motzkin paths from $(0,k)$ to $(n,l)$. 
Then  $\mathfrak{m}_{p,j}$ in \eqref{def of mathfrak m} is given by  
\begin{equation} \label{mathfrak m in terms of Motzkin}
   \mathfrak{m}_{p,j}=\sum_{\omega\in \mathrm{Mot}_{p,j,j}}\mathrm{wt}_{b,\lambda}(\omega);
\end{equation}
see e.g. \cite[Proposition 3.2]{BFO24}. 
Then together with \eqref{eqn: relation between m and mathfrak m}, this yields the  the combinatorial enumeration of the spectral moments
\begin{equation}\label{Flajolet Viennot}
    m_{N,p}=\sum_{j=0}^{N-1}\sum_{\omega\in \mathrm{Mot}_{p,j,j}}\mathrm{wt}_{b,\lambda}(\omega). 
\end{equation} 
We aim to obtain a closed formula for the RHS of \eqref{mathfrak m in terms of Motzkin} by using a combinatorial approach.
We also mention the work of Kerov~\cite{Ke99}, who applies such a discrete path-counting method to compute the moments of the GUE.

\subsection{Spectral moments of the LUE revisited} \label{Subsection_LUE moment revisit}

In preparation for enumerating the spectral moments of the $q$-deformed LUE, it is instructive to first revisit how the spectral moments of the classical LUE can be derived within our framework. 
In the next section, we will see that the combinatorial interpretation for the 
$q$-LUE arises naturally by imposing additional statistics on the 
underlying combinatorial structure of the original LUE.

We now state the specialisation of Theorem~\ref{Thm:Spectralmoment} (A) 
to the classical LUE case.
 
\begin{prop}[\textbf{Spectral moments of the LUE}]  \label{prop:spetral moment of LUE}
Let $m_{N,p}^{(\mathrm L)}$ be the spectral moment of the LUE. Then for any nonnegative integers $p,\alpha$ and $N \in \mathbb{N}$, we have 
\begin{equation}  \label{evaluation of LUE moments}
 m_{N,p}^{(\mathrm L)} = p! \sum_{j=0}^{N-1}\sum_{i=0}^{p} \binom{p}{i}  \binom{\alpha+j}{i}   \binom{p-i+j}{j}. 
\end{equation} 
\end{prop} 
 
We note that in the literature, some equivalent formulae for the spectral moment of LUE are investigated. For instance, we have 
\begin{align}
\begin{split}
\label{eqn:LUE moment 2}
    m_{N,p}^{(\mathrm{L})}&=\frac{1}{p}\sum_{i=1}^{p}(-1)^{i-1}\frac{(N+\alpha+p-i)_{p}(N+p-i)_{p}}{(p-i)!(i-1)!}
    \\
    &= N(N+\alpha)(p-1)!\sum_{j=0}^{\floor{\frac{p-1}{2}}}\frac{1}{j+1}\binom{N-1}{j}\binom{N+\alpha-1}{j}\binom{2N+\alpha+p-2j-2}{p-2j-1};
\end{split}
\end{align} 
see \cite[Theorem 2.5]{HSS92}, \cite[Proposition 9.1]{HT03}, \cite[Appendix A]{MRW15} and \cite[Proposition 3.11]{FRW17}. We also refer to a recent work \cite{FN26} for the corresponding results in the fixed-trace version of the LUE and for their applications to quantum information theory.

\medskip 

As noted earlier, the orthogonal polynomials corresponding to the Laguerre 
weight $w^{(\mathrm{L})}$ are the generalised Laguerre polynomials. 
They satisfy the orthogonality relation 
\begin{equation}
    \int_{0}^{\infty}L_{n}^{(\alpha)}(x)L_{m}^{(\alpha)}(x)w^{(\mathrm{L})}(x)\,dx=h^{(\mathrm{L})}_{n}\delta_{nm},\qquad h_{n}^{(\mathrm{L})}=\frac{\Gamma(n+\alpha+1)}{n!}.
\end{equation} 
Then it follows from \eqref{mathfrak m in terms of Motzkin} that
\begin{equation}\label{eqn:mathfrakml=wtl}    \mathfrak{m}_{p,j}^{(\mathrm{L})}:=\int_{0}^{\infty}\frac{x^{p}L^{(\alpha)}_{j}(x)^{2}}{h_{j}^{(\mathrm{L})}}w^{(\mathrm{L})}(x)\,dx =\sum_{\omega\in\mathrm{Mot}_{p,j,j}}\mathrm{wt}_{b,\lambda}(\omega),
\end{equation}
where, by \eqref{eqn:three term of laguerre}, two sequences are given by
\begin{equation} \label{bn lambdan for Laguerre}
    b_{n}=2n+\alpha+1,\qquad \lambda_{n}=n(n+\alpha). 
\end{equation} 
However, there is no general method to directly count the total weight over all weighted Motzkin paths. To overcome this difficulty, we introduce a different combinatorial model that allows for explicit enumeration. Specifically, we present an alternative way using a new combinatorial notion called \emph{bipartite matching}. We remark that the bijection introduced below is somewhat more intricate---but also more novel---than those previously used in \cite{BFO24,BJO25} for the $q$-deformed unitary ensembles with discrete Gaussian or Al Salam--Carlitz weights. Whereas the earlier constructions relied on single-row matchings, the combinatorial structure underlying the $q$-LUE naturally leads us to a two-row matching framework; see Figure~\ref{fig:comparison_models} and \cite[Figure~4]{BFO24}.

For nonnegative integers $n$, $j$ and $\alpha$, define a set
\begin{equation}
    \mathcal{S}^{\alpha}(n,j)=T_{-}\cup T_{+}\cup B_{-}\cup B_{+},
\end{equation}
where
\begin{equation}
    T_{-}=\{-\alpha-j,\cdots,-1\},\qquad T_{+}=\{1,\cdots,n\},\qquad B_{-}=\{-\widetilde{j},\cdots,-\widetilde{1}\},\qquad B_{+}=\{\widetilde{1},\cdots,\widetilde{n}\}.
\end{equation}
We refer to each element of $\mathcal{S}^{\alpha}(n,j)$ as a \emph{vertex}. 
The vertices are classified as follows:
\begin{itemize}
    \item Vertices in $T_{-}\cup B_{-}$ are called \emph{negative}, and all others are called \emph{positive}; 
    \smallskip
    \item Vertices in $T_{-} \cup T_{+}$ lie in the \emph{top row}, and those in $B_{-}\cup B_{+}$ in the \emph{bottom row}.
\end{itemize}
A bipartite matching is a set partition on $\mathcal{S}^{\alpha}(n,j)$ consisting of
\begin{itemize}
    \item singleton sets, called \emph{isolated vertices};
    \smallskip 
    \item set of two elements $a\in\ T_{-}\cup T_{+}$ and $\tilde{b}\in B_{-}\cup B_{+}$, called \emph{edges}.
\end{itemize}

Let $M_{n,j}^{(\alpha)}$ denote the set of all bipartite matchings on $\mathcal{S}^{\alpha}(n,j)$ satisfying the following condition: there is no edge $(a,\tilde{b})$ with $a\in T_{-}$ and $\widetilde{b}\in B_{-}$. Pictorially, a bipartite matching $M \in M_{n,j}^{(\alpha)}$ can be represented by two horizontal rows of vertices---($j+n+\alpha$) in the top row and $(j+n)$ in the bottom row---together with a collection of edges between them. The above condition requires that no edge connects $\alpha+j$ negative vertices in the top row to $j$ negative vertices in the bottom row; see Figure~\ref{fig:bipartite matching}.

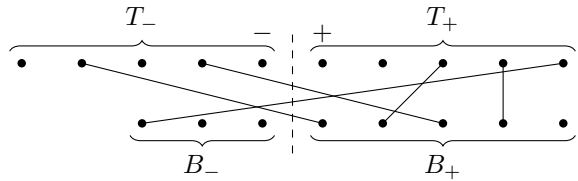
\begin{figure}[h]
    \centering
    \begin{tikzpicture}[scale=0.8]
        \foreach \i in {1,2,...,10}{\fill[black] (\i, 1) circle (2pt);}
        \foreach \i in {3,4,...,10}{\fill[black] (\i, 0) circle (2pt);}
        \draw[decorate,decoration={brace,amplitude=4pt}] (0.8,1.2) -- (5.2,1.2) node[midway,above=4pt] {$T_{-}$};
        \draw[decorate,decoration={brace,amplitude=4pt}] (5.8,1.2) -- (10.2,1.2) node[midway,above=4pt] {$T_{+}$};
        \draw[decorate,decoration={brace,mirror,amplitude=4pt}] (2.8,-0.2) -- (5.2,-0.2) node[midway,below=4pt] {$B_{-}$};
        \draw[decorate,decoration={brace,mirror,amplitude=4pt}] (5.8,-0.2) -- (10.2,-0.2) node[midway,below=4pt] {$B_{+}$};
        \draw[black] (6,0) -- (2,1);
        \draw[black] (8,1) -- (7,0);
        \draw[black] (4,1) -- (8,0);
        \draw[black] (9,1) -- (9,0);
        \draw[black] (10,1) -- (3,0);
        \draw[dashed] (5.5,-0.5) -- (5.5,1.5);
        \node at (5,1.5) {$-$};
        \node at (6,1.5) {$+$};
    \end{tikzpicture}
    \caption{A bipartite matching in $M_{5,3}^{(2)}$. The dashed vertical line separates negative and positive vertices. No edge connects a vertex in $T_{-}$ to a vertex in $B_{-}$.}
    \label{fig:bipartite matching}
\end{figure}

To construct a bijection between the frameworks of a Motzkin path and 
a bipartite matching, we introduce the notion of a \emph{Laguerre history}. 
This is obtained by assigning labels to the East and South--East steps 
of a Motzkin path as follows. 
For integers $0 \le m \le k+\alpha$ and $1 \le l \le k$,
\begin{itemize}
    \item an East step at height $k$ is labeled by either $m$ or $\widetilde{l}$;
    \smallskip 
    \item a South--East step at height $k$ is labeled by a pair 
    $[m,\widetilde{l}]$.
\end{itemize} 
Thus, there are $(2k+\alpha+1)$ possible labels for an East step and 
$k(k+\alpha)$ possible labels for a South--East step.

With this labeling rule, we obtain the following enumeration of 
Motzkin-path weights in terms of bipartite matchings.

\begin{lem} \label{Lem_Laguerre history bijection}
Let $M_{n,j}^{(\alpha)}(l)$ be the subset of matchings in $M_{n,j}^{(\alpha)}$ 
that contain exactly $l$ edges. Let $b_n$ and $\lambda_n$ be given by \eqref{bn lambdan for Laguerre}. Then we have
    \begin{equation} \label{eqn in Lem Laguerre history}
       \sum_{\omega\in\mathrm{Mot}_{p,j,j}}\mathrm{wt}_{b,\lambda}(\omega) =  \Big|M_{p,j}^{(\alpha)}(p)\Big| .
    \end{equation}
\end{lem}

\begin{proof}
The lemma follows from a one-to-one correspondence between Laguerre 
histories on Motzkin paths in $\mathrm{Mot}_{p,j,j}$ and bipartite 
matchings in $M_{p,j}^{(\alpha)}(p)$. 
We construct this correspondence as a dynamical bijection in the 
following manner: as we traverse the steps of a Motzkin path, we 
sequentially add vertices and edges to the bipartite matching according 
to the step type and its assigned label. Each rule is illustrated in  Figure~\ref{fig:comparison_models}. 
\begin{itemize}
    \item[(S1)] Initially, place $\alpha$ vertices horizontally along the top row.  
        \smallskip 
    \item[(S2)] Add $j$ vertices to each of the top and bottom rows, 
        corresponding to the steps from $(0,0)$ to $(0,j)$. 
        \smallskip 
    \item[(S3)] For an East step of height $k$ labeled $m$ with 
        $1 \le m \le k+\alpha$, introduce an edge connecting the newly 
        added vertex in the bottom row to the $m$-th isolated vertex from 
        the left in the top row.  
    \smallskip
    \item[(S4)] For a North–East step, simply add the two new vertices to the top and bottom rows, but do not introduce any edge.  
    \smallskip 
    \item[(S5)] For a South--East step of height $k$ labeled 
        $[m,\widetilde{l}]$, introduce two edges: one connecting the newly 
        added bottom vertex to the $m$-th isolated top vertex, and another 
        connecting the newly added top vertex to the $l$-th isolated bottom vertex. 
        \smallskip 
        \item[(S6)] For an East step labeled $0$, add a vertical edge between the two new vertices. 
        \smallskip 
        \item[(S7)] For an East step of height $k$ labeled $\widetilde{l}$ with $1 \le l \le k$, introduce an edge connecting the newly added vertex in the top row to the $l$-th isolated vertex from the left in the bottom row. 
\end{itemize}

Since the Motzkin path runs from $(0,j)$ to $(p,j)$, the numbers of 
South--East and North--East steps coincide, so the resulting bipartite 
matching contains exactly $p$ edges. Moreover, the height of each step 
records the number of isolated vertices in the top and bottom rows. 
Thus the construction is invertible, yielding a bijection between all 
Laguerre histories on $\mathrm{Mot}_{p,j,j}$ and the matchings in 
$M_{p,j}^{(\alpha)}(p)$.
\end{proof}

\begin{figure}[t]   \centering  
    \begin{tabular}{|C{1cm}|C{6cm}|C{6cm}|}
        \hline
       \rule{0pt}{0.75cm} 
       \cellcolor{gray!25}  & \cellcolor{gray!10} Motzkin paths  & \cellcolor{gray!10} Bipartite matchings
        \\
        \hline
        \cellcolor{gray!10} (S1)
        &   
        \medskip 
        \begin{tikzpicture}[scale=0.6]
            \draw[step=1cm,gray!30,very thin] (0,0) grid (5,4);
        \end{tikzpicture}
        &      
        \begin{tikzpicture}[scale=0.6]
            \foreach \i in {1,2,...,10}{\fill[black] (\i, 0) circle (0pt);}
            \foreach \i in {1,2,...,10}{\fill[black] (\i, 1) circle (0pt);}
            \foreach \i in {1,2}{\fill[red] (\i, 1) circle (2pt);}
        \end{tikzpicture}
         \\
        \hline 
               \cellcolor{gray!10} (S2)
        &    
        \medskip 
        \begin{tikzpicture}[scale=0.6]
            \draw[step=1cm,gray!30,very thin] (0,0) grid (5,4);
            \draw[very thick, blue] (0,0) -- (0,3) ;
        \end{tikzpicture}
        &     
\begin{tikzpicture}[scale=0.6]
    \foreach \i in {1,2,...,10}{\fill[black] (\i, 0) circle (0pt);}
    \foreach \i in {1,2,...,10}{\fill[black] (\i, 1) circle (0pt);}
    \foreach \i in {1,2}{\fill[red] (\i, 1) circle (2pt);}
    \foreach \i in {3,4,5}{\fill[blue] (\i, 0) circle (2pt);}
    \foreach \i in {3,4,5}{\fill[blue] (\i, 1) circle (2pt);}
\end{tikzpicture}
         \\
        \hline 
                       \cellcolor{gray!10} (S3)
        &    
        \medskip 
\begin{tikzpicture}[scale=0.6]
    \draw[step=1cm,gray!30,very thin] (0,0) grid (5,4);
    \draw[very thick] (0,0) -- (0,3) ;
    \draw[very thick, blue] (0,3) -- (1,3) ;
    \node[blue] at (0.5,2.6) {${2}$};
\end{tikzpicture}
        &    
\begin{tikzpicture}[scale=0.6]
    \foreach \i in {1,2,...,10}{\fill[black] (\i, 0) circle (0pt);}
    \foreach \i in {1,2,...,10}{\fill[black] (\i, 1) circle (0pt);}
    \foreach \i in {1,2}{\fill[red] (\i, 1) circle (2pt);}
    \foreach \i in {3,4,5}{\fill[black] (\i, 0) circle (2pt);}
    \foreach \i in {3,4,5}{\fill[black] (\i, 1) circle (2pt);}
    \fill[blue] (6,0) circle (2pt);
    \fill[blue] (6,1) circle (2pt);
    \draw[blue] (6,0) -- (2,1);
\end{tikzpicture}
         \\
        \hline 
                              \cellcolor{gray!10} (S4)
        &    
        \medskip 
 \begin{tikzpicture}[scale=0.6]
    \draw[step=1cm,gray!30,very thin] (0,0) grid (5,4);
    \draw[very thick] (0,0) -- (0,3) ;
    \draw[very thick] (0,3) -- (1,3) ;
    \node at (0.5,2.6) {${2}$};
    \draw[very thick, blue] (1,3) -- (2,4);
\end{tikzpicture}
        &     
\begin{tikzpicture}[scale=0.6]
    \foreach \i in {1,2,...,10}{\fill[black] (\i, 0) circle (0pt);}
    \foreach \i in {1,2,...,10}{\fill[black] (\i, 1) circle (0pt);}
    \foreach \i in {1,2}{\fill[red] (\i, 1) circle (2pt);}
    \foreach \i in {3,4,5}{\fill[black] (\i, 0) circle (2pt);}
    \foreach \i in {3,4,5}{\fill[black] (\i, 1) circle (2pt);}
    \fill[black] (6,0) circle (2pt);
    \fill[black] (6,1) circle (2pt);
    \draw[black] (6,0) -- (2,1);
    \fill[blue] (7,0) circle (2pt);
    \fill[blue] (7,1) circle (2pt);
\end{tikzpicture}
         \\
        \hline 
                                    \cellcolor{gray!10} (S5)
        &   
        \medskip 
 \begin{tikzpicture}[scale=0.6]
    \draw[step=1cm,gray!30,very thin] (0,0) grid (5,4);
    \draw[very thick] (0,0) -- (0,3) ;
    \draw[very thick] (0,3) -- (1,3) ;
    \node at (0.5,2.6) {${2}$};
    \draw[very thick] (1,3) -- (2,4);
    \draw[very thick, blue] (2,4) -- (3,3);
    \node[blue] at (2.15,3) {[3,$\widetilde{4}$]};
\end{tikzpicture} 
        &    
\begin{tikzpicture}[scale=0.6]
    \foreach \i in {1,2,...,10}{\fill[black] (\i, 0) circle (0pt);}
    \foreach \i in {1,2,...,10}{\fill[black] (\i, 1) circle (0pt);}
    \foreach \i in {1,2}{\fill[red] (\i, 1) circle (2pt);}
    \foreach \i in {3,4,5}{\fill[black] (\i, 0) circle (2pt);}
    \foreach \i in {3,4,5}{\fill[black] (\i, 1) circle (2pt);}
    \fill[black] (6,0) circle (2pt);
    \fill[black] (6,1) circle (2pt);
    \draw[black] (6,0) -- (2,1);
    \fill[black] (7,0) circle (2pt);
    \fill[black] (7,1) circle (2pt);
    \fill[blue] (8,0) circle (2pt);
    \fill[blue] (8,1) circle (2pt);
    \draw[blue] (8,1) -- (7,0);
    \draw[blue] (4,1) -- (8,0);
\end{tikzpicture}
         \\
        \hline 
                                            \cellcolor{gray!10} (S6)
        &   
        \medskip 
\begin{tikzpicture}[scale=0.6]
    \draw[step=1cm,gray!30,very thin] (0,0) grid (5,4);
    \draw[very thick] (0,0) -- (0,3) ;
    \draw[very thick] (0,3) -- (1,3) ;
    \node at (0.5,2.6) {${2}$};
    \draw[very thick] (1,3) -- (2,4);
    \draw[very thick] (2,4) -- (3,3);
    \node at (2.15,3) {[3,$\widetilde{4}$]};
    \draw[very thick, blue] (3,3) -- (4,3);
    \node[blue] at (3.5,2.6) {${0}$};
\end{tikzpicture} 
        &     
\begin{tikzpicture}[scale=0.6]
    \foreach \i in {1,2,...,10}{\fill[black] (\i, 0) circle (0pt);}
    \foreach \i in {1,2,...,10}{\fill[black] (\i, 1) circle (0pt);}
    \foreach \i in {1,2}{\fill[red] (\i, 1) circle (2pt);}
    \foreach \i in {3,4,5}{\fill[black] (\i, 0) circle (2pt);}
    \foreach \i in {3,4,5}{\fill[black] (\i, 1) circle (2pt);}
    \fill[black] (6,0) circle (2pt);
    \fill[black] (6,1) circle (2pt);
    \draw[black] (6,0) -- (2,1);
    \fill[black] (7,0) circle (2pt);
    \fill[black] (7,1) circle (2pt);
    \fill[black] (8,0) circle (2pt);
    \fill[black] (8,1) circle (2pt);
    \draw[black] (8,1) -- (7,0);
    \draw[black] (4,1) -- (8,0);
    \fill[blue] (9,0) circle (2pt);
    \fill[blue] (9,1) circle (2pt);
    \draw[blue] (9,1) -- (9,0);
\end{tikzpicture}
         \\
        \hline 
                                          \cellcolor{gray!10} (S7)
        &    
        \medskip 
    \begin{tikzpicture}[scale=0.6]
    \draw[step=1cm,gray!30,very thin] (0,0) grid (5,4);
    \draw[very thick] (0,0) -- (0,3) ;
    \draw[very thick] (0,3) -- (1,3) ;
    \node at (0.5,2.6) {${2}$};
    \draw[very thick] (1,3) -- (2,4);
    \draw[very thick] (2,4) -- (3,3);
    \node at (2.15,3) {[3,$\widetilde{4}$]};
    \draw[very thick] (3,3) -- (4,3);
    \node at (3.5,2.6) {${0}$};
    \draw[very thick, blue] (4,3) -- (5,3);
    \node[blue] at (4.5,2.6) {$\widetilde{1}$};
\end{tikzpicture} 
        &     
\begin{tikzpicture}[scale=0.6]
    \foreach \i in {1,2,...,10}{\fill[black] (\i, 0) circle (0pt);}
    \foreach \i in {1,2,...,10}{\fill[black] (\i, 1) circle (0pt);}
    \foreach \i in {1,2}{\fill[red] (\i, 1) circle (2pt);}
    \foreach \i in {3,4,5}{\fill[black] (\i, 0) circle (2pt);}
    \foreach \i in {3,4,5}{\fill[black] (\i, 1) circle (2pt);}
    \fill[black] (6,0) circle (2pt);
    \fill[black] (6,1) circle (2pt);
    \draw[black] (6,0) -- (2,1);
    \fill[black] (7,0) circle (2pt);
    \fill[black] (7,1) circle (2pt);
    \fill[black] (8,0) circle (2pt);
    \fill[black] (8,1) circle (2pt);
    \draw[black] (8,1) -- (7,0);
    \draw[black] (4,1) -- (8,0);
    \fill[black] (9,0) circle (2pt);
    \fill[black] (9,1) circle (2pt);
    \draw[black] (9,1) -- (9,0);
    \fill[blue] (10,0) circle (2pt);
    \fill[blue] (10,1) circle (2pt);
    \draw[blue] (10,1) -- (3,0);
\end{tikzpicture}
         \\
        \hline 
    \end{tabular}
\caption{A Laguerre history from $(0,3)$ to $(5,3)$ and its corresponding matching in $M_{5,3}^{(2)}$}
\label{fig:comparison_models}
\end{figure}

Next, we count the RHS of \eqref{eqn in Lem Laguerre history}. 

\begin{lem} \label{Lem_mathcing counting LUE}
For any nonnegative integers $\alpha,p $ and $j$, we have  
\begin{equation} \label{matching counting LUE}
\Big|M_{p,j}^{(\alpha)}(p)\Big|= p! \sum_{i=0}^{p} \binom{p}{i}  \binom{\alpha+j}{i}   \binom{p-i+j}{j}.
\end{equation}
\end{lem}
\begin{proof}
Let $l$ be the number of edges incident to vertices in $T_{-}$ or $B_{-}$, and $i$ be the number of edges incident to vertices in $T_{-}$. We then choose $i$ vertices in $B_{+}$ and form $i$ edges. These choices contribute the factor of $ \binom{\alpha+j}{i}\binom{p}{i}i!.$ 
Next, there are $l-i$ edges incident to the vertices in $B_{-}$. Choosing $l-i$ vertices in $B_{-}$ and $l-i$ vertices in $T_{+}$, and pairing them, contributes $\binom{j}{l-i}\binom{p}{l-i}(l-i)!.$ 
The remaining $p-l$ edges must connect vertices in $T_{+}$ with vertices in $B_{+}$. Since there are $p-(l-i)$ isolated vertices in $T_{+}$ and $p-i$ isolated vertices in $B_{+}$, this contributes $ \binom{p-l+i}{p-l}\binom{p-i}{p-l}(p-l)!.$ 
Combining the above contributions and summing over all admissible $i$ and $l$ gives that 
\begin{align} \label{matching counting LUE v1}
\Big|M_{p,j}^{(\alpha)}(p)\Big|= p! \sum_{l=0}^{p}\sum_{i=0}^{l} \binom{p}{i} \binom{p-i}{l-i} \binom{\alpha+j}{i}\binom{j}{l-i}.
\end{align}
Notice here that the RHS of \eqref{matching counting LUE v1} can be written as 
\begin{align*}
p!  \sum_{i=0}^{p}\sum_{l=i}^{p} \binom{p}{i} \binom{p-i}{l-i} \binom{\alpha+j}{i}\binom{j}{l-i} = p!  \sum_{i=0}^{p} \binom{p}{i}  \binom{\alpha+j}{i}  \sum_{k=0}^{p-i} \binom{p-i}{k} \binom{j}{k}.   
\end{align*}
Then by the classical Vandermonde identity (see e.g. \cite[p. 190]{Stanley12}), we obtain \eqref{matching counting LUE}.  
\end{proof}

We are now ready to show Proposition~\ref{prop:spetral moment of LUE}. 

\begin{proof}[Proof of Proposition~\ref{prop:spetral moment of LUE}]
It now follows immediately from Lemmas~\ref{Lem_Laguerre history bijection} and ~\ref{Lem_mathcing counting LUE} together with \eqref{eqn:mathfrakml=wtl}. 
\end{proof}

\subsection{Proof of Theorem \ref{Thm:Spectralmoment} (A)} \label{Subsection_Them moment A}

In this subsection, we show Theorem \ref{Thm:Spectralmoment} (A). As in the previous subsection, our starting point is the orthogonality relation of the $q$-Laguerre polynomials on a discrete $q$-lattice: 
\begin{equation}\label{eqn:orthogonality of little q Laguerre}
    \int_{0}^{1}p_{m}(x;q^{\alpha}|q)p_{n}(x;q^{\alpha}|q)w^{(\mathrm{qL})}(x;q)\,d_{q}x
    =h_{n}^{(\mathrm{qL})}\delta_{nm}, \qquad h_n^{ \rm (qL) }:=(1-q)q^{(\alpha+1)n}\frac{(q;q)_{n}}{(q^{\alpha+1};q)_{n}}\frac{(q;q)_{\infty}}{(q^{\alpha+1};q)_{\infty}},
\end{equation}
where $d_q$ is the Jackson $q$-integral \eqref{def of Jackson integral}. 
The spectral moments \eqref{def of qLUE moments} can then be written as 
\begin{equation}
  m_{N,p}^{(\mathrm{qL})}=\sum_{j=0}^{N-1}\mathfrak{m}_{p,j}^{(\mathrm{qL})}, \qquad    \mathfrak{m}_{p,j}^{(\mathrm{qL})}:=\int_{0}^{1}\frac{x^{p}p_{j}(x;q^{\alpha}|q)^{2}}{h_{j}^{(\mathrm{qL})}}w^{(\mathrm{qL})}(x;q)\,d_{q}x. 
\end{equation} 
To prove Theorem~\ref{Thm:Spectralmoment} (i), we introduce a new statistic that extends the bipartite matching interpretation used in the continuum case $q=1$ in the previous subsection. For this purpose, and in view of \eqref{eqn: continuum limit to Laguerre polynomial}, it is convenient to work with the rescaled (monic) little $q$-Laguerre polynomial
\begin{equation}
    \wh{p}_{n}^{(\alpha)}(x;q)
    =\frac{(-1)^{n}q^{-\binom{n}{2}}}{[\alpha+n+1]_{q}\cdots [\alpha+1]_{q}}p_n((1-q)x;q^{\alpha} \vert q).
\end{equation}
Let $\widehat{\mathfrak{m}}_{p,j}^{(\mathrm{qL})}$ denote the associated moments with respect to the rescaled polynomials. Then by the change of variables for Jackson $q$-integrals (see e.g. \cite[Eq. (1.38)]{BFO24}), we have 
\begin{equation} \label{m hat m relation}
\mathfrak{m}_{p,j}^{(\mathrm{qL})}=(1-q)^{p} \wh{\mathfrak{m}}_{p,j}^{(\mathrm{qL})}.
\end{equation}
To apply the Flajolet--Viennot theory, we first observe that, by \eqref{eqn:three term of qlaguerre}, the rescaled little $q$-Laguerre polynomial $\widehat{p}_{n}^{(\alpha)}(x;q)$ satisfies the three-term recurrence relation
\begin{equation} \label{three term for p hat}
    \widehat{p}_{n+1}^{(\alpha)}(x;q)= (x-b_n)\widehat{p}_{n}^{(\alpha)}(x;q) -\lambda_n \widehat{p}_{n-1}^{(\alpha)}(x;q),  
\end{equation}
where 
\begin{equation} \label{bn lambdan for q Laguerre}
b_n=q^{n}[n+\alpha+1]_{q}+q^{n+\alpha}[n]_{q}, \qquad \lambda_n= q^{2n+\alpha-1}[n]_{q}[n+\alpha]_{q}. 
\end{equation}
 
As previously noted, a fundamental way to interpret a $q$-natural number is as a weighted sum:  
instead of writing $n = 1 + 1 + \cdots + 1$, one has $[n]_{q} = 1 + q + \cdots + q^{n-1}$.  
In the same spirit, unlike in the previous subsection—where the bijection to bipartite matchings allowed us to count all matchings uniformly—in the $q$-deformed setting we should introduce an additional statistic on bipartite matchings that captures the structure of the corresponding Motzkin paths. Identifying the correct statistic is typically the crucial step in such problems, and in the present case it turns out that the appropriate choice is to record how many edges cross one another.

To be more precise, we introduce a statistic, which we call the \textit{crossing}, associated with a bipartite matching $M$ introduced in the previous subsection. A crossing is one of the following configurations:
\begin{enumerate}
    \item[(C1)] a pair of edges $(a,\widetilde{b})$ and $(c,\widetilde{d})$ with $a<c$ and $d<b$,  
    \smallskip
    \item[(C2)] a pair consisting of an edge $(a,\widetilde{b})$ and an isolated vertex $c$ with $c<a$,  
    \smallskip
    \item[(C3)] a pair consisting of an edge $(a,\widetilde{b})$ and an isolated vertex $\widetilde{d}$ with $d<b$.
\end{enumerate}
The conditions \textup{(C2)} and \textup{(C3)} may appear less intuitive at first, but they also admit a direct pictorial interpretation.  
By introducing an imaginary \emph{infinity} vertex—placed to the right of $M$ and connected to all isolated vertices—we see that these conditions count precisely the effective crossings created by adjoining this additional vertex; see Figure~\ref{Fig_crossing def}.

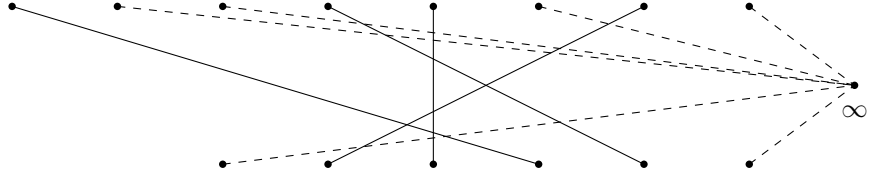
\begin{figure}[h]
    \centering
    \begin{tikzpicture}[scale=0.7]
        \foreach \i in {1,2,...,6}{\fill[black] (2*\i+2,0) circle (2pt);}
        \foreach \i in {1,2,...,6}{\fill[black] (2*\i+2,3) circle (2pt);}
        \fill[black] (0,3) circle (2pt);
        \fill[black] (2,3) circle (2pt);
        \fill[black] (16,1.5) circle (2pt);
        \draw (6,0) -- (12, 3);
        \draw (10,0) -- (0, 3);
        \draw (12,0) -- (6, 3);
        \draw (8,0) -- (8, 3);
        \draw[dashed] (2,3) -- (16,1.5);
        \draw[dashed] (4,3) -- (16,1.5);
        \draw[dashed] (10,3) -- (16,1.5);
        \draw[dashed] (14,3) -- (16,1.5);
        \draw[dashed] (4,0) -- (16,1.5);
        \draw[dashed] (14,0) -- (16,1.5);
        \node at (16,1) {$\infty$};
    \end{tikzpicture}
    \caption{A matching $M$ in $M_{2,4}^{(2)}$ with $\Cr(M)=17$}
    \label{Fig_crossing def}
\end{figure}

Recall that $M_{p,j}^{(\alpha)}(l)$ denotes the collection of matchings in $M_{p,j}^{(\alpha)}$ that contain exactly $l$ edges.  
A key ingredient in the proof of Theorem~\ref{Thm:Spectralmoment} (A) is the following bijection, which extends Lemma~\ref{Lem_Laguerre history bijection} by incorporating a weighted summation determined by the $q$-weights arising from the crossing statistic.

\begin{prop} \label{Prop_m har q cr stat}
Let $b_n$ and $\lambda_n$ be given by \eqref{bn lambdan for q Laguerre}. Then we have 
    \begin{equation}\label{eqn:mhat=qcr}
    \sum_{\omega\in\mathrm{Mot}_{p,j,j}}\mathrm{wt}_{b,\lambda}(\omega) = \sum_{M\in M_{p,j}^{(\alpha)}(p)}q^{\Cr(M)}.
    \end{equation}
Here, $\Cr(M)$ is the total number of crossings in the matching $M$.
\end{prop}

Observe that when $q=1$, the RHS of \eqref{eqn:mhat=qcr} simply counts the total number of elements in $M_{p,j}^{(\alpha)}(p)$, and thus recovers the RHS of \eqref{eqn in Lem Laguerre history}. In addition, it is evident from the RHS of \eqref{eqn:mhat=qcr} that the expression is an increasing function of $q$.

\begin{proof}[Proof of Proposition~\ref{Prop_m har q cr stat}]
    By the bijection between Motzkin paths and bipartite matchings established in the proof of Lemma~\ref{Lem_Laguerre history bijection}, it suffices to show that $\mathrm{wt}_{b,\lambda}(\omega)=\sum^* q^{\Cr(M)}$ where the RHS runs over all matchings induced by every possible Laguerre history on the Motzkin path $\omega$. More precisely, the various admissible labellings arising in the Laguerre histories are all counted in the sum on the RHS.
    
   Recall that in a Laguerre history consisting of $p$ steps, each East step contributes one edge, each North–East step contributes no edge, and each South–East step contributes two edges.  Since the Motzkin path starts and ends at the same height, the total number of edges in the corresponding bipartite matching is exactly $p$. Moreover, in the construction of a bipartite matching from a Laguerre history, the height of a step encodes the number of isolated vertices. More precisely, when a step begins at height $k$, there are $\alpha + k$ isolated vertices on the top row and $k$ isolated vertices on the bottom row; see Figure~\ref{fig:comparison_models}.

To establish \eqref{eqn:mhat=qcr}, it remains to compute the total contribution of the crossings in the associated bipartite matching.  
We perform this counting in a dynamical manner.  
Recall that the weight of a Motzkin path is given by the product of the weights of its individual steps: an East step and a South–East step at height $k$ carry weights $b_{k}$ and $\lambda_{k}$, respectively, while a North–East step has weight $1$.  
Here $b_{k}$ and $\lambda_{k}$ are defined in \eqref{bn lambdan for q Laguerre}.  
Thus, in this dynamical viewpoint, it suffices to verify that the contribution of the crossing statistic is exactly $b_k$ for each East step at height $k$, equals $\lambda_{k}$ for each South–East step, and is $1$ for each North–East step.

We begin by analysing the contribution of an East step.
    \begin{itemize}
        \item (East step at height $k$ with label $0$) In this case, the newly added vertical edge creates
    \begin{itemize}
        \item $\alpha + k$ crossings with isolated vertices in the top row, and
        \item $k$ crossings with isolated vertices in the bottom row.
    \end{itemize}
    Note that in this case the new edge cannot create any crossings with previously existing edges.
        \smallskip 
        \item (East step at height $k$ with label $m$ and $1 \le m \le k+\alpha$) In this case, note that any potential crossing between the newly added edge and an already existing edge has, in fact, been accounted for at the moment the corresponding vertex ceased to be isolated.  
    More precisely, when a vertex was isolated at some earlier step but later became matched, the crossing between that vertex and the infinity vertex was already counted; this corresponds exactly to the crossing that would otherwise appear here.  
    Thus, only the crossings involving isolated vertices need to be considered.  
    Consequently, the newly added edge creates
    \begin{itemize}
        \item $m - 1$ crossings with isolated vertices in the top row, and
        \item $k$ crossings with isolated vertices in the bottom row.
    \end{itemize}
    \smallskip
        \item (East step at height $k$ with label $\widetilde{l}$ and $1 \le l \le k$). For the same reason as in the preceding case, the newly added edge creates
    \begin{itemize}
        \item $\alpha + k$ crossings with isolated vertices in the top row, and
        \item $l - 1$ crossings with isolated vertices in the bottom row.
    \end{itemize}
    \end{itemize}
   Combining the above cases, one sees that an East step of height $k$ contributes
    \begin{equation*}
        q^{\alpha+2k}+\sum_{m=1}^{\alpha+k}q^{m-1+k}+\sum_{l=1}^{k}q^{\alpha+k+l -1}
        =q^{\alpha+k}[k]_{q}+q^{k}[\alpha+k+1]_{q}=b_k
    \end{equation*}
    to our statistic.
    
Next, we consider the contribution of a South–East step.  If we advance with a South–East step of height $k$ labelled $(m,\widetilde{l})$, then the two newly added edges create  $\alpha + k + m - 2$ and $k + l - 2$ crossings with isolated vertices, respectively, and an additional single crossing with each other.  Hence a South–East step of height $k$ contributes
    \begin{equation*}
        \sum_{m=1}^{\alpha+k}\sum_{l=1}^{k}q^{\alpha+2k+m+l-3}=q^{\alpha+2k-1}[\alpha+k]_{q}[k]_{q}=\lambda_{k}.
    \end{equation*}
Since North–East steps and the addition of isolated vertices do not contribute to the statistic, this completes the proof.
\end{proof}

The remaining task is to evaluate the RHS of \eqref{eqn:mhat=qcr}.  
To this end, we proceed by induction, beginning with the base case $j = 0$.

\begin{lem} \label{Lem_q cr evaluation j=0}
For any nonnegative integer $t$ with $0 \le t \le p$, we have
    \begin{equation} \label{crossing counting for j=0}
        \sum_{M\in M_{p,0}^{(\alpha)}(t)}q^{\Cr(M)}=\qbinom{p+\alpha}{t} \qbinom{p}{t}   [t]_{q}!.
    \end{equation}
\end{lem}
\begin{proof}
This lemma can be derived by a direct argument in elementary enumerative combinatorics. To connect our computation with the standard framework in \cite{Stanley12}, it is convenient to work with the inversion statistic for permutations. Recall that for a permutation $w = w_1 w_2 \cdots w_m \in S_m$, an \emph{inversion} is a pair $(i,j)$ with $i<j$ and $w_i > w_j$; the total number of inversions is denoted by $\textup{inv}(w)$.

We regard $\Cr(M)$ as an inversion-type statistic by encoding a matching $M \in M_{p,0}^{(\alpha)}(t)$ as follows.  
First, record which vertices in the top row are incident to an edge by a
binary word  $ \varepsilon = (\varepsilon_1,\dots,\varepsilon_{p+\alpha}) \in \{0,1\}^{p+\alpha},$ where $\varepsilon_i = 1$ if the $i$-th top vertex is matched and $\varepsilon_i = 0$ otherwise. Similarly, record the matched vertices in the
bottom row by a binary word $\delta = (\delta_1,\dots,\delta_p) \in \{0,1\}^{p},$ where $\delta_j = 1$ if the $j$-th bottom vertex is matched. Each of $\varepsilon$ and $\delta$ is a multiset permutation of $\{0^{m-t}, 1^t\}$. 

Next, order the matched vertices in each row from left to right, and let
$\sigma \in S_t$ be the permutation describing how the $k$-th matched top
vertex is connected to the $\sigma(k)$-th matched bottom vertex. This gives a
bijection between $M$ and $(\varepsilon,\delta,\sigma)$. 

With this encoding, our crossing statistic decomposes as
\begin{equation} \label{crossing inv decomp}
\Cr(M) = \textup{inv}(\delta) + \textup{inv}(\varepsilon) +  \textup{inv}(\sigma),
\end{equation} 
where $\textup{inv}(\cdot)$ denotes the inversion number of a word or permutation. Note that each term on the RHS of \eqref{crossing inv decomp} corresponds precisely to the contributions from (C1)--(C3) in the definition of the crossing statistic.
On the other hand, by \cite[Propositions~1.7.1 and ~1.3.13]{Stanley12}, the inversion generating functions are given by 
$$
\sum_{\varepsilon} q^{\textup{inv}(\varepsilon)} = \qbinom{p+\alpha}{t},
\qquad
\sum_{\delta} q^{\textup{inv}(\delta)} = \qbinom{p}{t}, \qquad \sum_{\sigma\in S_t} q^{\textup{inv}(\sigma)} = [t]_{q}!.
$$
Multiplying these three factors, we obtain \eqref{crossing counting for j=0}. 
\end{proof}

Next, we proceed to the case $j\geq1$. As a $q$-counterpart of Lemma~\ref{Lem_mathcing counting LUE}, we show the following. 

\begin{lem} \label{Lem_q cr evaluation}
For any nonnegative integers $\alpha,p$ and $j,$ we have 
\begin{equation} \label{eqn in Lema q cr eval}
\sum_{M\in M_{p,j}^{(\alpha)}(p)}q^{\Cr(M)}= [p]_q!  \sum_{i=0}^{p} q^{ (p-i)(\alpha-i)+pj }  \qbinom{p}{i} \qbinom{\alpha+j}{i} \qbinom{p-i+j}{j}.  
\end{equation}
\end{lem}
 
\begin{proof}
We claim indeed a slightly general one: for a nonneatgive integer $t$ with $0 \le t \le p$, we have 
\begin{align}
\begin{split}
\label{eqn:lem2.3}
\sum_{M\in M_{p,j}^{(\alpha)}(t)}q^{\Cr(M)} 
       &=\sum_{l=0}^{t}\sum_{i=0}^{l}  q ^{i(l-i)+(t-i)(j-i+\alpha)+(j-l+i)(t-l+i)}
       \\
       &\quad \times \frac{[t-l+i]_{q}![t-i]_{q}!}{[t-l]_{q}!}\qbinom{p}{t-i}\qbinom{p}{t-l+i}\qbinom{\alpha+j}{i}\qbinom{j}{l-i}. 
\end{split}
\end{align} 
By specialising \eqref{eqn:lem2.3} to $t=p$ and performing a change of indices, we have 
\begin{align*}
\sum_{M\in M_{p,j}^{(\alpha)}(p)}q^{\Cr(M)} &= [p]_q! \sum_{l=0}^{p}\sum_{i=0}^{l}q ^{i(l-i)+(p-i)(j-i+\alpha)+(j-l+i)(p-l+i)}    \qbinom{p}{i} \qbinom{p-i}{l-i}   \qbinom{\alpha+j}{i}\qbinom{j}{l-i}
\\
&= [p]_q! \sum_{i=0}^{p}\sum_{l=i}^{p} q ^{i(l-i)+(p-i)(j-i+\alpha)+(j-l+i)(p-l+i)} \qbinom{p}{i} \qbinom{p-i}{l-i} \qbinom{\alpha+j}{i}\qbinom{j}{l-i}
\\
&= [p]_q!  \sum_{i=0}^{p} q^{ (p-i)(\alpha-i)+pj }  \qbinom{p}{i} \qbinom{\alpha+j}{i}  \sum_{k=0}^{p-i} q ^{ (p-i-k)(j-k) }   \qbinom{p-i}{k}  \qbinom{j}{j-k}. 
\end{align*}
Then by applying the $q$-Vandermonde identity (see e.g. \cite[p. 190]{Stanley12}), we obtain \eqref{eqn in Lema q cr eval}. 

To compute the LHS of \eqref{eqn:lem2.3}, we further decompose the set of matchings in $M_{p,j}^{(\alpha)}(t)$ as follows. 
\begin{itemize}
 \item Let $l$, with $0 \le l \le \min\{t,j\}$, denote the number of edges incident to vertices in $T_-$ or $B_-$. Equivalently, there are $t-l$ edges connecting $T_+$ and $B_+$.
We denote by $M_{p,j}^{(\alpha)}(t;l)$ the subset of $M_{p,j}^{(\alpha)}(t)$ consisting of matchings with exactly $l$ such edges.
\smallskip 
    \item
    Let $i$, with $0 \le i \le l$, denote the number of edges incident to $T_-$.
    Then $l-i$ edges are incident to $B_-$. 
\end{itemize}
By definition, we have 
\begin{equation}  \label{eqn:lem2.3 decomp}
    \sum_{M\in M_{p,j}^{(\alpha)}(t)}q^{\Cr(M)}
    =\sum_{l=0}^{t}\sum_{i=0}^l \sum_* q^{\Cr(M)}, 
\end{equation} 
where the inner sum $\sum_*$ runs over all matchings $M \in M_{p,j}^{(\alpha)}(t;l)$ such that exactly $i$ edges are incident to $T_-$.
Therefore, to prove \eqref{eqn:lem2.3}, it suffices to show that
\begin{equation}\label{eqn:mpj0tl}
   \sum_* q^{\Cr(M)}   = q ^{i(l-i)+(t-i)(j-i+\alpha)+(j-l+i)(t-l+i)}\frac{[t-l+i]_{q}![t-i]_{q}!}{[t-l]_{q}!}\qbinom{p}{t-i}\qbinom{p}{t-l+i}\qbinom{\alpha+j}{i}\qbinom{j}{l-i}.
\end{equation} 
 
Given a matching in $M_{p,j}^{(\alpha)}(t;l)$, we consider the submatching
consisting of those edges that connect $T_+$ and $B_+$, together with their incident vertices.
This submatching is an element of $M_{t-l,0}^{(0)}(t-l)$ for some $0 \le l \le t$.
Conversely, starting from a matching in $M_{t-l,0}^{(0)}(t-l)$, we construct a matching in $M_{p,j}^{(\alpha)}(t;l)$ as follows.
In the top graphical representation of Figure~\ref{fig:construction_steps} below, the filled vertices and edges represent a matching in $M_{t-l,0}^{(0)}(t-l)$, where $t-l=2$.
By contrast, all vertices---both filled and empty---in the same figure correspond to those in $M_{p,j}^{(\alpha)}(t;l)$, with parameters $\alpha=1$, $j=3$, and $p=5$.
To obtain a matching in $M_{p,j}^{(\alpha)}(t;l)$ from the given one in
$M_{t-l,0}^{(0)}(t-l)$, we therefore need to add
\begin{itemize}
     \item[(i)] $t-(t-l)=l$ additional edges that are connected to $T_-$ or $B_-$, 
     \smallskip 
    \item[(ii)] $\alpha+p+j-(t-l)$ vertices in the top row and $p+j-(t-l)$ vertices in the bottom row.
\end{itemize}
We perform this construction through the steps described below. After completing all steps, the resulting collection of vertices and edges forms a matching, illustrated in the bottom graphical representation of Figure~\ref{fig:construction_steps}. 
\begin{itemize}
    \item \textbf{(Step 1)} Recall that $i$ denotes the number of edges incident to vertices in $T_-$ at the end of the process.
    In the first step, we add $i$ such edges together with their incident
    vertices.
    \smallskip 
    \item \textbf{(Step 2)} Next, we add the remaining $l-i$ edges incident to vertices in $B_-$.
    After this step, all $t$ edges have been added.
    Then it remains to add $\alpha+p+j-t$ isolated vertices in the top row and $p+j-t$ isolated vertices in the bottom row.
    \smallskip 
    \item \textbf{(Step 3)} We next add isolated vertices in $T_+$ and $B_+$.
    Specifically, we add $p-(t-i)$ vertices in $T_+$ and
    $p-(t-l+i)$ isolated vertices in $B_+$.
    \smallskip 
    \item \textbf{(Step 4)}  Finally, we add the remaining $\alpha+j-i$ isolated vertices in $T_-$ and $j-(l-i)$ isolated vertices in $B_-$.
\end{itemize}

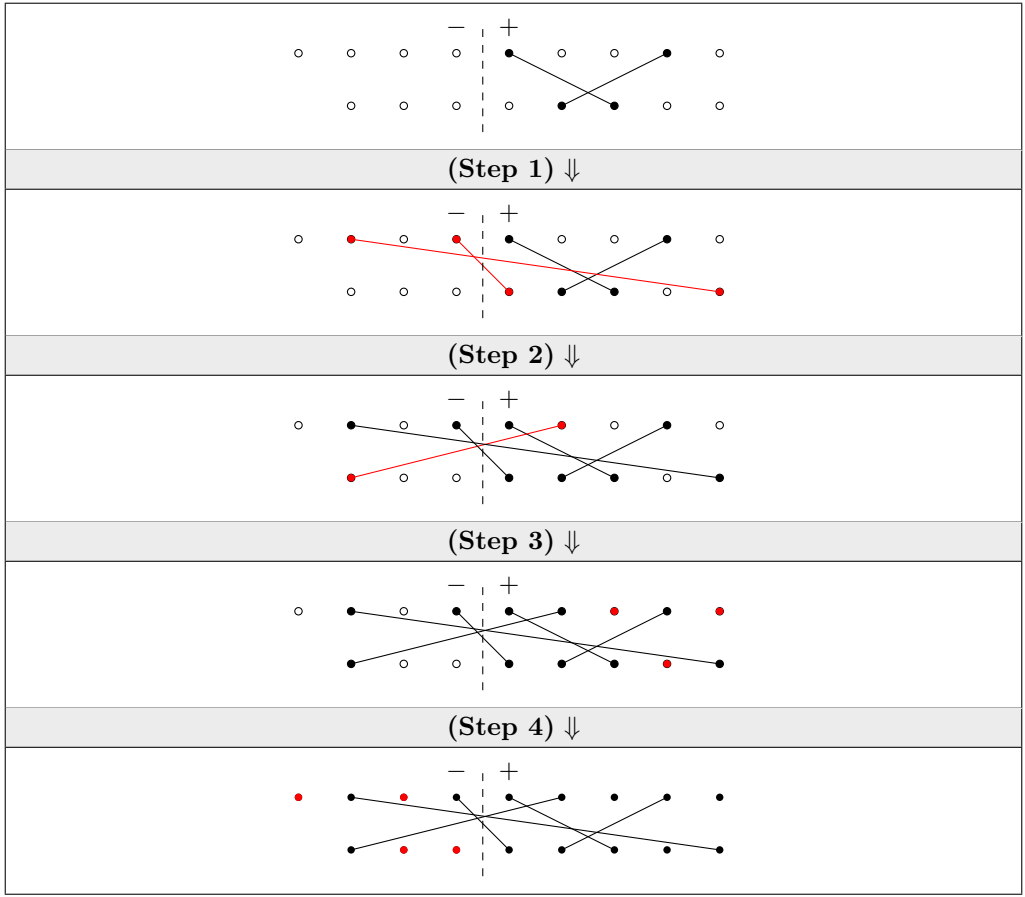
\begin{figure}[h]
\centering
\renewcommand{\arraystretch}{1.2}
\setlength{\tabcolsep}{0pt}

\begin{tabular}{|C{13.5cm}|}
\hline
\smallskip  
      \begin{tikzpicture}[scale=0.7]
            \foreach \i in {-1,0,1,2,3,4,5,6}{\draw[black] (\i,0) circle (2pt);}
            \foreach \i in {-2,-1,0,1,2,3,4,5,6}{\draw[black] (\i,1) circle (2pt);}
            \fill[black] (3,0) circle (2pt);
            \fill[black] (4,0) circle (2pt);
            \fill[black] (2,1) circle (2pt);
            \fill[black] (5,1) circle (2pt);
            \draw[black] (5,1) -- (3,0);
            \draw[black] (2,1) -- (4,0);
               \draw[dashed] (1.5,-0.5) -- (1.5,1.5);
            \node at (1,1.5) {$-$};
            \node at (2,1.5) {$+$};
        \end{tikzpicture}
\smallskip 
\\ \hline

{\cellcolor{gray!15}\centering
\textbf{(Step 1)}\ $\Downarrow$}
\\ \hline

\smallskip 
  \begin{tikzpicture}[scale=0.7]
            \foreach \i in {-1,0,1,2,3,4,5,6}{\draw[black] (\i,0) circle (2pt);}
            \foreach \i in {-2,-1,0,1,2,3,4,5,6}{\draw[black] (\i,1) circle (2pt);}
            \fill[black] (3,0) circle (2pt);
            \fill[black] (4,0) circle (2pt);
            \fill[black] (2,1) circle (2pt);
            \fill[black] (5,1) circle (2pt);
            
            \fill[red] (2,0) circle (2pt); 
             \fill[red] (1,1) circle (2pt); 
            \fill[red] (6,0) circle (2pt);
            \fill[red] (-1,1) circle (2pt);
            
            \draw[black] (5,1) -- (3,0);
            \draw[black] (2,1) -- (4,0);
            
            \draw[red] (2,0) -- (1,1);
            \draw[red] (6,0) -- (-1,1);
               \draw[dashed] (1.5,-0.5) -- (1.5,1.5);
            \node at (1,1.5) {$-$};
            \node at (2,1.5) {$+$};
        \end{tikzpicture}
        \smallskip 
\\ \hline

{\cellcolor{gray!15}\centering
\textbf{(Step 2)}\ $\Downarrow$}
\\ \hline
\smallskip 
   \begin{tikzpicture}[scale=0.7]
            \foreach \i in {-1,0,1,2,3,4,5,6}{\draw[black] (\i,0) circle (2pt);}
            \foreach \i in {-2,-1,0,1,2,3,4,5,6}{\draw[black] (\i,1) circle (2pt);}
            \fill[black] (3,0) circle (2pt);
            \fill[black] (4,0) circle (2pt);
            \fill[black] (2,1) circle (2pt);
            \fill[black] (5,1) circle (2pt);
            
            \fill[red] (-1,0) circle (2pt);
            \fill[red] (3,1) circle (2pt);
            
            \fill[black] (2,0) circle (2pt);
            \fill[black] (1,1) circle (2pt);
            \fill[black] (6,0) circle (2pt);
            \fill[black] (-1,1) circle (2pt);
            
            \draw[black] (5,1) -- (3,0);
            \draw[black] (2,1) -- (4,0);
            
            \draw[red] (-1,0) -- (3,1);
            
            \draw[black] (2,0) -- (1,1);
            \draw[black] (6,0) -- (-1,1);
            
            \draw[dashed] (1.5,-0.5) -- (1.5,1.5);
            \node at (1,1.5) {$-$};
            \node at (2,1.5) {$+$}; 
        \end{tikzpicture}
        \smallskip 
\\ \hline

{\cellcolor{gray!15}\centering
\textbf{(Step 3)}\ $\Downarrow$}
\\ \hline
\smallskip 
      \begin{tikzpicture}[scale=0.7]
            \foreach \i in {-1,0,1,2,3,4,5,6}{\draw[black] (\i,0) circle (2pt);}
            \foreach \i in {-2,-1,0,1,2,3,4,5,6}{\draw[black] (\i,1) circle (2pt);}
            \fill[black] (3,0) circle (2pt);
            \fill[black] (4,0) circle (2pt);
            \fill[black] (2,1) circle (2pt);
            \fill[black] (5,1) circle (2pt);
            \fill[black] (-1,0) circle (2pt);
            \fill[black] (3,1) circle (2pt);
            \fill[black] (2,0) circle (2pt);
            \fill[black] (1,1) circle (2pt);
            \fill[black] (6,0) circle (2pt);
            \fill[black] (-1,1) circle (2pt);
            
            \fill[red] (5,0) circle (2pt);
            \fill[red] (4,1) circle (2pt);
            \fill[red] (6,1) circle (2pt);
            
            \draw[black] (5,1) -- (3,0);
            \draw[black] (2,1) -- (4,0);
            \draw[black] (-1,0) -- (3,1);
            \draw[black] (2,0) -- (1,1);
            \draw[black] (6,0) -- (-1,1);

            \draw[dashed] (1.5,-0.5) -- (1.5,1.5);
            \node at (1,1.5) {$-$};
            \node at (2,1.5) {$+$};  
    \end{tikzpicture}
        \smallskip 
\\ \hline
 
{\cellcolor{gray!15}\centering
\textbf{(Step 4)}\ $\Downarrow$}
\\ \hline
\smallskip 
     \begin{tikzpicture}[scale=0.7]
       
            \foreach \i in {-1,0,1,2,3,4,5,6}{\fill[black] (\i+10,0) circle (2pt);}
            \foreach \i in {-2,-1,0,1,2,3,4,5,6}{\fill[black] (\i+10,1) circle (2pt);}
            \fill[red] (8,1) circle (2pt);
            \fill[red] (10,1) circle (2pt);
            \fill[red] (10,0) circle (2pt);
            \fill[red] (11,0) circle (2pt);
            
            \draw[black] (15,1) -- (13,0);
            \draw[black] (12,1) -- (14,0);
            \draw[black] (9,0) -- (13,1);
            \draw[black] (16,0) -- (9,1);
            \draw[black] (12,0) --(11,1);

            \draw[dashed] (11.5,-0.5) -- (11.5,1.5);
            \node at (11,1.5) {$-$};
            \node at (12,1.5) {$+$};
    \end{tikzpicture}
        \smallskip 
\\ \hline

\end{tabular}

\caption{A step-by-step illustration of the construction.}
\label{fig:construction_steps}
\end{figure}

At each step, we keep track of the combinatorial factors contributing to the resulting summation defining our statistic. 
Note that by Lemma~\ref{Lem_q cr evaluation j=0}, we have 
   \begin{equation} \label{factor in Step0}
        \sum_{M\in M_{t-l,0}^{(0)}(t-l)}q^{\Cr(M)}=  [t-l]_{q}!.
    \end{equation} 
    
\smallskip 

\noindent \textbf{(Step 1)} We add one vertex at the leftmost position of the top row and one vertex in the bottom row, and then introduce an edge connecting these two vertices.
Since there are $t-l$ existing vertices in the bottom row, there are $t-l+1$ possible positions for the newly added bottom vertex. 
If the added vertex in $B_+$ is the $n$-th vertex from the left in the bottom row, then the added edge contributes $n-1$ to the crossing number.
Consequently, summing over all admissible choices, the contribution of the first added edge to our statistic is $[t-l+1]_q$.

To add the second edge, we again insert a new leftmost vertex in the top row and a new vertex in the bottom row.
Now there are $t-l+2$ possible positions for the newly added bottom vertex, and by the same reasoning, the admissible choices of the second edge contribute $[t-l+2]_q$.
Repeating this process $i$ times, this step produces the factor
        \begin{equation} \label{factor in Step1} 
         \frac{[t-l+i]_{q}!}{[t-l]_{q}!}.
        \end{equation}
\noindent \textbf{(Step 2)} Using the same argument as in the previous step, we proceed as follows.
Since $i$ edges were added, each newly added edge in the present step creates $i$ additional crossings, contributing a factor of $q^{i}$.
As this operation is performed $l-i$ times, the total contribution from crossings is $q^{i(l-i)}$.

Moreover, as in \textbf{(Step 1)}, the admissible positions of the added vertices in the bottom row contribute the factor
    \begin{equation} \label{factor in Step2}
        q^{i(l-i)}\frac{[t-i]_{q}!}{[t-l]_{q}!}.
    \end{equation} 

\noindent \textbf{(Step 3)}  
Since there are $t-i$ existing vertices in $T_+$, we add $p-(t-i)$ isolated vertices inductively.
By the definition of the crossing number, the contribution of a newly added vertex to the crossing count equals the number of vertices incident to edges that are positioned to the right of the newly added one.
Hence, by the standard combinatorial interpretation of $q$-binomial coefficients, this procedure contributes the factor
        \begin{equation*}
            \qbinom{(t-i)+(p-t+i)}{t-i}=\qbinom{p}{t-i}.
        \end{equation*}
Similarly, we add $p-(t-l+i)$ isolated vertices in $B_+$. By the same reasoning, this step contributes the factor
        \begin{equation*}
            \qbinom{p}{t-l+i}.
        \end{equation*} 
Therefore, the total contribution of this step is
\begin{equation}  \label{factor in Step3}
\qbinom{p}{t-i}  \qbinom{p}{t-l+i}.
\end{equation} 

\noindent \textbf{(Step 4)}  
Since there are $i$ existing vertices in $T_-$, we add $\alpha+j-i$ isolated vertices in $T_-$. Each of these newly added vertices creates crossings with the $t-i$ edges incident to vertices in $T_+$.
Moreover, there are already $i$ connected vertices in $T_-$.
By the same reasoning as in \textbf{(Step 3)}, this yields the factor
        \begin{equation*}
            q^{(t-i)(\alpha+j-i)}\qbinom{i+(\alpha+j-i)}{i}=q^{(t-i)(\alpha+j-i)}\qbinom{\alpha+j}{i}.
        \end{equation*}
Similarly, since there are $l-i$ existing vertices in $B_-$, we add $j-(l-i)$ isolated vertices in $B_-$. These vertices create crossings with the $t-l+i$ edges incident to vertices in $B_+$.
Therefore, this contributes the factor
        \begin{equation*}
            q^{(t-l+i)(j-l+i)}\qbinom{(l-i)+(j-l+i)}{j-l+i}=q^{(t-l+i)(j-l+i)}\qbinom{j}{l-i}.
        \end{equation*}   
Combining these, the total contribution of this step is
\begin{equation}  \label{factor in Step4}
q^{(t-i)(\alpha+j-i)+(t-l+i)(j-l+i)}\qbinom{\alpha+j}{i} \qbinom{j}{l-i}. 
\end{equation}
Now by combining all of \eqref{factor in Step0}, \eqref{factor in Step1}, \eqref{factor in Step2}, \eqref{factor in Step3} and \eqref{factor in Step4}, we obtain the desired result \eqref{eqn:mpj0tl}. 
\end{proof}

\begin{proof}[Proof of Theorem \ref{Thm:Spectralmoment} (A)]
The desired evaluation \eqref{eqn:m} follows by combining \eqref{mathfrak m in terms of Motzkin}, \eqref{m hat m relation}, Proposition~\ref{Prop_m har q cr stat} and Lemma~\ref{Lem_q cr evaluation}.  
\end{proof}
 
\subsection{Proof of Theorem~\ref{Thm:Spectralmoment} (B)} \label{Subsection_Them moment B}

In this subsection, we show Theorem~\ref{Thm:Spectralmoment} (B). Throughout this subsection we assume $p \ge 1$, since the case $p=0$ is trivial: by definition $m_{N,0}^{\rm (qL)} = N$.

Recall that $\alpha=cN+d$. Then we can rewrite \eqref{eqn:m} as 
\begin{equation}\label{eqn:rewrite mNp}
  {m}_{N,p}^{(\mathrm{qL})}\Big|_{\alpha=cN+d}
    =(1-q)^{p}[p]_{q}!\sum_{i=0}^{p}q^{-(p-i)i}\qbinom{p}{i}\sum_{j=0}^{N-1}f_{i}^{(c,d)}(j),
\end{equation}
where 
\begin{equation}
    f_{i}^{(c,d)}(j)=q^{(cN+d)(p-i)+pj}\qbinom{cN+d+j}{i}\qbinom{p-i+j}{p-i}. 
\end{equation}
We now analyse the large-$N$ asymptotics of each factor in the RHS of \eqref{eqn:rewrite mNp}.

\begin{lem}\label{lem:large N of qfactorials}
    Let $q$ be scaled as in \eqref{eqn:q scale}. As $N\rightarrow\infty$, we have
    \begin{equation}\label{eqn:lemma 3.1}
        q^{-(p-i)i}(1-q)^{p}[p]_{q}!\qbinom{p}{i}
        =\Big(\frac{\lambda}{N}\Big)^{p}p!\binom{p}{i}\Big(1+\frac{-p^{2}-p-2i^{2}+2ip}{4}\frac{\lambda}{N}+O(N^{-2})\Big).
    \end{equation}
\end{lem}

\begin{proof}
Note that as $N \to \infty$, we have 
    \begin{equation*}
        q^{-(p-i)i}=1+i(p-i)  \frac{ \lambda}{N}+O(N^{-2}).
    \end{equation*}
Also by the definition of $q$-factorial, we have 
\begin{align*}
 (1-q)^{p}[p]_{q}!=\prod_{m=1}^{p}(1-q^{m})=\Big(\frac{\lambda}{N}\Big)^{p}p!\Big(1-\frac{p(p+1)}{4}\frac{\lambda}{N}+O(N^{-2})\Big). 
\end{align*} 
Similarly, we have 
\begin{align*}
\qbinom{p}{i}=\prod_{m=1}^{i}\frac{1-q^{p-i+m}}{1-q^{m}} =\binom{p}{i}\Big(1-\frac{i(p-i)}{2}\frac{\lambda}{N}+O(N^{-2})\Big).
\end{align*}
Combining the above, the desired asymptotic formula \eqref{eqn:lemma 3.1} follows. 
\end{proof}

We now examine the asymptotic behaviour of the term $\sum_{j=0}^{N-1}f_{l,i}^{(c,d)}(j)$.

\begin{lem}
    Let $q$ be scaled as in \eqref{eqn:q scale}. Then, as $N\rightarrow\infty$, we have
    \begin{equation}\label{eqn:intf(Nt) asymptotic}
        \int_{0}^{1}f_{i}(Nt)\,dt
        =\Big(\frac{N}{\lambda}\Big)^{p}\frac{1}{i!(p-i)!}\Big(\mathcal{A}_{i,0}+(\mathcal{A}_{i,1}^{(1)}+\mathcal{A}_{i,1}^{(2)}+\mathcal{A}_{i,1}^{(3)})\frac{\lambda}{N}+O(N^{-2})\Big), 
    \end{equation}
    where
        \begin{align} \label{def of mathcal Ai0}
        \mathcal{A}_{i,0}
        &=\mathsf{s}^{c(p-i)} \int_{0}^{1} \mathsf{s}^{pt}(1-\mathsf{s}^{t})^{p-i}(1-\mathsf{s}^{c+t})^{i}\,dt, 
        \\
        \mathcal{A}_{i,1}^{(1)} &=\mathsf{s}^{c(p-i+1)}\frac{i(2d-i+1)}{2}\int_{0}^{1}\mathsf{s}^{(p+1)t}(1-\mathsf{s}^{t})^{p-i}(1-\mathsf{s}^{c+t})^{i-1}\,dt,
        \\
        \mathcal{A}_{i,1}^{(2)} &=\mathsf{s}^{c(p-i)} \frac{(p-i)(p-i+1)}{2}\int_{0}^{1} \mathsf{s}^{(p+1)t}(1-\mathsf{s}^{t})^{p-i-1}(1-\mathsf{s}^{c+t})^{i}\,dt, 
        \\
        \mathcal{A}_{i,1}^{(3)} &=\frac{4di+2i^{2}+p-4dp-2ip+p^2}{4}\mathcal{A}_{i,0}. 
    \end{align}
Here, we recall that $\mathsf{s}=e^{-\lambda}.$   
\end{lem}

\begin{proof}
    First, we find the asymptotic behaviour of $f_{i}^{(c,d)}(Nt)$. By definition, we have
    \begin{equation*}
        \qbinom{(c+t)N+d}{i}=\prod_{m=1}^{i}\frac{1-q^{(c+t)N+d-i+m}}{1-q^{m}}, \qquad  \qbinom{p-i+tN}{p-i}=\prod_{m=1}^{p-i}\frac{1-q^{tN+m}}{1-q^{m}}. 
    \end{equation*}
    Note that as $N \to \infty,$ we have 
    \begin{equation*}
    \begin{split}
        \prod_{m=1}^{i}(1-q^{(c+t)N+d-i+m})
        &=(1-\mathsf{s}^{c+t})^{i}\prod_{m=1}^{i}\Big(1+\frac{\mathsf{s}^{c+t}}{1-\mathsf{s}^{c+t}}(d-i+m)\frac{\lambda}{N}+O(N^{-2})\Big)\\
        &=(1-\mathsf{s}^{c+t})^{i}\Big(1+\frac{\mathsf{s}^{c+t}}{1-\mathsf{s}^{c+t} }\frac{i(2d-i+1)}{2}\frac{\lambda}{N}+O(N^{-2})\Big)
    \end{split}
    \end{equation*}
    and
    \begin{equation*}
        \prod_{m=1}^{i}\frac{1}{1-q^{m}}=\prod_{m=1}^{i}\frac{N}{m\lambda}\Big(1+\frac{m}{2}\frac{\lambda}{N}+O(N^{-2})\Big)=\Big(\frac{N}{\lambda}\Big)^{i}\frac{1}{i!}\Big(1+\frac{i(i+1)}{4}\frac{\lambda}{N}+O(N^{-2})\Big).
    \end{equation*}
    Thus it follows that
    \begin{equation*}
        \qbinom{(c+t)N+d}{i}=\Big(\frac{N}{\lambda}\Big)^{i} \frac{(1-\mathsf{s}^{c+t})^{i}}{i!}\Big(1+\Big(\frac{\mathsf{s}^{c+t}}{1-\mathsf{s}^{c+t}}\frac{i(2d-i+1)}{2}+\frac{i(i+1)}{4}\Big)\frac{\lambda}{N}+O(N^{-2})\Big).
    \end{equation*} 
    Similarly, we obtain
    \begin{equation*}
        \qbinom{Nt+p-i}{p-i}=\Big(\frac{N}{\lambda}\Big)^{p-i}\frac{(1-\mathsf{s}^{t})^{p-i}}{(p-i)!}\Big(1+\Big(\frac{1+\mathsf{s}^{t} }{1-\mathsf{s}^{t}} \frac{(p-i)(p-i+1)}{4}\Big)\frac{\lambda}{N}+O(N^{-2})\Big).
    \end{equation*}
    Combining these asymptotics with
    \begin{equation*}
        q^{(cN+d)(p-i)+ptN}
        =\mathsf{s}^{pt+c(p-i)}\Big(1-d(p-i)\frac{\lambda}{N}+O(N^{-2})\Big),
    \end{equation*}
    we obtain  
    \begin{align} \label{eqn:f(Nt) asymptotic}
    \begin{split}
      &\quad f_{i}^{(c,d)}(Nt) \Big(\frac{\lambda}{N}\Big)^{p}\frac{i!(p-i)!}{(1-\mathsf{s}^{c+t})^{i}(1-\mathsf{s}^{t})^{p-i}\mathsf{s}^{pt+c(p-i)}} -1
        \\
        &=\bigg(\frac{\mathsf{s}^{c+t}}{1-\mathsf{s}^{c+t}}\frac{i(2d-i+1)}{2}+\frac{\mathsf{s}^{t}}{1-\mathsf{s}^{t}}\frac{(p-i)(p-i+1)}{2}+\frac{4di+2i^{2}+p-4dp-2ip+p^{2}}{4}\bigg)\frac{\lambda}{N}  +O(N^{-2}). 
    \end{split}
    \end{align} 
    Now, by integrating both sides with respect to $t$ from $0$ to $1$, we complete the proof.
\end{proof}

\begin{lem}\label{lem:asymptotic of sum f_{l,i}}
    As $N\rightarrow\infty$, we have
    \begin{equation}
        \sum_{j=0}^{N-1}f_{l,i}^{(c,d)}(j)
        =\Big(\frac{N}{\lambda}\Big)^{p+1}\frac{1}{i!(p-i)!}\Big(\mathcal {B}_{i,0}+\mathcal{B}_{i,1}\frac{\lambda}{N}+O(N^{-2})\Big), 
    \end{equation}
    where 
     \begin{align} \label{def of mathcal Bi01}
        \mathcal{B}_{i,0}=\lambda\mathcal{A}_{i,0},\qquad 
\mathcal{B}_{i,1}=\lambda(\mathcal{A}_{i,1}^{(1)}+\mathcal{A}_{i,1}^{(2)}+\mathcal{A}_{i,1}^{(3)}) -\frac{(1-\mathsf{s}^{c+1})^{i}(1-\mathsf{s})^{p-i}\mathsf{s}^{p+c(p-i)} }{2}+\frac{(1-\mathsf{s}^{c})^{p}}{2}\delta_{i,p}.
    \end{align}  
\end{lem}

\begin{proof}
    Applying the Euler–Maclaurin formula (see e.g. \cite[Section 2.10]{NIST}), we have 
    \begin{equation}\label{eqn:Euler Maclaurin}
        \sum_{j=0}^{N-1}f_{i}^{(c,d)}(j)\sim\int_{0}^{N}f_{i}^{(c,d)}(x)\,dx-\frac{f_{i}^{(c,d)}(N)-f_{i}^{(c,d)}(0)}{2}
    \end{equation}
    as $N\rightarrow\infty$. Substituting $t=0$ and $t=1$ into \eqref{eqn:f(Nt) asymptotic}, it follows that 
    \begin{align}
    \label{eqn:f(0) asymptotic}
        f_{i}^{(c,d)}(0)&=\Big(\frac{N}{\lambda}\Big)^{i}\frac{\mathsf{s}^{c(p-i)}(1-\mathsf{s}^{c})^{i}}{i!}(1+O(N^{-1})),
        \\
        \label{eqn:f(N) asymptotic}
        f_{i}^{(c,d)}(N)&=\Big(\frac{N}{\lambda}\Big)^{p}\frac{(1-\mathsf{s}^{c+1})^{i}(1-\mathsf{s})^{p-i}\mathsf{s}^{p+c(p-i)}}{i!(p-i)!}(1+O(N^{-1})).
    \end{align} 
  Since $f_{i}^{(c,d)}(0)$ is of order $N^{i}$, only the term with $i=p$ 
contributes to the first subleading order. Combining 
\eqref{eqn:intf(Nt) asymptotic}, \eqref{eqn:f(N) asymptotic}, and 
\eqref{eqn:f(0) asymptotic} with \eqref{eqn:Euler Maclaurin} 
then completes the proof.
\end{proof}

We are ready to show Theorem~\ref{Thm:Spectralmoment} (B). 

\begin{proof}[Proof of Theorem~\ref{Thm:Spectralmoment} (B)]
By combining Lemmas \ref{lem:large N of qfactorials} and \ref{lem:asymptotic of sum f_{l,i}}  with \eqref{eqn:rewrite mNp}, we obtain 
    \begin{equation} \label{mNp interms of mathcal B}
        m_{N,p}^{(\mathrm{qL})}
        =\frac{N}{\lambda}\sum_{i=0}^{p}\binom{p}{i}^{2}\Big(\mathcal{B}_{i,0}+\Big(\mathcal{B}_{i,1}+\frac{-p^{2}-p-2i^2+2ip}{4}\mathcal{B}_{i,0}\Big)\frac{\lambda}{N}+O(N^{-2})\Big).
    \end{equation}
Then it follows from \eqref{def of mathcal Ai0} and \eqref{def of mathcal Bi01} that 
\begin{equation}  \label{def of leading order moment by integral}
\lim_{N \to \infty} \frac{ 1 }{N}m_{N,p}^{(\mathrm{qL})} =\int_{0}^{1} \mathsf{s}^{p t}\sum_{i=0}^{p}\Big(1-\mathsf{s}^{c+t}\Big)^{i}\Big(\mathsf{s}^{c}(1-\mathsf{s}^{t})\Big)^{p-i}\binom{p}{i}^{2}\,dt. 
\end{equation}
Now, in order to show the first assertion of Theorem~\ref{Thm:Spectralmoment} (B), it remains to show that the RHS of \eqref{def of leading order moment by integral} is same as the RHS of \eqref{def of leading order moment}. 

For this, note that by the change of variables $v=1-e^{- t\lambda}$, the RHS of \eqref{def of leading order moment by integral} is written as   
    \begin{align*}
       \frac{e^{-c\lambda p}}{\lambda} \int_{0}^{1-e^{-\lambda}}  (1-v)^{p-1} \sum_{i=0}^{p}\Big( e^{c\lambda}-1+v \Big)^{i}v^{p-i}\binom{p}{i}^{2}\,dv. 
    \end{align*}
    Note here that with $u=e^{c\lambda}-1,$ we have 
    \begin{align*}
        \sum_{i=0}^{p} ( u+v )^{i}v^{p-i}\binom{p}{i}^{2}& = \sum_{i=0}^p \sum_{k=0}^i  \binom{i}{k}\binom{p}{i}^{2} u^{i-k} v^{p-i+k}
        =\sum_{m=0}^p u^{p-m} \bigg( \sum_{r=0}^m \binom{p}{r}^2 \binom{p-r}{m-r} \bigg) v^m, 
    \end{align*}
    where we have used the index changes $r=p-i$ and $k=m-r$. Furthermore, by using 
    \begin{align*}
        \sum_{r=0}^m \binom{p}{r}^2 \binom{p-r}{m-r} &= \binom{p}{m} \sum_{r=0}^m  \binom{p}{r} \binom{m}{m-r} = \binom{p}{m} \binom{p+m}{m},  
    \end{align*}
    we obtain 
    \begin{align*}
        \mathcal{M}^{ \rm (qL) }_{p,0}=  \frac{e^{-c\lambda p}}{\lambda} \sum_{m=0}^p   \binom{p}{m} \binom{p+m}{m} (e^{c\lambda}-1)^{p-m} \int_{0}^{1-e^{-\lambda}} v^m  (1-v)^{p-1}   \,dv. 
    \end{align*}
    Then the desired formula, the RHS of \eqref{def of leading order moment} follows from the definition of the regularised incomplete beta function \eqref{def of beta ftn}. 
    
    Next, we show the second assertion of Theorem~\ref{Thm:Spectralmoment} (B). 
    Note that, when $c=0$, by \eqref{def of mathcal Bi01}, the summand in \eqref{mNp interms of mathcal B} is computed as 
    \begin{align*}
        &\quad \mathcal{B}_{i,1}+\frac{-p^2-p-2i^2+2ip}{4}\mathcal{B}_{i,0}
        \\
        &=\frac{2di+p^2-2ip+p}{2}\int_{0}^{1- \mathsf{s} }v^{p-1}(1-v)^{p}\,dv+d(i-p)\int_{0}^{1-\mathsf{s} }v^{p}(1-v)^{p-1}\,dv-\frac{\mathsf{s}^{p}(1-\mathsf{s})^{p} }{2}.
    \end{align*}
    Then by using \eqref{def of beta ftn} and the recursion formula of the incomplete beta function \cite[Eq.~(8.17.18)]{NIST} 
    \begin{equation*}
        I_{1- \mathsf{s} }(p,p+1)=I_{1- \mathsf{s} }(p+1,p)+ \binom{2p}{p} \mathsf{s}^{p}(1-\mathsf{s})^{p} ,
    \end{equation*}
    we obtain that 
    \begin{align*}
        & \quad \sum_{i=0}^{p}\binom{p}{i}^{2} \Big(\mathcal{B}_{i,1}+\frac{-p^{2}-5p-2i^2+2ip}{4}\mathcal{B}_{i,0}\Big) 
        \\
        &=\sum_{i=0}^{p}\binom{p}{i}^{2}\Big(\frac{p^{2}+p-2ip-2dp+4di}{2}\frac{p!(p-1)!}{(2p)!}I_{1-\mathsf{s}}(p+1,p)+\frac{2di+p^{2}-2ip}{2p}\mathsf{s}^{p}(1-\mathsf{s})^{p} \Big)\\
        &= \frac{1}{2}I_{1-\mathsf{s}}(p+1,p)+\frac{d}{2}\binom{2p}{p}\mathsf{s}^{p}(1-\mathsf{s})^{p} .
    \end{align*}
    Combining this with \eqref{mNp interms of mathcal B}, the desired asymptotic expansion \eqref{qLUE moment expansion c0} follows. 
    \end{proof}

\section{Proofs of Theorem \ref{Thm:limiting density in growth regime}} \label{Section_Thm main proof}

In this section, we prove Theorem~\ref{Thm:limiting density in growth regime}. Parts (A), (B), and (C) are established in the subsequent subsections.

\subsection{Proof of Theorem \ref{Thm:limiting density in growth regime} (A)} \label{Subsec_Thm main (A)}

By Theorem~\ref{Thm:Spectralmoment}(B), this part follows from a standard argument in random matrix theory, which derives the limiting spectral distribution from the large-$N$ behaviour of the moments.
We denote by  
\begin{equation}
    g^{(c)}(y):=\int_{0}^{1}\frac{\rho^{(c)}(x)}{y-x} \, dx, \qquad y \in \C\setminus \R 
\end{equation}
the Cauchy transform of the limiting density $\rho^{(c)}$. Then by Theorem~\ref{Thm:Spectralmoment} (B), we have 
\begin{align}
\begin{split}
    g^{(c)}(y)
    &= \frac{1}{y}\sum_{p=0}^{\infty}\frac{1}{y^{p}}\int_{0}^{1}x^{p}\rho^{(c)}(x) \, dx
    =\frac{1}{y}+\frac{1}{y}\sum_{p=1}^{\infty}\frac{1}{y^{p}}\int_{0}^{1}\mathsf{s}^{pu}\sum_{i=0}^{p}(1-\mathsf{s}^{c+u})^{i}(\mathsf{s}^{c}(1-\mathsf{s}^{u}))^{p-i}\binom{p}{i}^{2}\, du\\
    &=\frac{1}{\lambda y}\int_{\mathsf{s}}^{1}\frac{1}{t}\sum_{p=0}^{\infty}\sum_{i=0}^{p}\binom{p}{i}^{2}\Big(\frac{1-\mathsf{s}^{c}t}{\mathsf{s}^{c}(1-t)}\Big)^{i}\Big(\frac{t(1-t)\mathsf{s}^{c}}{y}\Big)^{p}\, dt
\end{split}
\end{align} 
as $ y \to \infty.$ Here, we used the change of variable $\mathsf{s}^{u}=t$ in the last equality.
To simplify the double sum in the integrand, we first note that
\begin{equation}
    \sum_{i=0}^{p}\binom{p}{i}z^{i}=(1-z)^{p}P_{p}\Big(\frac{1+z}{1-z}\Big),
\end{equation}
where $P_n$ is the Legendre polynomial. Then by the well-known generating function of the Legendre polynomial (see e.g. \cite[Eq.~(9.8.68)]{KLS10}), it follows that
\begin{equation} \label{3.15}
  \sum_{p=0}^{\infty}\sum_{i=0}^{p}\binom{p}{i}^{2}z^{i}w^{p}=\frac{1}{\sqrt{1-2w(1+z)+w^2 (1-z)^2}}, \qquad \textup{for }\abs{w}<\min{\Big|{\frac{1+z\pm2\sqrt{z}}{(1-z)^2}}}\Big|.
\end{equation}
Using this identity and substituting $z$ and $w$ accordingly, we obtain
\begin{equation}
    g^{(c)}(y)=\frac{1}{\lambda}\int_{\mathsf{s}}^{1}\frac{1}{t}\frac{dt}{\sqrt{y^2 -2(t+\mathsf{s}^{c}t-2\mathsf{s}^{c}t^2)y+t^2(1-\mathsf{s}^{c})^2}}.
\end{equation}
Notice here that the condition for $w$ in \eqref{3.15} is satisfied for
\begin{equation}
    y\geq \begin{cases}
       b& \textup{if } \lambda \leq \lambda_c,
       \smallskip \\
       1 & \textup{if } \lambda >\lambda_c, 
    \end{cases}
\end{equation}
where $b$ is given by \eqref{def of endpts a and b}. 
Moreover, note that for a fixed $0<x<1$, the inequality
\begin{equation}\label{eqn:x inequality}
    x^2-2(t+\mathsf{s}^{c}t -2\mathsf{s}^{c}t^2)x+t^2(1-\mathsf{s}^{c})^2 <0
\end{equation}
is satisfied for 
\begin{equation} 
    t\in (\max\{x_{0}-x_{1},\mathsf{s}\},x_0 +x_1),  
\end{equation}
where 
 \begin{equation}\label{eqn:def of x0x1}
        x_0=\frac{x(1+\mathsf{s}^{c})}{(1-\mathsf{s}^{c})^2 +4\mathsf{s}^{c}x}, \qquad 
        x_1=\frac{2x\sqrt{\mathsf{s}^{c}(1-x)}}{(1-\mathsf{s}^{c})^{2}+ 4 \mathsf{s}^{c}x}.
    \end{equation}
Note that $x_0 \pm x_1$ are defined as the roots of the LHS of \eqref{eqn:x inequality}. 
Then by the Sokhotski-Plemelj formula, we have
\begin{equation}\label{eqn:integral representation of rho}
\begin{split}
    \rho^{(c)}(x)
    &=\frac{1}{\pi}\lim_{\varepsilon\downarrow0}\Im\, g^{(c)}(x-i\varepsilon) =\frac{1}{\pi\lambda}\int_{\max\{x_0-x_1,\mathsf{s}\}}^{x_0 +x_1}\frac{1}{t}\frac{dt}{\sqrt{-x^2 +2(t+\mathsf{s}^{c}t-2\mathsf{s}^{c}t^2)x-t^2(1-\mathsf{s}^{c})^2}}
    \\
    &=\frac{1}{\pi\lambda \sqrt{(1-\mathsf{s}^{c})^2+4\mathsf{s}^{c}x}}\int_{\max\{x_0-x_1,\mathsf{s}\}}^{x_0 +x_1}\frac{1}{t}\frac{dt}{\sqrt{(x_0+x_1 -t)(t-(x_0 -x_1))}}. 
\end{split}
\end{equation}
Then it follows that
\begin{equation}
    \begin{split}
        \rho^{(c)}(x)
        &=\frac{1}{\pi\lambda \sqrt{(1-\mathsf{s}^{c})^2+4\mathsf{s}^{c}x}}\int_{\mathsf{s}}^{x_0 +x_1}\frac{1}{t}\frac{dt}{\sqrt{(x_0+x_1 -t)(t-(x_0 -x_1))}}\mathbbm{1}_{(a,b)}(x)
        \\
        &\quad+\frac{1}{\pi\lambda \sqrt{(1-\mathsf{s}^{c})^2+4\mathsf{s}^{c}x}}
        \begin{cases}
            0 & \textup{if } \lambda \leq\lambda _c ,
            \smallskip 
            \\
            \displaystyle\int_{{x_0-x_1}}^{x_0 +x_1}\frac{1}{t}\frac{dt}{\sqrt{(x_0+x_1 -t)(t-(x_0 -x_1))}}\mathbbm{1}_{(b,1)}(x) &\textup{if } \lambda >\lambda_c. 
        \end{cases}
    \end{split}
\end{equation}
We now simplify the integrals in this expression. 

First, note that if $e^{-\lambda}<x_0 -x_1$, the integral evaluation
\begin{equation*}
    \int_{\alpha}^{\beta}\frac{1}{t}\frac{dt}{\sqrt{(t-\alpha)(\beta-t)}}=\pi\abs{\alpha \beta}^{-1/2} 
\end{equation*}
gives rise to 
\begin{equation}
    \frac{1}{\pi\lambda \sqrt{(1-\mathsf{s}^{c})^2+4\mathsf{s}^{c}x}}\int_{x_0-x_1}^{x_0 +x_1}\frac{1}{t}\frac{dt}{\sqrt{(x_0+x_1 -t)(t-(x_0 -x_1))}}=\frac{1}{\lambda x}. 
\end{equation} 
On the other hand, for $x$ satisfying $x_0-x_1<\mathsf{s}$, we make use of the integral evaluation
\begin{equation*}
    \int_{\alpha}^{\gamma}\frac{1}{t}\frac{dt}{\sqrt{(t-\alpha)(\beta-t)}}
    =\int_{0}^{\theta_{\gamma}}\frac{2 d\theta}{({\alpha+\beta})+({\beta-\alpha})\cos\theta},\quad \alpha<\gamma<\beta
\end{equation*}
where $\cos\theta_{\gamma}=(2\gamma-\alpha-\beta)/(\beta-\alpha)$, and obtain 
\begin{equation}
    \frac{1}{\pi\lambda \sqrt{(1-\mathsf{s}^{c})^2+4\mathsf{s}^{c}x}}\int_{\mathsf{s}}^{x_0 +x_1}\frac{1}{t}\frac{dt}{\sqrt{(x_0+x_1 -t)(t-(x_0 -x_1))}}
    =\frac{2}{\pi\lambda x}\arctan\sqrt{\frac{x_0-x_1}{x_0 +x_1}\frac{{x_0+x_1-\mathsf{s}}}{{\mathsf{s}-x_0+x_1}}}.
\end{equation}
Combining the above, it follows that 
\begin{equation} \label{def of limiting density in Section 3}
            \rho^{(c)}(x)
            =\frac{2}{\pi\lambda x}\arctan\sqrt{\frac{x_0-x_1}{x_0 +x_1}\frac{{x_0+x_1- \mathsf{s}  }}{{ \mathsf{s} -x_0+x_1}}}\mathbbm{1}_{(a,b)}(x) +\begin{cases}
                0& \textup{if } \lambda \leq \lambda_c , \smallskip \\
                \displaystyle\frac{1}{\lambda x}\mathbbm{1}_{(b,1)}(x) & \textup{if } \lambda > \lambda_c,
            \end{cases} 
    \end{equation} 
Furthermore, by \eqref{eqn:def of x0x1}, we have the expression \eqref{eqn:def of limiting density v0}. 
On the other hand, by \eqref{def of endpts a and b}, we have  
\begin{align*}
\sqrt{ab}=\mathsf{s}-\mathsf{s}^c,\qquad  \sqrt{(1-a)(1-b)}=|1-\mathsf{s}-\mathsf{s}^{1+c}|. 
\end{align*} 
Then by straightforward computations, one can show that \eqref{def of limiting density in Section 3} coincides with \eqref{eqn:def of limiting density}.  
Since $\rho^{(c)}$ is non-negative and integrates to one, it defines a probability density and can therefore be integrated against any continuous test function. This shows the desired convergence. 
\qed

\subsection{Proof of Theorem~\ref{Thm:limiting density in growth regime} (B)} \label{Section_Thm B}

In this subsection, we prove Theorem~\ref{Thm:limiting density in growth regime} (B).
By the general theory of large deviation principles \cite{DD22,HMO25}, the main task is to show that the minimiser of the functional $I[\mu]$ in
\eqref{eq:functional_rhp} is given by \eqref{eqn:def of limiting density}.
For this, we solve  the associated Euler--Lagrange variational conditions explicitly.

To identify the rate function $I[\mu]$, we begin with the Hamiltonian in \eqref{def of jpdf for qL} and analyse its large-$N$ behaviour, which naturally suggests a candidate for the rate function.
For this purpose, we require the asymptotic behaviour of the weight function.

\begin{lem} \label{Lem_asymp of weight}
Let $q=e^{-\lambda/N}$ with $\lambda \ge 0$. Let $\alpha$ be scaled according to \eqref{def of alpha scaling}. Then for $0<x<1$, the weight function \eqref{def of qLUE weight} satisfies 
\begin{equation}
\log  w^{(\mathrm{qL})}(x;q) = -N \Big( \frac{1}{\lambda} \Li_2(x) -c \log x \Big)+O(1),
\end{equation}
as $N \to \infty.$ Here,  $\Li_2(x)$ is given by \eqref{eq:Li2_def}. 
\end{lem}
\begin{proof}
This follows immediately from the asymptotic behaviour of the $q$-Pochhammer symbol established in \cite[Theorem~4]{Mc99}; see also \cite[Lemma~4.1]{FKP24}.   
\end{proof}

By Lemma~\ref{Lem_asymp of weight}, we formally obtain the asymptotic of the joint probability distribution function \eqref{def of jpdf for qL}: 
\begin{equation}
 \prod_{1\le j<k \le N}(x_j-x_k)^2 \prod_{j=1}^N \omega^{\rm (qL)}(x_j)  \approx \exp \Big( -N^2 I[\mu] \Big), 
\end{equation}
where $I[\mu]$ is given by \eqref{eq:functional_rhp}. Here, the $N^2$ terms is coherent with the speed of convergence claimed in Theorem \ref{Thm:limiting density in growth regime} (B).

A standard approach to identifying the unique minimiser of an energy
functional of the form $I[\mu]$ is via the Euler--Lagrange variational
conditions. In the present setting, the corresponding Euler--Lagrange conditions are formulated as follows;
see \cite[Section~2.1.4]{BKMM}.

\begin{prop} \label{Prop_EL conditions}
Let $\mu(dx)=\rho^{(c)}(x)\,dx$, where $\rho^{(c)}$ is given by \eqref{eqn:def of limiting density}. Define  
\begin{equation} \label{def of parameter d}
d= \begin{cases}
b &\textup{if } \lambda < \lambda_c,
\smallskip 
\\
1 &\textup{if } \lambda > \lambda_c. 
\end{cases}
\end{equation}
Then there exists an $\Omega \in \R$ such that the following Euler-Lagrange variational conditions hold: 
\begin{equation} \label{def of EL conditions}
 -2 \int_a^d \log(|x-y|)\, \mu(dy) + \frac{1}{\lambda}\Li_2(x) -c\log(x) 
 \begin{cases}
    = -\Omega &\textup{if } x \in (a,b),
    \smallskip 
    \\
    \ge -\Omega  &\textup{if } x\in(0,a)\cup(d,1), 
    \smallskip 
    \\
    \le -\Omega  &\textup{if }  x\in(b,d).
 \end{cases}
\end{equation}
\end{prop}

Although it is sufficient to verify this proposition directly, it is more
instructive to illustrate how one can construct a probability measure
satisfying \eqref{def of EL conditions} in a setting where the solution is not known \emph{a priori}. 

For this purpose, we first recall some standard definitions from \cite{BKMM}.
Let $\mu(dx)\in \mathcal{P}_{\lambda}(0,1)$. For any interval
$I \subseteq(0,1)$, we use the following terminology: 
\begin{itemize}
    \item $I$ is called a \textit{void} if $\mu(I)\equiv 0$;
\smallskip
    \item $I$ is called a \textit{saturated region} if the upper constraint $\frac{1}{\lambda x}$ is active for all $x\in I$, i.e. $\mu(I)=\int_I\frac{1}{\lambda x}dx$; 
\smallskip 
    \item $I$ is called a \textit{band} if it is neither a void nor a saturated region. 
\end{itemize}
To construct a probability measure $\mu(dx)$ satisfying
\eqref{def of EL conditions}, one needs to assume a basic structural form of the measure. To be more precise, we suppose that there exist parameters
$0 \le a<b \le d\leq 1$ such that $(0,a)$ and $(d,1)$ are voids; $(a,b)$ is a band; and  $(b,d)$ is a saturated region. 
Furthermore, we assume that $\mu(dx)$ is absolutely continuous with respect
to the Lebesgue measure, so that there exists a density $\mu(x)$ satisfying
$\mu(dx)=\mu(x)\,dx$. 

We note that such assumptions are far from obvious in general. However, in the present setting we already know the solution---albeit temporarily setting
this knowledge aside---from complementary approaches. This allows us to
verify \emph{a posteriori} that these assumptions are natural, in view of the
explicit form \eqref{eqn:def of limiting density}.

\begin{proof}[Proof of Proposition~\ref{Prop_EL conditions}; Euler-Lagrange equality]
We first prove the equality part in \eqref{def of EL conditions}. Let $x \in (a,b).$ 
By differentiating \eqref{def of EL conditions} with respect to $x$, it follows from \eqref{eq:Li2_def} that 
\begin{equation}
\label{eq:pv_relation}
    2\pv \int_a^d \frac{\mu(y)}{x-y}\, dy  + \frac{\log(1-x)}{\lambda x} + \frac{c}{x} = 0, \qquad x \in (a,b), 
\end{equation}
where $\pv$ is the Cauchy principal value. 

We consider the Cauchy transform
\begin{equation} \label{def of g Cauchy}
g(z) := \int_a^d \frac{\mu(y)}{z-y}\, dy
\end{equation}
of the equilibrium measure $\mu$. Under the assumed structure of the support of $\mu(dx)$, the Sokhotski--Plemelj theorem implies that $g(z)$ satisfies the following Riemann--Hilbert problem (RHP).

\begin{rhp}
\label{rhp:little_g}
    The function $g(z)$ satisfies the following conditions: 
    \begin{enumerate}
        \item $g(z)$ is analytic for $z\in \mathbb{C}\setminus(a,d)$;
        \smallskip 
        \item \label{item:recovery_formula} $\displaystyle g_+(x)-g_-(x) = -2\pi i \mu(x)$ for all $x\in(a,b)$;
        \smallskip 
        \item $\displaystyle g_+(x)-g_-(x) = -\frac{2\pi i}{\lambda x}$ for all $x\in(b,d)$;
        \smallskip 
        \item $\displaystyle g_+(x)+g_-(x) 
        = - \frac{\log(1-x)}{\lambda x}-\frac{c}{x} $ for all $x\in(a,b)$;
        \smallskip 
        \item $\displaystyle g(z) \sim \frac{1}{z} + \frac{ \mathsf{M}_1 }{z^2} + O(z^{-3})$ as $z \to \infty$, where $\mathsf{M}_1:= \int x \mu(x)\,dx$. 
    \end{enumerate}
\end{rhp}
Here, $g_\pm(x) = \lim_{\varepsilon\to0^+} g(x\pm i\varepsilon)$. Here, the last condition follows from 
\begin{align*}
g(z) = \frac{1}{z}\int_a^d  \frac{\mu(y)}{1-y/z}\,dy = \frac{1}{z} \int_a^d \Big( 1+\frac{y}{z} +O(z^{-2}) \Big)\mu(y)\,dy,
\end{align*}
as $z \to \infty.$ At this stage, we do not specify $\mathsf{M}_1$; however, we will later make essential use of the fact that $\mathsf{M}_1=\mathcal{M}_1^{(\mathrm{qL})}$, where $\mathcal{M}_1^{(\mathrm{qL})}$ is given in
\eqref{eq:first_moment}. This is precisely where the alternative approach based on the moment method plays a complementary role within the large deviation framework.

To solve \textbf{RHP~\ref{rhp:little_g}}, we introduce the transformation
\begin{equation}
\varphi(z):= \frac{g(z)}{R(z)}, \qquad    R(z) := (z-a)^{\frac{1}{2}}(z-b)^{\frac{1}{2}},
\end{equation}
where the square roots are taken with canonical branch cuts. 
Note that 
\begin{equation} \label{jump of R pm}
R_+(x) = \begin{cases}
 R_-(x) &\textup{if } x <a \textup{ or } x>b, 
 \smallskip 
 \\
 -R_-(x) &\textup{if } a<x<b. 
\end{cases} 
\end{equation}
The function
$\varphi(z)$ then satisfies the following Riemann--Hilbert problem.

\begin{rhp}
\label{rhp:varphi}
    The function $\varphi(z)$ satisfies the following conditions:
    \begin{enumerate}
        \item $\varphi(z)$ is analytic for $z\in \mathbb{C}\setminus(a,d)$;
        \smallskip 
        \item $\displaystyle \varphi_+(x)-\varphi_-(x) = -\frac{2\pi i}{\lambda xR(x)}$ for all $x\in(b,d)$;
        \smallskip 
        \item $\displaystyle \varphi_+(x)-\varphi_-(x) = - \frac{\log(1-x)}{\lambda xR_+(x)} -\frac{c}{xR_+(x)}$ for all $x\in(a,b)$; 
        \smallskip 
        \item $\displaystyle \varphi(z) \sim \frac{1}{z^2} + \frac{ \mathsf{M}_1 + \frac{1}{2}(a+b)}{z^3} + O(z^{-4})$ as $z \to \infty$. 
    \end{enumerate}
\end{rhp}

Now, by using the second and the third condition in \textbf{RHP \ref{rhp:varphi}} and the Sokhotski–Plemelj formula, we have 
\begin{equation}
\label{eq:sol_phi}
    \varphi(z) =  \frac{1}{2\pi i}  \int_a^b \Big( \frac{\log(1-s)}{\lambda sR_+(s) } + \frac{c}{sR_+(s) } \Big) \,\frac{ds}{z-s} + \int_b^d \frac{1}{\lambda R(s)s } \,\frac{ds}{z-s} . 
\end{equation}
By expanding \eqref{eq:sol_phi} as $z \to \infty$ we obtain the asymptotic behaviour 
\begin{align}
\label{eq:var_phi_asymp}
    \begin{split}
    \varphi(z)& = \bigg( \frac{1}{2\pi i}\int_a^b  \Big( \frac{\log(1-s)}{\lambda sR_+(s) } + \frac{c}{sR_+(s) } \Big)\, ds + \int_b^d \frac{1}{\lambda s R(s)}\,ds \bigg)  \frac{1}{z}
    \\ 
    &\quad +  \bigg( \frac{1}{2\pi i}\int_a^b \Big( \frac{\log(1-s)}{\lambda R_+(s) } + \frac{c}{R_+(s) } \Big)\, ds + \int_b^d \frac{1}{\lambda  R(s)} \,ds \bigg) \frac{1}{z^2}
    \\
    & \quad +  \bigg( \frac{1}{2\pi i}\int_a^b \Big( \frac{s\log(1-s)}{\lambda R_+(s) } + \frac{cs}{R_+(s) } \Big)\, ds + \int_b^d \frac{s}{\lambda  R(s)} \,ds \bigg)\frac{1}{z^3} + O(z^{-4})\,.
    \end{split}
\end{align} 
We now match the coefficients in \eqref{eq:var_phi_asymp} with the final condition in \textbf{RHP~\ref{rhp:varphi}}. For this purpose, the integrals appearing in \eqref{eq:var_phi_asymp} are evaluated separately in Propositions~\ref{Prop_int evaluation 1}, \ref{Prop_int evaluation 2} and \ref{Prop_int evaluation 3} below. As a result, we deduce that the parameters
$a,b,d$ must satisfy the following nonlinear system of equations. 
\begin{align}
&  -\frac{c}{2} - \frac{1}{\lambda } \log\Big( \frac{\sqrt{(1-a)b} + \sqrt{(1-b)a}}{\sqrt{b} + \sqrt{a}}\Big)+ \frac{1}{\lambda}  \log\Big( \frac{\sqrt{b(d-a)} + \sqrt{a(d-b)}}{\sqrt{b(d-a)} - \sqrt{a(d-b)}}\Big) = 0,  \label{eqn for system 1}
\\
&   -\frac{c}{2} - \frac{1}{\lambda} \log\Big(\frac{\sqrt{1-a} + \sqrt{1-b}}{2}\Big)+ \frac{2 }{\lambda }  \log \Big(\frac{\sqrt{d-a}+\sqrt{d-b}}{\sqrt{b-a}}\Big) =  1 ,  \label{eqn for system 2}
\\
\begin{split} 
&  -\Big(c+\frac{1}{\lambda} + 2\Big)\frac{a+b}{4} + \frac{1}{2\lambda}\bigg(1-\sqrt{(1-b)(1-a)} - (a+b) \log\Big(\frac{\sqrt{1-a} + \sqrt{1-b}}{2}\Big)\bigg)  
\\
&\quad +\frac{1}{\lambda} \bigg( \sqrt{(d-a) (d-b)}+(a+b) \log \Big(\frac{\sqrt{d-a}+\sqrt{d-b}}{\sqrt{b-a}}\Big) \bigg) = \mathsf{M}_1. \label{eqn for system 3}
\end{split}
\end{align}
Up to this point, we have not used any a priori information about $\mu(x)$ except its structure of the support. To solve the resulting system of equations, however, we must specify $\mathsf{M}_1=\mathcal{M}_1^{(\mathrm{qL})}$ as given in \eqref{eq:first_moment}. It then follows from lengthy but straightforward computations that the parameters $a,b$, and $d$, defined in \eqref{def of endpts a and b} and \eqref{def of parameter d}, indeed satisfy \eqref{eqn for system 1}, \eqref{eqn for system 2}, and \eqref{eqn for system 3}.

Next, we compute $\mu(x)\, dx$ using the second condition in \textbf{RHP~\ref{rhp:little_g}}. By definition, we have  
\begin{align}
\begin{split}
    \mu(x) &= -\frac{1}{2\pi i}(R_+(x)\varphi_+(x) - R_-(x)\varphi_-(x))
    \\
    &= - \frac{R_+(x)}{2\pi i} ( \varphi_+(x) + \varphi_-(x) ) \mathbbm{1}_{(a,b)}(x) - \frac{R(x)}{2\pi i}(\varphi_+(x)-\varphi_-(x))\mathbbm{1}_{(b,d)}(x)\,.
\end{split}
\end{align} 
We first consider the case $\lambda\leq \lambda_c$. Recall from \eqref{def of parameter d} that, in the present case, we have $d=b$.
By evaluating the integrals in \eqref{eq:sol_phi} using Propositions~\ref{Prop_int evaluation 1} and~\ref{Prop_int evaluation 3}, we obtain that, for $z \notin (a,b)$, 
\begin{align*}
    \varphi(z) &= -\frac{c}{2z\sqrt{ab}} - \frac{c}{2zR(z)} -\frac{\log(1-z)}{2\lambda zR(z)} +\frac{1}{2z\lambda R(z)}\log\bigg( \frac{ \sqrt{z-b}+ \sqrt{ z-a} }{ \sqrt{z-b}-\sqrt{ z-a} }\bigg)
    \\
    &\quad -\frac{1}{2z\lambda R(z)}\log\bigg( \frac{ \sqrt{(1-a)(z-b)} + \sqrt{ (1-b)(z-a)} }{  \sqrt{(1-a)(z-b)}-\sqrt{(1-b)(z-a)} }\bigg)  - \frac{1}{z\lambda \sqrt{ab}} \log\bigg( \frac{\sqrt{b(1-a)} +\sqrt{a(1-b)}}{\sqrt{b} +\sqrt{a} }\bigg)\,.
\end{align*} 
Then, by taking the plus and minus boundary values of this expression,
and using both \eqref{eqn for system 1} and the identity~\cite[Eq.~(4.23.26)]{NIST} 
\begin{equation}
    \arctan(z) = \frac{i}{2}\log \Big(\frac{i+z}{i-z}\Big), \qquad iz\not\in (-\infty,-1)\cup(1,+\infty)\,,
\end{equation}
we deduce that for $x \in (a,b)$,  
\begin{equation}
\begin{split}  
\varphi_+(x)+\varphi_-(x) =-\frac{2i}{x\lambda R_+(x)}\bigg(\arctan\bigg(\sqrt{\frac{1-a}{1-b}}\sqrt{\frac{b-x}{x-a}}\bigg) - \arctan\bigg(\sqrt{\frac{b-x}{x-a}}\bigg)\bigg)\,.
\end{split}
\end{equation}
Therefore, we obtain that for $x \in (a,b)$, 
\begin{align*}
  \mu(x) &= \frac{1}{x\lambda\pi}\bigg(\arctan\bigg(\sqrt{\frac{1-a}{1-b}}\sqrt{\frac{b-x}{x-a}}\bigg) - \arctan\bigg(\sqrt{\frac{b-x}{x-a}}\bigg)\bigg) 
       \\
       & =\frac{1}{x\lambda\pi}\arctan\bigg(\frac{(\sqrt{1-a}-\sqrt{1-b})\sqrt{(b-x)(x-a)}}{\sqrt{1-b}(x-a) + \sqrt{1-a}(b-x)}\bigg)
       \\
       & =\frac{2}{x\lambda\pi} \arctan \bigg( \frac{(\sqrt{1-a}-\sqrt{1-b})\sqrt{(b-x)(x-a)}}{\sqrt{1-b}(x-a) + \sqrt{1-a}(b-x) + (b-a)\sqrt{1-x}} \bigg), 
\end{align*} 
where we have used the addition formula and the double-angle formula
for the arctangent function; see e.g. \cite[Eq.~(4.24.15)]{NIST}. This shows that for $\lambda \le \lambda_c$, the density of $\mu(x)$ is given by \eqref{eqn:def of limiting density}. 
 
We now consider the case $\lambda > \lambda_c$.
It is straightforward to see that $\mu(x) = 1/(\lambda x)$ for $x \in (b,1)$.
On the other hand, for $x \in (a,b)$, by calculations analogous to those above, one can show that
\begin{align*}
 \mu(x) &= \frac{1}{x\lambda\pi}\bigg(\arctan\bigg(\sqrt{\frac{1-a}{1-b}}\sqrt{\frac{b-x}{x-a}}\bigg) - \arctan\bigg(\sqrt{\frac{b-x}{x-a}}\bigg)+2\arctan\bigg(\sqrt{\frac{(x-a)(1-b)}{(b-x)(1-a)}}\bigg)\bigg)  
       \\ 
       & = \frac{2}{x\lambda\pi} \bigg( \arctan \bigg( \frac{(\sqrt{1-a}-\sqrt{1-b})\sqrt{(b-x)(x-a)}}{\sqrt{1-b}(x-a) + \sqrt{1-a}(b-x) + (b-a)\sqrt{1-x}} \bigg) + \arctan\bigg( \sqrt{\frac{(1-b)(x-a)}{(1-a)(b-x)}}\bigg) \bigg)  
       \\
       &=  \frac{2}{x\lambda\pi}  \arctan \bigg(  \sqrt{ \frac{x-a }{b-x} } \frac{\sqrt{1-x}+\sqrt{1-b} }{ \sqrt{1-x}+\sqrt{1-a} }  \bigg) . 
\end{align*} 
This again shows that for $\lambda >  \lambda_c$, the density of $\mu(x)$ is given by \eqref{eqn:def of limiting density}. 

We have shown that, for $x \in (a,b)$, the derivative of the LHS of \eqref{def of EL conditions} with respect to $x$ vanishes. Therefore, the equality part of \eqref{def of EL conditions} follows.
\end{proof}

Next we show the inequalities in \eqref{def of EL conditions}.

\begin{proof}[Proof of Proposition~\ref{Prop_EL conditions}; Euler-Lagrange inequalities]

Let 
\begin{equation}
\label{eq:h_func}
    h(x) = g_+(x) + g_-(x) + \frac{c}{x} + \frac{\log(1-x)}{\lambda x},
\end{equation}
where $g(x)$ is given by \eqref{def of g Cauchy}. Note that, by \textbf{RHP~\ref{rhp:little_g}}, we have $h(x)=0$ for $x \in (a,b)$.

We first consider the case $\lambda > \lambda_c$; thus $d=1$.  
We claim that 
\begin{equation} \label{h sign lambda big}
h(x) \begin{cases}
>0 &\textup{if } x \in (0,a),
\smallskip 
\\
<0  &\textup{if } x \in (b,1).  
\end{cases}
\end{equation}
Then the desired conclusion follows by integration. 
 
By residue calculations analogous to those above and \eqref{eqn for system 2}, it follows that, for $x \in (0,a)\cup(b,1)$, 
\begin{equation}
  h(x) = -R(x)\bigg(\frac{c}{x\sqrt{ab}} +\frac{1}{\lambda}\int_1^\infty \frac{1}{sR(s)(s-x)}\, ds\bigg)= -\frac{R(x)}{x}\Big(\frac{c}{\sqrt{ab}} -\frac{J(0)}{\lambda} +\frac{  J(x) }{\lambda }  \Big), 
\end{equation}
where 
\begin{equation} \label{def of J}
J(x) := \int_1^\infty \frac{1}{(s-x)R(s)}\, ds\,.
\end{equation}
By standard integration techniques, such as Euler substitutions, one can evaluate $J(x)$ as 
\begin{equation}  \label{integral eval of J}
    J(x) = 
    \begin{cases}
    \displaystyle  
    \frac{2 }{\sqrt{(a-x)(b-x)}} \log\bigg( \frac{(\sqrt{b-x} + \sqrt{a-x})\sqrt{1-x} }{\sqrt{b-x}\sqrt{1-a} + \sqrt{a-x}\sqrt{1-b}}\bigg) &\textup{if } 0 \le x < a, 
    \smallskip
    \\
    \displaystyle 
    -\frac{ 2 }{\sqrt{(x-a)(x-b)}} \log\bigg( \frac{(\sqrt{x-b} + \sqrt{x-a})\sqrt{1-x} }{\sqrt{x-b}\sqrt{1-a} + \sqrt{x-a}\sqrt{1-b}}\bigg) &\textup{if } b<x<1. 
    \end{cases} 
\end{equation} 

Therefore, $J(x)>0$ for $x\in(0,a)$, whereas $J(x)<0$ for $x \in (b,1)$. Furthermore, by \eqref{eqn for system 1} we have $\frac{c}{ \sqrt{ab} }- \frac{J(0)}{\lambda}=0$.  Therefore we conclude that \eqref{h sign lambda big} holds. 

Next, we consider the case $\lambda \le \lambda_c$, thus $d=b$. In this case, it suffices to show that  
\begin{equation} \label{h sign lambda small}
h(x)>0 \quad \textup{if } x \in (0,a)\cup (b,1). 
\end{equation} 
We notice that in this case, one can compute $h(x)$ as 
\begin{equation}
    h(x) = -\frac{R(x)}{x }\left(\frac{c}{ \sqrt{ab} }-\frac{J(0)}{\lambda} - \frac{2I(0)}{\lambda} + \frac{J(x)}{\lambda} + \frac{2I(x)}{\lambda}\right)\,,
    \end{equation}
where $J(x)$ is given by \eqref{def of J} and  
\begin{equation}
I(x):= \int_b^1 \frac{1}{(s-x)R(s)} \, ds. 
\end{equation}
Similar to the above, $I(x)$ is evaluated as 
\begin{equation}  \label{integral eval of I}
 I(x) =  \begin{cases}
   \displaystyle     \frac{1}{\sqrt{(b-x)(a-x)}}\log\bigg(\frac{\sqrt{b-x}\sqrt{1-a} + \sqrt{a-x}\sqrt{1-b} }{\sqrt{b-x}\sqrt{1-a} - \sqrt{a-x}\sqrt{1-b}}\bigg)  & \textup{if } 0<x<a, 
        \smallskip 
        \\
     \displaystyle      -\frac{1}{\sqrt{(x-b)(x-a)}}\log\bigg(\frac{\sqrt{x-b}\sqrt{1-a} + \sqrt{x-a}\sqrt{1-b} }{ \sqrt{x-a}\sqrt{1-b} - \sqrt{x-b}\sqrt{1-a}}\bigg) & \textup{if } b<x<1. 
    \end{cases}
\end{equation} 
Notice that by \eqref{eqn for system 1}, we have $c-\frac{J(0)}{\lambda} - \frac{2I(0)}{\lambda}=0$. Moreover, since $I(x)>0$ for $x \in (0,a) \cup (b,1)$, we obtain \eqref{h sign lambda small}. This completes the proof. 
\end{proof}

\begin{proof}[Proof of Theorem~\ref{Thm:limiting density in growth regime} (B)] The first assertion follows from the general theory developed in \cite[Theorem~1.1]{HMO25} and \cite[Theorem~1.7]{DD22}, together with Lemma~\ref{Lem_asymp of weight}. In particular, by considering the special case $\eta=\theta=1$ and $\beta=2$ in \cite[Theorem~1.1]{HMO25} and then applying Varadhan’s lemma~\cite{LDP_book}, the claim follows.
The second assertion follows directly from Proposition~\ref{Prop_EL conditions} and the fact that the Euler--Lagrange variational condition \eqref{def of EL conditions} characterises the unique minimiser of the energy functional.
\end{proof}

\subsection{Proof of Theorem~\ref{Thm:limiting density in growth regime} (C)} \label{Subsec_Thm main (C)}

In this subsection, we study the limiting zero distribution of the little $q$-Laguerre polynomials as a third approach to the $q$-deformed Marchenko-Pastur law \eqref{eqn:def of limiting density}, and thereby establish Theorem~\ref{Thm:limiting density in growth regime} (C).
As mentioned earlier, our analysis follows the general method introduced by Kuijlaars and Van Assche~\cite{KV99}; see also the earlier work~\cite{VK91}. 
To apply this method, we consider the double-indexed \emph{normalised} little $q$-Laguerre polynomials $P_{n,N}(x)$ defined by
\begin{equation}
    P_{n,N}(x)=-\frac{1}{\sqrt{h_{n}}}p_{n}(x|e^{-\frac{\lambda (cN+d)}{N}};e^{-\frac{\lambda}{N}}).
\end{equation}
Here, $h_n$ is given by \eqref{eqn:orthogonality of little q Laguerre}, and we use the scalings $q=e^{-\lambda/N}$ and $\alpha=cN+d$.
Then, by the three-term recurrence relation \eqref{eqn:three term of qlaguerre} for $p_n$, we have
\begin{equation}
    xP_{n,N}(x)=a_{n+1,N}P_{n+1,N}(x)+b_{n,N}P_{n,N}(x)+a_{n,N}P_{n-1,N}(x), 
\end{equation}
where the recurrence coefficients are given by
\begin{equation}
\begin{split}
    a_{n,N}&=e^{-\frac{\lambda n}{N}}\sqrt{e^{-\frac{\lambda(cN+(d-1))}{N}}(1-e^{-\frac{\lambda n}{N}})(1-e^{-\frac{\lambda(n+cN+d)}{N}})}, 
    \\
    b_{n,N}&=e^{-\frac{\lambda n}{N}}(1-e^{-\frac{\lambda(n+cN+d+1)}{N}})+e^{-\frac{\lambda(n+cN+d)}{N}}(1-e^{-\frac{\lambda n}{N}}).
\end{split}
\end{equation}
For $t>0$, we write $\lim_{n/N\to t} X_{n,N} = X$ if $\lim_{j\to\infty} X_{n_j,N_j} = X$ for any sequences $\{n_j\}$ and $\{N_j\}$ such that $N_j \to \infty$ and $n_j/N_j \to t$ as $j \to \infty$.
Note that by definition, we have 
\begin{equation}
\begin{split}
    A(t):=\lim_{n/N\rightarrow t}a_{n,N}
    &=e^{-\lambda t}\sqrt{e^{-c\lambda}(1-e^{-\lambda t})(1-e^{-(c+t)\lambda})},\\
    B(t):=\lim_{n/N\rightarrow t}b_{n,N}
    &=e^{-\lambda t}(1-e^{-(c+t)\lambda})+e^{-(c+t)\lambda}(1-e^{-\lambda t}).
\end{split}
\end{equation}

Recall that for $\mathfrak{a}<\mathfrak{b}$, the arcsine distribution $\omega_{[\mathfrak{a},\mathfrak{b}]}$ is a measure defined by 
\begin{equation}
    \frac{d\omega_{[\mathfrak{a},\mathfrak{b}]}(x)}{dx}=
    \begin{cases}
        \displaystyle\frac{1}{\pi\sqrt{(\mathfrak{b}-x)(x-\mathfrak{a}})},&\text{if }x\in(\mathfrak{a},\mathfrak{b}),\\
        0,&\text{otherwise.}
    \end{cases}
\end{equation}
Applying \cite[Theorem~1.4]{KV99}, we deduce that the empirical zero distribution $\nu_{n,N}(dx)$ of the little $q$-Laguerre polynomials converges to
\begin{equation}
\nu:=\lim_{n/N\rightarrow1}\nu_{n,N}=\int_{0}^{1}\omega_{[B(s)-2A(s),B(s)+2A(s)]}\,ds.
\end{equation}
We also note that by \cite[Eq.~(1.10)]{KV99}, the support of the limiting zero distribution $\nu$ is given by
\begin{equation}
    I_\lambda:=\Big[\inf_{0<s<1}(B(s)-2A(s)),\sup_{0<s<1}(B(s)+2A(s))\Big].
\end{equation}
Therefore, it follows that the density function of $\nu$ is given by
\begin{equation}
    \frac{d\nu(x)}{dx}=\begin{cases}
        \displaystyle\frac{1}{\pi}\int_{t_{-}(x)}^{\min\{1,t_{+}(x)\}}\frac{ds}{\sqrt{(B(s)+2A(s)-x)(x-B(s)+2A(s))}} , &\text{if } x\in I_{\lambda},  
        \smallskip 
        \\
        0,&\text{otherwise,}
    \end{cases}
\end{equation}
where $t_{\pm}(x)$ denote the endpoints of the interval
\begin{equation}
    \big\{s>0:B(s)-2A(s)\leq x\leq B(s)+2A(s)\big\}.
\end{equation}
Since this expression coincides with the integral representation of the limiting spectral density \eqref{eqn:integral representation of rho} under the change of variables $t = e^{-\lambda s}$, this completes the proof of Theorem~\ref{Thm:limiting density in growth regime} (C).  \qed

\appendix
\section{Integral evaluation}
\label{app:integrals}

In this appendix, we provide explicit computations of the integrals that appear in the coefficients of the asymptotic expansion~\eqref{eq:var_phi_asymp} for $\varphi(z)$. Recall that for $0<a<b \le 1$, 
$$
R(z)=(z-a)^\frac{1}{2}(z-b)^\frac{1}{2}, \qquad R_+(z) = \lim_{\varepsilon\to0^+} R(z+i\varepsilon). 
$$

\begin{prop} \label{Prop_int evaluation 1}
Let $0<a < b \le 1$. Then we have 
\begin{equation}
\label{eq:main_part_prop1}
   \frac{1}{2\pi i} \int_a^b \frac{1}{s(z-s)R_+(s)} \, ds =  -\frac{1}{2z\sqrt{ab}} - \frac{1}{2zR(z)} \qquad \textup{for } z \not\in[a,b]\,,
\end{equation}
and 
    \begin{equation}
    \label{eq:expansion_prop1}
        \frac{1}{2\pi i}\int_a^b \frac{s^j}{R_+(s)} \, ds = \begin{cases}
            -\frac{1}{2\sqrt{ab}} &\textup{if } j=-1, 
            \smallskip 
            \\
            -\frac{1}{2} &\textup{if } j=0, 
            \smallskip 
            \\
    -\frac{1}{4} - \frac{a+b}{4}  &\textup{if } j=1.  
        \end{cases}
    \end{equation} 
\end{prop}

\begin{proof}
We observe that~\eqref{eq:main_part_prop1} implies~\eqref{eq:expansion_prop1}, since the three integrals correspond to the coefficients of $z^{-1}$, $z^{-2}$, and $z^{-3}$ in the large-$z$ expansion of \eqref{eq:main_part_prop1}. 

Let $\gamma$ denote the contour depicted in Figure~\ref{fig:deformation} (A). Deforming $\gamma$ into a contour $\gamma'$ that lies close to and encircles the interval $[a,b]$, we obtain
\begin{align*}
\int_\gamma \frac{1}{s(z-s)R(s)} \, ds = \int_{\gamma'}
\frac{1}{s(z-s)R(s)} \, ds = \int_{ a }^b \frac{1}{s(z-s)} \Big( \frac{1}{R_+(s)}-\frac{1}{R_-(s)}\Big) \,ds  = 2\int_a^b \frac{1}{s(z-s)R_+(s)} \, ds, 
\end{align*}
where we have used \eqref{jump of R pm}. 
Then it follows that 
\begin{equation}
    \frac{1}{2\pi i}\int_a^b \frac{1}{s(z-s)R_+(s)} \, ds = \frac{1}{4\pi i}\int_\gamma \frac{1}{s(z-s)R(s)} \, ds\,.
\end{equation}
Then \eqref{eq:main_part_prop1} follows from the residue calculus. 
\end{proof}

\begin{figure}[t]
	\begin{subfigure}{0.48\textwidth}
	\begin{center}

\tikzset{every picture/.style={line width=0.75pt}}  

\begin{tikzpicture}[scale=0.7,x=0.75pt,y=0.75pt,yscale=-1,xscale=1]

\draw    (280.2,210.8) -- (459,210) ;
\draw [shift={(459,210)}, rotate = 359.74] [color={rgb, 255:red, 0; green, 0; blue, 0 }  ][fill={rgb, 255:red, 0; green, 0; blue, 0 }  ][line width=0.75]      (0, 0) circle [x radius= 3.35, y radius= 3.35]   ;
\draw [shift={(280.2,210.8)}, rotate = 359.74] [color={rgb, 255:red, 0; green, 0; blue, 0 }  ][fill={rgb, 255:red, 0; green, 0; blue, 0 }  ][line width=0.75]      (0, 0) circle [x radius= 3.35, y radius= 3.35]   ;

\draw  [fill={rgb, 255:red, 0; green, 0; blue, 0 }  ,fill opacity=1 ] (216.5,210.79) .. controls (216.5,208.84) and (218.09,207.25) .. (220.04,207.25) .. controls (222,207.25) and (223.58,208.84) .. (223.58,210.79) .. controls (223.58,212.75) and (222,214.33) .. (220.04,214.33) .. controls (218.09,214.33) and (216.5,212.75) .. (216.5,210.79) -- cycle ;

\draw  [fill={rgb, 255:red, 0; green, 0; blue, 0 }  ,fill opacity=1 ] (290,118.29) .. controls (290,116.34) and (291.59,114.75) .. (293.54,114.75) .. controls (295.5,114.75) and (297.08,116.34) .. (297.08,118.29) .. controls (297.08,120.25) and (295.5,121.83) .. (293.54,121.83) .. controls (291.59,121.83) and (290,120.25) .. (290,118.29) -- cycle ;

\draw  [color={rgb, 255:red, 65; green, 117; blue, 5 }  ,draw opacity=1 ][line width=0.75]  (271.58,121.83) .. controls (271.58,109.15) and (281.86,98.87) .. (294.54,98.87) .. controls (307.22,98.87) and (317.5,109.15) .. (317.5,121.83) .. controls (317.5,134.51) and (307.22,144.79) .. (294.54,144.79) .. controls (281.86,144.79) and (271.58,134.51) .. (271.58,121.83) -- cycle ;

\draw  [color={rgb, 255:red, 65; green, 117; blue, 5 }  ,draw opacity=1 ][line width=0.75]  (197.08,210.79) .. controls (197.08,198.11) and (207.36,187.83) .. (220.04,187.83) .. controls (232.72,187.83) and (243,198.11) .. (243,210.79) .. controls (243,223.47) and (232.72,233.75) .. (220.04,233.75) .. controls (207.36,233.75) and (197.08,223.47) .. (197.08,210.79) -- cycle ;

\draw  [draw opacity=0][line width=0.75]  (510.8,252.26) .. controls (492.24,323.56) and (427.37,376.2) .. (350.2,376.2) .. controls (258.56,376.2) and (184.27,301.97) .. (184.27,210.4) .. controls (184.27,118.83) and (258.56,44.6) .. (350.2,44.6) .. controls (441.84,44.6) and (516.13,118.83) .. (516.13,210.4) .. controls (516.13,224.75) and (514.31,238.67) .. (510.88,251.95) -- (350.2,210.4) -- cycle ; \draw  [color={rgb, 255:red, 65; green, 117; blue, 5 }  ,draw opacity=1 ][line width=0.75]  (510.8,252.26) .. controls (492.24,323.56) and (427.37,376.2) .. (350.2,376.2) .. controls (258.56,376.2) and (184.27,301.97) .. (184.27,210.4) .. controls (184.27,118.83) and (258.56,44.6) .. (350.2,44.6) .. controls (441.84,44.6) and (516.13,118.83) .. (516.13,210.4) .. controls (516.13,224.75) and (514.31,238.67) .. (510.88,251.95) ;  
\draw  [color={rgb, 255:red, 65; green, 117; blue, 5 }  ,draw opacity=1 ][fill={rgb, 255:red, 65; green, 117; blue, 5 }  ,fill opacity=1 ] (179.07,213.29) -- (184.32,203.22) -- (189.4,213.38) ;
\draw  [color={rgb, 255:red, 65; green, 117; blue, 5 }  ,draw opacity=1 ][fill={rgb, 255:red, 65; green, 117; blue, 5 }  ,fill opacity=1 ] (512.08,153.97) -- (509.96,165.13) -- (502.16,156.88) ;
\draw  [color={rgb, 255:red, 65; green, 117; blue, 5 }  ,draw opacity=1 ][fill={rgb, 255:red, 65; green, 117; blue, 5 }  ,fill opacity=1 ] (223.76,192.99) -- (213.24,188.71) -- (222.87,182.7) ;
\draw  [color={rgb, 255:red, 65; green, 117; blue, 5 }  ,draw opacity=1 ][fill={rgb, 255:red, 65; green, 117; blue, 5 }  ,fill opacity=1 ] (299.21,104.23) -- (289.22,98.83) -- (299.45,93.9) ;
 
\draw (278.2,214.2) node [anchor=north east] [inner sep=0.75pt]    {$a$};
 
\draw (461,213.4) node [anchor=north west][inner sep=0.75pt]    {$b$};
 
\draw (222.04,214.19) node [anchor=north west][inner sep=0.75pt]    {$0$};
 
\draw (295.54,121.69) node [anchor=north west][inner sep=0.75pt]    {$z$};
 
\draw (312.67,47.4) node [anchor=north west][inner sep=0.75pt]    {$\gamma $};
\end{tikzpicture}
	\end{center}
	\subcaption{ Deformed contour in Proposition~\ref{Prop_int evaluation 1} }
\end{subfigure}	
	\begin{subfigure}{0.48\textwidth}
	\begin{center}	
\tikzset{every picture/.style={line width=0.75pt}} 

\begin{tikzpicture}[scale=0.7, x=0.75pt,y=0.75pt,yscale=-1,xscale=1]
  
\draw    (260.2,190.8) -- (330.2,190.4) ;
\draw [shift={(330.2,190.4)}, rotate = 359.67] [color={rgb, 255:red, 0; green, 0; blue, 0 }  ][fill={rgb, 255:red, 0; green, 0; blue, 0 }  ][line width=0.75]      (0, 0) circle [x radius= 3.35, y radius= 3.35]   ;
\draw [shift={(260.2,190.8)}, rotate = 359.67] [color={rgb, 255:red, 0; green, 0; blue, 0 }  ][fill={rgb, 255:red, 0; green, 0; blue, 0 }  ][line width=0.75]      (0, 0) circle [x radius= 3.35, y radius= 3.35]   ;
 
\draw  [fill={rgb, 255:red, 0; green, 0; blue, 0 }  ,fill opacity=1 ] (196.5,190.79) .. controls (196.5,188.84) and (198.09,187.25) .. (200.04,187.25) .. controls (202,187.25) and (203.58,188.84) .. (203.58,190.79) .. controls (203.58,192.75) and (202,194.33) .. (200.04,194.33) .. controls (198.09,194.33) and (196.5,192.75) .. (196.5,190.79) -- cycle ;
 
\draw    (496.2,190.4) -- (390.71,190) ;
\draw [shift={(390.71,190)}, rotate = 180.22] [color={rgb, 255:red, 0; green, 0; blue, 0 }  ][fill={rgb, 255:red, 0; green, 0; blue, 0 }  ][line width=0.75]      (0, 0) circle [x radius= 3.35, y radius= 3.35]   ;
\draw [shift={(496.2,190.4)}, rotate = 180.22] [color={rgb, 255:red, 0; green, 0; blue, 0 }  ][fill={rgb, 255:red, 0; green, 0; blue, 0 }  ][line width=0.75]      (0, 0) circle [x radius= 3.35, y radius= 3.35]   ;
 
\draw  [dash pattern={on 4.5pt off 4.5pt}]  (330.2,190.4) -- (427,56.25) ;
 
\draw  [fill={rgb, 255:red, 0; green, 0; blue, 0 }  ,fill opacity=1 ] (270,98.29) .. controls (270,96.34) and (271.59,94.75) .. (273.54,94.75) .. controls (275.5,94.75) and (277.08,96.34) .. (277.08,98.29) .. controls (277.08,100.25) and (275.5,101.83) .. (273.54,101.83) .. controls (271.59,101.83) and (270,100.25) .. (270,98.29) -- cycle ;
 
\draw  [color={rgb, 255:red, 65; green, 117; blue, 5 }  ,draw opacity=1 ][line width=0.75]  (250.58,101.83) .. controls (250.58,89.15) and (260.86,78.87) .. (273.54,78.87) .. controls (286.22,78.87) and (296.5,89.15) .. (296.5,101.83) .. controls (296.5,114.51) and (286.22,124.79) .. (273.54,124.79) .. controls (260.86,124.79) and (250.58,114.51) .. (250.58,101.83) -- cycle ;
 
\draw  [color={rgb, 255:red, 65; green, 117; blue, 5 }  ,draw opacity=1 ][line width=0.75]  (177.08,190.79) .. controls (177.08,178.11) and (187.36,167.83) .. (200.04,167.83) .. controls (212.72,167.83) and (223,178.11) .. (223,190.79) .. controls (223,203.47) and (212.72,213.75) .. (200.04,213.75) .. controls (187.36,213.75) and (177.08,203.47) .. (177.08,190.79) -- cycle ;
 
\draw [color={rgb, 255:red, 65; green, 117; blue, 5 }  ,draw opacity=1 ][line width=0.75]    (494.98,169.98) -- (389.24,170.44) ;
 
\draw [color={rgb, 255:red, 65; green, 117; blue, 5 }  ,draw opacity=1 ][line width=0.75]    (494.96,210.57) -- (389.24,210.67) ;
 
\draw  [draw opacity=0][line width=0.75]  (389.24,210.67) .. controls (389.22,210.67) and (389.19,210.67) .. (389.17,210.67) .. controls (380.48,210.67) and (373.44,201.66) .. (373.44,190.56) .. controls (373.44,179.45) and (380.48,170.44) .. (389.17,170.44) .. controls (389.19,170.44) and (389.22,170.44) .. (389.24,170.44) -- (389.17,190.56) -- cycle ; \draw  [color={rgb, 255:red, 65; green, 117; blue, 5 }  ,draw opacity=1 ][line width=0.75]  (389.24,210.67) .. controls (389.22,210.67) and (389.19,210.67) .. (389.17,210.67) .. controls (380.48,210.67) and (373.44,201.66) .. (373.44,190.56) .. controls (373.44,179.45) and (380.48,170.44) .. (389.17,170.44) .. controls (389.19,170.44) and (389.22,170.44) .. (389.24,170.44) ;  
 
\draw  [draw opacity=0][line width=0.75]  (494.96,210.57) .. controls (484.91,292.53) and (415.01,356) .. (330.27,356) .. controls (238.62,356) and (164.33,281.77) .. (164.33,190.2) .. controls (164.33,98.63) and (238.62,24.4) .. (330.27,24.4) .. controls (415.06,24.4) and (484.99,87.95) .. (494.98,169.98) -- (330.27,190.2) -- cycle ; \draw  [color={rgb, 255:red, 65; green, 117; blue, 5 }  ,draw opacity=1 ][line width=0.75]  (494.96,210.57) .. controls (484.91,292.53) and (415.01,356) .. (330.27,356) .. controls (238.62,356) and (164.33,281.77) .. (164.33,190.2) .. controls (164.33,98.63) and (238.62,24.4) .. (330.27,24.4) .. controls (415.06,24.4) and (484.99,87.95) .. (494.98,169.98) ;  

\draw  [color={rgb, 255:red, 65; green, 117; blue, 5 }  ,draw opacity=1 ][fill={rgb, 255:red, 65; green, 117; blue, 5 }  ,fill opacity=1 ] (159.07,193.29) -- (164.32,183.22) -- (169.4,193.38) ;
\draw  [color={rgb, 255:red, 65; green, 117; blue, 5 }  ,draw opacity=1 ][fill={rgb, 255:red, 65; green, 117; blue, 5 }  ,fill opacity=1 ] (325.88,19.79) -- (336,24.93) -- (325.9,30.12) ;
\draw  [color={rgb, 255:red, 65; green, 117; blue, 5 }  ,draw opacity=1 ][fill={rgb, 255:red, 65; green, 117; blue, 5 }  ,fill opacity=1 ] (492.08,133.97) -- (489.96,145.13) -- (482.16,136.88) ;
\draw  [color={rgb, 255:red, 65; green, 117; blue, 5 }  ,draw opacity=1 ][fill={rgb, 255:red, 65; green, 117; blue, 5 }  ,fill opacity=1 ] (438.59,175.45) -- (428.49,170.27) -- (438.6,165.12) ;
\draw  [color={rgb, 255:red, 65; green, 117; blue, 5 }  ,draw opacity=1 ][fill={rgb, 255:red, 65; green, 117; blue, 5 }  ,fill opacity=1 ] (418.64,204.63) -- (428.6,210.09) -- (418.34,214.96) ;
\draw  [color={rgb, 255:red, 65; green, 117; blue, 5 }  ,draw opacity=1 ][fill={rgb, 255:red, 65; green, 117; blue, 5 }  ,fill opacity=1 ] (330.2,360.91) -- (320.22,355.48) -- (330.46,350.58) ;
\draw  [color={rgb, 255:red, 65; green, 117; blue, 5 }  ,draw opacity=1 ][fill={rgb, 255:red, 65; green, 117; blue, 5 }  ,fill opacity=1 ] (203.76,172.99) -- (193.24,168.71) -- (202.87,162.7) ;
\draw  [color={rgb, 255:red, 65; green, 117; blue, 5 }  ,draw opacity=1 ][fill={rgb, 255:red, 65; green, 117; blue, 5 }  ,fill opacity=1 ] (279.21,84.23) -- (269.22,78.83) -- (279.45,73.9) ;
 
\draw (258.2,194.2) node [anchor=north east] [inner sep=0.75pt]    {$a$};
 
\draw (332.2,193.8) node [anchor=north west][inner sep=0.75pt]    {$b$};
 
\draw (202.04,194.19) node [anchor=north west][inner sep=0.75pt]    {$0$};
 
\draw (375.5,73.4) node [anchor=north west][inner sep=0.75pt]    {$R$};
 
\draw (392.71,193.4) node [anchor=north west][inner sep=0.75pt]    {$1$};
 
\draw (498.2,193.8) node [anchor=north west][inner sep=0.75pt]    {$R$};
 
\draw (275.54,101.69) node [anchor=north west][inner sep=0.75pt]    {$z$};
 
\draw (292.67,27.4) node [anchor=north west][inner sep=0.75pt]    {$\gamma $};

\end{tikzpicture}
	\end{center}
	\subcaption{Deformed contour in Proposition~\ref{Prop_int evaluation 3}}
\end{subfigure}		
	\caption{Contour deformations in Propositions~\ref{Prop_int evaluation 1} and~\ref{Prop_int evaluation 3}. }   \label{fig:deformation}
\end{figure}
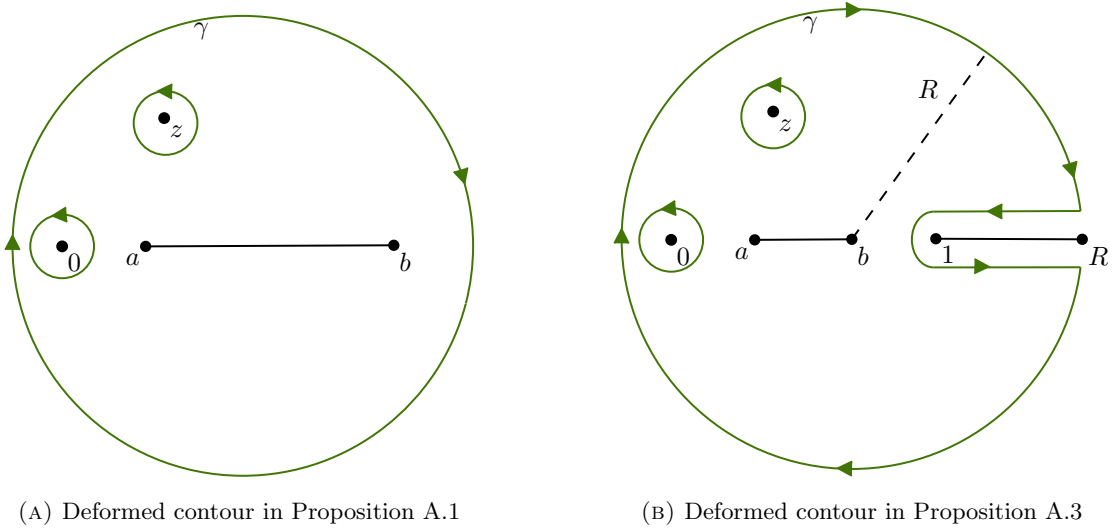

    \begin{prop} \label{Prop_int evaluation 2}
    Let $0<a<b \le d \le 1$. Then we have 
    \begin{equation}
       \int_b^d \frac{s^j}{R(s)}\, ds = \begin{cases}
            \frac{1}{\sqrt{ab}} \log\Big( \frac{\sqrt{b(d-a)} + \sqrt{a(d-b)}}{\sqrt{b(d-a)} - \sqrt{a(d-b)}}\Big) &\textup{if } j=-1, 
            \smallskip 
            \\
            2 \log \Big(\frac{\sqrt{d-a}+\sqrt{d-b}}{\sqrt{b-a}}\Big)  &\textup{if } j=0, 
            \smallskip 
            \\
    \sqrt{(d-a) (d-b)}+(a+b) \log \Big(\frac{\sqrt{d-a}+\sqrt{d-b}}{\sqrt{b-a}}\Big) &\textup{if } j=1. 
        \end{cases}
    \end{equation}
\end{prop}
\begin{proof}
We first observe that the argument of the square root is positive. We then proceed by introducing the change of variables $u=\sqrt{(s-b)/(s-a)}$. Then by straightforward computations, it follows that 
\begin{align}
      \int_b^d \frac{1}{sR(s)}\,ds &= 2\int_0^{\sqrt{\frac{d-b}{d-a}}} \frac{1}{b-u^2a}\,du = \tfrac{1}{\sqrt{ab}} \log\Big( \tfrac{\sqrt{b(d-a)} + \sqrt{a(d-b)}}{\sqrt{b(d-a)} - \sqrt{a(d-b)}}\Big), 
      \\
    \int_b^d \frac{1}{R(s)}ds &= 2\int_0^{\sqrt{\frac{d-b}{d-a}}} \frac{1}{1-u^2}\, du =  2\log \Big(\tfrac{\sqrt{d-a}+\sqrt{d-b}}{\sqrt{b-a}}\Big),
    \\
      \int_b^d \frac{s}{R(s)} \,ds &= 2\int_0^{\sqrt{\frac{d-b}{d-a}}} \frac{b-au^2}{(1-u^2)^2}\, du = \sqrt{(d-a) (d-b)}+(a+b) \log \Big(\tfrac{\sqrt{d-a}+\sqrt{d-b}}{\sqrt{b-a}}\Big). 
\end{align}
This completes the proof. 
\end{proof}

\begin{prop} \label{Prop_int evaluation 3}
   Let $0<a<b \le d \le 1$. Then for $z\not\in \R$, we have 
\begin{align}
\begin{split}
\frac{1}{2\pi i}\int_a^b \frac{\log(1-s)}{R_+(s)s(z-s)} \, ds  &= -\frac{\log(1-z)}{2zR(z)} + \frac{1}{ 2 z R(z)}\log\bigg( \frac{ \sqrt{ z-b} + \sqrt{z-a} }{ \sqrt{ z-b} -\sqrt{ z-a} }\bigg) 
\\
&\quad - \frac{1}{2 zR(z)}\log\bigg( \frac{ \sqrt{ (1 -a)(z-b) } + \sqrt{ (1-b)(z-a)} }{  \sqrt{(1-a)(z-b)} -\sqrt{ (1-b)(z-a)} }\bigg) 
\\
& \quad  - \frac{1}{z\sqrt{ab}} \log\bigg( \frac{\sqrt{b(1-a)} +\sqrt{a(1-b)}}{\sqrt{b} +\sqrt{a} }\bigg). 
\end{split}
\end{align} 
\end{prop}

\begin{proof}
    As in the proof of Proposition~\ref{Prop_int evaluation 1}, we have 
    \begin{equation}
        \frac{1}{2\pi i}\int_a^b \frac{\log(1-s)}{R_+(s)s(z-s)} \, ds  = \frac{1}{4\pi i}\int_\gamma \frac{\log(1-s)}{R(s)s(z-s)}\,ds\,,
    \end{equation}
    where $\gamma$ is the contour in Figure \ref{fig:deformation} (B). By taking the radius $R\to\infty$, we deduce that
    \begin{equation}
    \begin{split}
         \frac{1}{2\pi i}\int_a^b \frac{\log(1-s)}{R_+(s)s(z-s)} \, ds &  = \frac{1}{2} \bigg(\Res_{s=0} \frac{\log(1-s)}{R(s)s(z-s)} + \Res_{s=z} \frac{\log(1-s)}{R(s)s(z-s)} - \int_1^\infty\frac{1}{s(z-s)R(s)} \, ds \bigg)
         \\
         & = -\frac{\log(1-z)}{2zR(z)}+ \frac{1}{2z} \bigg(\int_1^{\infty} \frac{1}{R(s)(z-s)}\, ds -  \int_1^{\infty} \frac{1}{R(s)s}\,ds\bigg). 
    \end{split}
    \end{equation}
The remaining integral can now be evaluated using a standard Euler substitution, yielding the desired identity. 
\end{proof}
 

\bibliographystyle{abbrv}

\end{document}